\documentclass{article}
\usepackage[utf8]{inputenc}
\usepackage{lmodern}
\usepackage{mathscinet}     
\usepackage{a4wide,amsmath,amssymb}        
\usepackage{amsfonts}       
\usepackage{amsthm}         
\usepackage{bbding}         
\usepackage{bm}             
\usepackage{graphicx}       
\usepackage{fancyvrb}       
\usepackage[numbers]{natbib}         
\usepackage[nottoc]{tocbibind} 
\usepackage{enumitem}

\usepackage{authblk}
\usepackage[hyphens]{url}

\usepackage{dcolumn}        
\usepackage{booktabs}       
\usepackage[
	textsize=tiny,
	textwidth=3.2cm,
	colorinlistoftodos,
	backgroundcolor=teal!30!white,
	linecolor=teal!50!white
	]{todonotes}
\usepackage{hyperref}
\usepackage{mathtools}
\usepackage[framemethod=tikz]{mdframed}
\usepackage[a4paper, left=2cm, right=2cm, top=3cm, bottom=3cm]{geometry}
\usepackage{aligned-overset}

\hypersetup{colorlinks=true, linkcolor=teal, citecolor=violet}

\setlist[enumerate]{label={\rm(\roman*)}}

\newtheorem*{theorem*}{Theorem}

\newtheorem{theorem}{Theorem}[section]

\theoremstyle{definition}
\newtheorem{definition}[theorem]{Definition}

\theoremstyle{definition}
\newtheorem{example}[theorem]{Example}

\theoremstyle{plain}
\newtheorem{proposition}[theorem]{Proposition}

\theoremstyle{plain}
\newtheorem{lemma}[theorem]{Lemma}

\theoremstyle{plain}
\newtheorem{corollary}[theorem]{Corollary}

\theoremstyle{definition}
\newtheorem{remark}[theorem]{Remark}

\theoremstyle{plain}
\newtheorem{observation}[theorem]{Observation}

\numberwithin{equation}{section}

\def\diff{\;\mathrm{d}}
\def\reals{\mathbb{R}}
\def\nat{\mathbb{N}}

\def\esssup{\operatornamewithlimits{ess\,\sup}}
\def\diam{\operatorname{diam}}
\def\dist{\operatorname{{dist}}}

\def\vol{\operatorname{vol}}

\def\H{\mathcal{H}}
\def\Hn{\mathcal{H}^n}
\def\sp{\operatorname{sp}}

\def\mv{\operatorname{mv}}
\def\metric{d}

\newcommand{\tLip}{\textnormal{Lip}}
\newcommand{\tspan}{\operatorname{span}}

\newcommand{\infl}[2]{(|\cdot|_{#1}, |\cdot|_{#2})}
\newcommand{\factorlambda}{\lambda}
\newcommand{\numberdelta}{\Delta}

\title{Typical Lipschitz images of rectifiable metric spaces}
\date{\today}
\author[$\dagger$,$\star$]{David Bate}
\author[$\ddag$,$\star$]{Jakub Tak\'a\v{c}}
\affil[$\star$]{Zeeman building, University of Warwick, Coventry CV47AL}
\affil[$\dagger$]{david.bate@warwick.ac.uk}
\affil[$\ddag$]{jakub.takac@warwick.ac.uk}
\begin{document}
\maketitle
\renewcommand{\thefootnote}{\fnsymbol{footnote}} 
\renewcommand{\thefootnote}{\arabic{footnote}} 
\begin{abstract}
    This article studies typical 1-Lipschitz images of $n$-rectifiable metric spaces $E$ into $\reals^m$ for $m\geq n$.
    For example, if $E\subset \reals^k$, we show that the Jacobian of such a typical 1-Lipschitz map equals 1 $\Hn$-almost everywhere and, if $m>n$, preserves the Hausdorff measure of $E$.
    In general, we provide sufficient conditions, in terms of the tangent norms of $E$, for when a typical 1-Lipschitz map preserves the Hausdorff measure of $E$, up to some constant multiple.
    Almost optimal results for strongly $n$-rectifiable metric spaces are obtained.

    On the other hand, for any norm $|\cdot|$ on $\reals^m$, we show that, in the space of 1-Lipschitz functions from $([-1,1]^n,|\cdot|_\infty)$ to $(\reals^m,|\cdot|)$, the $\Hn$-measure of a typical image is not bounded below by any $\numberdelta>0$.
\end{abstract}
\section{Introduction}
Recall that an $\Hn$-measurable subset $E\subset X$ of a (complete) metric space is $n$-rectifiable if there exists countably many Lipschitz $f_i\colon A_i\subset \reals^n \to X$ such that
\begin{equation}
    \label{def-rect}
    \Hn\left(E\setminus \bigcup_{i\in\nat}f_i(A_i)\right)=0.
\end{equation}
Here and throughout this article, $\Hn$ denotes the $n$-dimensional Hausdorff measure on $X$.
Rectifiable subsets of a metric space were studied by Ambrosio \cite{Amb90}, Kirchheim \cite{Kirch} and Ambrosio--Kirchheim \cite{AKrec}.
In particular, \cite{AKrec} gives a description of a rectifiable set $E\subset X$ in terms of weak* tangent spaces, after isometrically embedding $E$ into a dual space of a separable space such as $\ell^\infty$.
Area and coarea formulas are also obtained in terms of the weak* tangent structure.

An $\Hn$-measurable set $S\subset X$ is called $n$-purely unrectifiable if it intersects every $n$-rectifiable set $E$ in an $\Hn$-null set.
The following characterisation of rectifiability in metric spaces has been obtained by the first named author in \cite{B} in terms of non-linear Lipschitz projections on $X$.
We denote by $\tLip_1(X,\reals^m)$ the set of all bounded $1$-Lipschitz functions $f\colon X \to \reals^m$ equipped with the supremum distance, a complete metric space.
Recall that a \emph{typical} element of $\tLip_1(X,\reals^m)$ satisfies some property, if the set of the elements satisfying said property is residual (that is, it contains a countable intersection of open dense sets).
Since residual sets are closed under countable intersections and are dense, they form a suitable notion of ``large'' sets.
\begin{theorem*}[\cite{B}]
    Let $X$ be a complete metric space.
    \begin{enumerate}
        \item If $S\subset X$ is purely $n$-purely unrectifiable, $\Hn(S)<\infty$ and
        \begin{equation*}
            \liminf_{r\to 0} \frac{\Hn(B(x,r)\cap S)}{r^n}>0 \quad \text{for $\Hn$-a.e.~$x\in S$,}
        \end{equation*}
        then a typical $f \in \tLip_1(X,\reals^m)$ satisfies $\Hn(f(S))=0$.
        \item If $E\subset X$ is $n$-rectifiable, $\Hn(E)>0$ and $m\geq n$, a typical $f \in \tLip_1(X,\reals^m)$ satisfies $\Hn(f(E))>0$.
    \end{enumerate}
\end{theorem*}
This should be viewed as an analogue of the Besicovitch--Federer projection theorem \cite[Theorem 18.1]{Mat}.

In this article we give a finer description of rectifiable subsets of a metric space.
Namely we answer the question, under what conditions is it possible to ensure that a typical $f\in \tLip_1(X,\reals^m)$ satisfies $\Hn(f(E))\geq\numberdelta$ for some $\numberdelta=\numberdelta(X,E)>0$.
The answer depends on the local geometry of $E$ and in particular its tangent spaces.
To illustrate our results, we first mention that when the ambient metric space is Euclidean, the strongest possible result holds.
\begin{theorem}\label{T:main-Euclidean}
    Suppose $E\subset \reals^k$ is $n$-rectifiable and $m\geq n$. Then the set of functions $f\in \tLip_1(\reals^k,\reals^m)$ satisfying
    \begin{equation*}
        \int_E J_E f \diff \Hn = \Hn(E)
    \end{equation*}
    is residual.
    Moreover, if $m>n$, the set of functions $f\in \tLip_1(\reals^k,\reals^m)$ satisfying
    \begin{equation*}
        \Hn(f(E)) = \Hn(E)
    \end{equation*}
    is residual.
\end{theorem}
Here $J_E f$ denotes the Jacobian of $f$ with respect to the rectifiable set $E$ (see Definition \ref{d:jacobian}).
In other words, for a typical $f\in \tLip_1(\reals^k,\reals^m)$, the (approximate) tangential Frech\'et differential $f'(x)$ is a linear isometry for $\Hn$-a.e.~$x\in E$. Via the area formula (see Theorem \ref{E:Prel-area-formula-metric-eucl}), the second statement  asserts that a typical $f$ does not lose measure by overlapping, provided $m>n$. This is false in the case $m=n$; for example, if $n=m=1$ then the measure of the image of any function converging to a constant function must converge to 0. The result of
Theorem \ref{T:main-Euclidean} is new even if one assumes $E$ to be the unit $n$-dimensional cube in $\reals^n$.
In particular, we see that a typical element of $\tLip_1(\reals^k,\reals^m)$ preserves the measure of a given $n$-rectifiable set, whilst destroying the measure of a given purely $n$-unrectifiable set.

On the other hand, this result may fail in the strongest possible way whenever the ambient space is not Euclidean.
In what follows, we work with general norms on $\reals^n$ and these shall be denoted by $|\cdot|_a$, $|\cdot|_b$ and similar, without the letters $a$, $b$ having any separate meaning. Using an abuse of notation we will also denote by $|\cdot|_2$ the Euclidean norm.
Recall that a point $u$ in a convex subset $K$ of a vector space $X$ is an extremal point if, for any $v\in X$, $u+v\in K$ and $u-v\in K$ imply $v=0$.
\begin{theorem}\label{T:main-general-failure}
    Suppose $n\in\nat$ and let $|\cdot|_a$ be any norm on $\reals^n$ such that the unit sphere of $|\cdot|_a$ contains a non-extremal point of the unit ball of $|\cdot|_a$. Let $X=([-1,1]^n, |\cdot|_{a})$ and, for $m\geq n$, let $|\cdot|_b$ be an arbitrary norm on $\reals^m$. The set
    \begin{equation*}
        \{f\in \tLip_1(X,(\reals^m, |\cdot|_b)): \Hn(f(X))> \numberdelta\}
    \end{equation*}
    is residual in $\tLip_1{(X,(\reals^m, |\cdot|_b))}$ if and only if $\numberdelta = 0$.
\end{theorem}
A particular example of $|\cdot|_a$ with a non-extremal point in the boundary is the maximum norm.

In general, for $m\geq n$, we provide sufficient conditions on a pair of normed spaces $(\reals^n,|\cdot|_a)$ and $(\reals^m,|\cdot|_b)$ for when it is possible to find a $\factorlambda>0$ such that a typical $f\in \tLip_1{((\reals^m, |\cdot|_a),(\reals^m, |\cdot|_b))}$ preserves the measure of any rectifiable $E\subset \reals^n$ up to a multiplicative factor of $\factorlambda$.
Indeed, in Definition \ref{D:inflation} we introduce the notion of a \emph{$\lambda$-inflating pair} of normed spaces, for $\lambda>0$. Intuitively, this holds whenever any linear map $A\colon (\reals^n, |\cdot|_a) \to (\reals^m, |\cdot|_b)$ of operator norm at most $1$ and full rank can be \emph{inflated} in a linear way so that the operator norm of the resulting inflated map is still at most $1$, but the volume (Jacobian) of the inflated map is at least $\lambda$ and, moreover, the ``inflation'' in question does not shrink in any direction. This can be viewed as a geometric condition relating the unit ball of the $|\cdot|_a \to |\cdot|_b$ operator norm in $\reals^{n\times m}$ to level sets of the Jacobian functional. 
In Theorem \ref{T:positive-result-cubes}, we show that a typical $f\in \tLip_1((\reals^n,|\cdot|_a),(\reals^m,|\cdot|_b))$ preserves the Hausdorff measure of a given rectifiable set by a factor of $\lambda$, whenever $(\reals^n,|\cdot|_a)$ and $(\reals^m,|\cdot|_b)$ are $\lambda$-inflating.

Theorem \ref{T:positive-result-cubes} can be extended to a rectifiable subset $E$ of a metric space as follows by considering the (equivalence classes of) \emph{approximate tangent norms} $T(E,\cdot)$ of $E$ (see Definition \ref{tangent-norm}).
For a fixed normed space $(\reals, |\cdot|_b)$, we write $\mathcal{N}^b_{\textnormal{infl}(\lambda)}(n)$ for the set of equivalence classes of norms on $\reals^n$ for which $(\reals^n,|\cdot|_a)$ and $(\reals^m,|\cdot|_b)$ are $(\vol(|\cdot|_a)\lambda)$-inflating (see Definition \ref{D:inflating-space} and formula \eqref{E:vol-of-norm}).
\begin{theorem}\label{T:intro-MSR-residuality}
    Suppose that $n,m\in \nat$, $n\leq m$, $X$ is a complete metric space and $E\subset X$ an $n$-rectifiable subset. Suppose $|\cdot|_b$ is a norm on $\reals^m$.
    Let $\lambda>0$ and assume that for $\Hn$-a.e.~$x\in E$, one has
    \begin{equation*}
        T(E,x)\in \mathcal{N}^b_{\textnormal{infl}(\lambda)}(n).
    \end{equation*}
    Then for each $\varepsilon>0$, there is a set $\widetilde{E}\subset E$ with $\Hn(E\setminus \widetilde{E})<\varepsilon$ and such that the set
    \begin{equation*}
        \{f\in\tLip_1(\widetilde E, (\reals^m, |\cdot|_b)): \int_{\widetilde E} J_{\widetilde E} f \diff \Hn \geq \lambda \Hn(\widetilde E)\}
    \end{equation*}
    is residual in $\tLip_1(\widetilde E, (\reals^m, |\cdot|_b))$.
    Moreover, if $m>n$, then the set
    \begin{equation*}
        \{f\in\tLip_1(\widetilde E, (\reals^m, |\cdot|_b)): \Hn(f(\widetilde E)) \geq \lambda \Hn(\widetilde E)\}
    \end{equation*}
    is residual in $\tLip_1(\widetilde E, (\reals^m, |\cdot|_b))$.
\end{theorem}

The simplest example of an $n$-rectifiable metric space, whose $\Hn$-a.e.~approximate tangent lies in $\mathcal{N}^b_{\textnormal{infl}(\lambda)}(n)$ is obtained simply via choosing any representative $|\cdot|_a$ of any equivalence class $[|\cdot|_a]\in \mathcal{N}^b_{\textnormal{infl}(\lambda)}(n)$ (provided the set is non-empty) and letting $E$ be an $\Hn$-measurable subset of $(\reals^n, |\cdot|_a)$. In fact, if $(\reals^n, |\cdot|_a)$, we are able to extend the relevant functions onto the whole space and may take $\widetilde E=E$, see Theorem \ref{T:positive-result-cubes}.
Since any pair of Euclidean norms are 1-inflating (see Example \ref{Ex:Euclidean-inflation}), Theorem \ref{T:main-Euclidean} follows from Theorem \ref{T:positive-result-cubes}.
In fact, this observation allows us to prove results in the spirit of Theorem \ref{T:main-Euclidean} for strongly $n$-rectifiable subsets of a metric space.

A set $E\subset X$ is \emph{strongly $n$-rectifiable} if, for any $\varepsilon>0$, we may find functions $f_i$ as in \eqref{def-rect} that are $(1+\varepsilon)$-biLipschitz (see Definition \ref{d:strong-rect}).
In Lemma \ref{l:strong-rect}, we will show that this is equivalent, for $n$-rectifiable sets $E$, to the condition that $T(E,x)$ contains the Euclidean norm for $\H^n$-a.e.\ $x\in E$. (See also Remark \ref{strong-rect-rmk} for the case that $E$ is not assumed to be $n$-rectifiable.)
An achievement of recent analysis on metric spaces is that any RCD metric space satisfies this condition \cite{MonNa,BruEli,Amb,KellMon, GiPas}, see Remark \ref{RCD}.
\begin{theorem}\label{T:intro-k-Euclidean-subspace}
    Suppose $n\in \nat$ and let $E$ be an $n$-rectifiable subspace of a complete metric space $X$. Denote by $|\cdot|_2$ the Euclidean norm on $\reals^n$ and
    let $k\in \nat$ , $k\leq n$ and
    \begin{equation*}
        E^*=\{x\in E: T(E,x)=[|\cdot|_2]\}.
    \end{equation*}
    Then, for any $k$-rectifiable subset $K$ of $E^*$
     we have the following. To each $\varepsilon>0$, there is a set $\widetilde K \subset K$ with $\H^k(K\setminus \widetilde K)<\varepsilon$ such that for every 
     $m\geq k$ a typical $f\in \tLip_1(\widetilde{K}, \reals^m)$ satisfies $J_{\widetilde K} f=1$ $\H^k$-a.e.~in $\widetilde K$.
    Moreover, for any $m>k$, the set
    \begin{equation*}
        \{f\in \tLip_1(\widetilde{K}, \reals^m): \H^k(f(\widetilde K)) = \H^k(\widetilde K)\}
    \end{equation*}
    is residual in $\tLip(\widetilde{K}, \reals^m)$. 
\end{theorem}
Note that, if $E$ is strongly $n$-rectifiable, $\Hn(E\setminus E^*)=0$ and so, in the case $k=n$, Theorem \ref{T:intro-k-Euclidean-subspace} holds for any positive measure subset of $E$.

Note that our most general results do not apply to the entire rectifiable set $E$. The principal difficulty of obtaining results on the whole of $E$ lies in the lack of a useful Lipschitz extension result. Recall that in a general metric space an $L$-Lipschitz function into $\reals^m$ may be extended to any larger domain, as a $(\sqrt{m}L)$-Lipschitz function, and this constant is sharp \cite[7.2 Theorem]{Mat}, \cite[2.10.44]{Federer}.
Consequently, when $m=1$, we do obtain residuality results for 1-rectifiable metric spaces, see Theorem \ref{1-rect}.

Within the study of these objects, we naturally arrive at the question whether a typical function $f\in \tLip_1(X,\reals^m)$ satisfies $f_\# \Hn_{|E}\ll \Hn$. Here $f_\# \Hn_{|E}$ denotes the pushforward of $\Hn_{|E}$, which is the restriction of $\Hn_X$ onto $E$. It is an interesting question, whether the set of these functions is residual in $\tLip_1(X,\reals^m)$.
It follows from Theorem \ref{T:main-Euclidean} and the area formula, that in the Euclidean case, this is true.
Unfortunately, in general we cannot answer this. We are able to get residuality in the strong space however. That is, we are able to show that a typical element of $\tLip^{\textnormal{str}}_1(X,\reals^m)$ satisfies $f_\#\Hn_E \ll \Hn$ (Corollary \ref{C:pushforward}), for
\begin{equation*}
    \tLip^{\textnormal{str}}_1(X,\reals^m)=(\tLip_1(X,\reals^m), \lVert \cdot \rVert_{\ell^\infty} + \tLip(\cdot)).    
\end{equation*}

The paper is structured as follows.
Section \ref{S:prel} contains preliminaries needed for the rest of the text.

In order to prove our residuality statements, we will show that, under various conditions, a set of the form
\begin{equation}\label{E:intro-thesets2}
    \{f\in \tLip_1(X,\reals^m): \Hn(f(E))> \factorlambda \Hn(E)\}
\end{equation}
is open and dense.

The openness statements, which hold in any metric space, are contained in Section \ref{S:openness}. Openness of the sets in \eqref{E:intro-thesets2} is equivalent to lower semi-continuity of the ``area'' functional
\begin{equation*}
    f\mapsto\Hn(f(E))
\end{equation*}
and the theorems are stated in this form.
We also prove lower semi-continuity results for the ``area formula'' functional
\begin{equation*}
    f\mapsto \int_E J_E f\diff\Hn.
\end{equation*}

The density statements are harder, and most require additional hypotheses. Section \ref{S:density-general} concentrates on the few that actually hold in any metric space. In particular, we show that sets of functions which essentially do not overlap on a fixed $n$-rectifiable set $E$ is dense even in the stronger space $\tLip_1^{\textnormal{str}}(X,\reals^m)$ (see Corollary \ref{C:no-overlaps}). Such functions satisfy the stronger area formula
\begin{equation*}
    \int_E J_E f \diff\Hn= \Hn(f(E)).
\end{equation*}

The rest of the paper deals with proving or disproving the density of the sets in \eqref{E:intro-thesets2}. Sections \ref{S:residuality-positive} and \ref{S:residuality-negative} give sufficient and necessary conditions, respectively, for the case $E\subset (\reals^n, |\cdot|_a)$ for some norm $|\cdot|_a$ on $\reals^n$.
In Section \ref{S:residuality-positive} we show that a $\lambda$-inflating pair of norms is sufficient to deduce the density of \eqref{E:intro-thesets2}, see Theorem \ref{T:positive-result-cubes}. In particular, this proves Theorem \ref{T:main-Euclidean}.
On the other hand, in Section \ref{S:residuality-negative} we give necessary conditions, in terms of extremal points of the unit ball, for \eqref{E:intro-thesets2} to be dense, see Theorem \ref{T:mv-negative}.
This section contains the proof of Theorem \ref{T:main-general-failure}. 

Finally, Section \ref{S:results-metric} provides residuality results for $n$-rectifiable metric spaces with suitable tangent spaces. It is there that we prove Theorems \ref{T:intro-MSR-residuality} and \ref{T:intro-k-Euclidean-subspace} through a combination of theory developed in preceding sections and a modified version of Kirchheim's decomposition result.

\subsection{Acknowledgements}
J.T. is supported by the Warwick Mathematics Institute Centre for Doctoral Training.
Both D.B. and J.T. are supported by the European Union's Horizon 2020 research and innovation programme (Grant agreement No. 948021).

\section{Preliminaries}\label{S:prel}

\subsection{Spaces of Lipschitz functions}
Let $(X,d)=(X,\metric_X)$ be a metric space.
For $x\in X$ and $r>0$, we shall denote by $B_X(x,r)=\{z\in X: \metric_X(x,z)\leq r\}$ the closed ball of radius $r$ in $X$.
Open balls will be denoted by $B^\circ_X(x,r)=\{z\in X: \metric(x,z)< r\}$.

Given a set $S\subset X$ and $r\geq 0$, we denote by $B_X(S,r)=\{z\in X: \dist(z,X)\leq r\}$ its $r$-neighbourhood. For $r=0$ this coincides with the topological closure and we shall use the notation $\overline{S}=B_X(S,0)$.

Let $(Y, \metric_Y)$ be another metric space.
Recall that a function $f\colon X \to Y$ is called $L$-Lipschitz for some $L\in [0,\infty)$ if
\begin{equation*}
    \metric_Y(f(x),f(y))\leq L \metric_X(x,y)\quad \text{for all $x,y\in X$.}
\end{equation*}
The least such $L$ is called the Lipschitz constant of $f$ and is denoted by $\tLip(f)$ or, if we need to be more specific, $\tLip_{X\to Y}(f)$.
A function $f\colon X \to Y$ is called Lipschitz if $\tLip(f)<\infty$. 
A function $f\colon X \to Y$ is called biLipschitz, if it is Lipschitz, injective and the inverse $f^{-1}\colon f(X) \to X$ is Lipschitz. In this case, if both $f$ and $f^{-1}$ are $L$-Lipschitz, we say that $f$ is $L$-biLipschitz.

The set of all \emph{bounded} Lipschitz functions from $X$ to $Y$ will be denoted by $\tLip(X,Y)$.
Given a fixed $L\in[0,\infty)$, we denote
\begin{equation*}
    \tLip_L(X,Y)=\{f\in \tLip(X,Y): \tLip(f)\leq L\}.
\end{equation*}
Given a set $\Gamma$ and a normed space $(Y, \lVert \cdot \rVert_Y)$, we denote $\lVert \varphi \rVert_{\infty}=\lVert \varphi \rVert_{\ell^\infty(\Gamma, Y)}= \sup_{\gamma\in \Gamma} \lVert\varphi(\gamma)\rVert_{Y}$ for any $\varphi\colon \Gamma \to Y$.

In the case $Y$ is a normed space as above, we consider the sets $\tLip(X,Y)$ and $\tLip_L(X,Y)$ to be equipped with metrics induced by the supremum norm $\lVert \cdot \rVert_{\ell^\infty(X,Y)}$.
With these metrics, the space $\tLip(X,Y)$ is a normed linear space, which needs not be complete. 
However, if $Y$ is complete, then the space $\tLip_L(X,Y)$ is complete for any $L\in [0,\infty)$.

Occasionally, it will be useful to consider this space equipped with a stronger norm denoted by
\begin{equation*}
    \tLip^{\textnormal{str}}(X,Y)= (\tLip(X,Y), \lVert \cdot \rVert_{\ell^{\infty}(X,Y)} + \tLip(\cdot)).
\end{equation*}
The symbol $\tLip_L^{\textnormal{str}}(X,Y)$ denotes the space $\tLip_L(X,Y)$ equipped with the metric inherited from $\tLip^{\textnormal{str}}(X,Y)$.
If $Y$ is complete, then both $\tLip^{\textnormal{str}}(X,Y)$ and $\tLip_L^{\textnormal{str}}(X,Y)$ are complete.

We shall routinely use the following two classical Lipschitz extension results. Firstly,
Kirszbraun's theorem (see e.g.~\cite[2.10.43]{Federer})  asserts that if $H_1$ and $H_2$ are Hilbert spaces, $S\subset H_1$ and $f\colon S \to H_2$ is $L$-Lipschitz, then $f$ admits an extension $f\colon H_1 \to H_2$ which is $L$-Lipschitz.
Secondly, we recall McShane's extension theorem \cite[7.2 Theorem]{Mat} asserting that given any metric space $X$, $S\subset X$ and $f\colon S \to \reals$ an $L$-Lipschitz function, there is an $L$-Lipschitz extension $\widetilde f\colon X \to \reals$. 
In particular, if $f$ is bounded on its domain and, say, $C= \sup_{x\in S} |f(x)|$, then the Lipschitz extension can be also assumed to be bounded by $C$. Indeed, the function $f_0\colon X \to \reals$ given by
\begin{equation*}
    f_0=\widetilde f \chi_{\{|\widetilde f|\leq C\}} + C\chi_{\{\widetilde f >C\}} - C \chi_{\{\widetilde f <-C\}}
\end{equation*}
is easily observed to be bounded by $C$, an extension of $f$ and $L$-Lipschitz.
Also note that by extending coordinate-wise, if $f\colon S \to \reals^m$ is $L$-Lipschitz, then there is an $\sqrt{m}L$-Lipschitz extension $f\colon X \to \reals^m$. If $f$ is bounded on its domain with, say, $\sup_{x\in S} |f(x)|_2= C$ then the extension can also be assumed to be bounded by $\sqrt{m}C$, i.e. $\sup_{x\in X} |f(x)|_2\leq \sqrt{m}C$.
Here $|\cdot|_2$ denotes the Euclidean norm on $\reals^m$.
Moreover if $f\colon S \to \reals^m_\infty$ is $L$-Lipschitz, then there is an $L$-Lipschitz extension $f\colon X \to \reals^m_\infty$. Here $\reals^m_\infty$ stands for $\reals^m$ equipped with the maximum norm. 
The result remains true after replacing $\reals^m_\infty$ with the Banach space
\begin{equation*}
    \ell^\infty(\Gamma)=(\{\varphi\colon \Gamma \to \reals: \sup_{\gamma\in \Gamma} |\varphi(\gamma)|<\infty\}, \lVert \cdot \rVert_{\ell^\infty(\Gamma, \reals)}),
\end{equation*}
for any set $\Gamma$.

If $X$ is a topological space and $H\subset X$, we call $H$ \emph{residual} if $H$ contains an intersection of countably many dense open sets.
Baire's theorem asserts that if $X$ is a complete metric space, then all of its residual subsets are dense in $X$.
This means that the family of residual subsets of $X$ is closed under countable intersections, supersets and contains only dense sets.
Therefore it is a suitable notion of ``large'' sets.
We say that a typical element of $X$ satisfies some property (P), if the set of its elements satisfying the property (P) is residual in $X$.

Recall that if $Y$ is a Banach space and $X$ is a metric space, the spaces $\tLip_L(X,Y)$, $\tLip_L^\textnormal{str}(X,Y)$ and $\tLip^{\textnormal{str}}(X,Y)$ are all complete, so residual subsets of these spaces are dense.

It should be also stated that if $H\subset \tLip_L(X,Y)$ is residual in $\tLip^{\textnormal{str}}_L(X,Y)$, then it must be dense in $\tLip_L(X,Y)$.
However, it might not need to be residual in $\tLip_L(X,Y)$.
In particular, the family of sets which are residual in $\tLip^{\textnormal{str}}_L(X,Y)$ forms a ``reasonable'' notion of large sets in $\tLip_L(X,Y)$ as the family is again closed under countable intersection, supersets and contains only dense sets.

As a nice illustrative example, it can be easily checked that $\{f\in \tLip_L(X,Y): \tLip(f)=L\}$ is residual in $\tLip_L(X,Y)$.

\subsection{Norms on finite dimensional spaces}
Given $n\in\nat$ a norm on $\reals^n$ will generally be denoted by symbols such as $|\cdot|_a$, $|\cdot|_b$ etc.
Note that the letters $a$, $b$ on their own do not have any meaning.
Using abuse of notation, we will also denote by $|\cdot|_2$ the Euclidean norm and by $|\cdot|_\infty$ 
the supremum norm. If the particular space (dimension) needs to be specified, we write
$|\cdot|_{\reals^n_a}$ instead.

The ball of a norm $|\cdot|_a$ of radius $r>0$ centred at $x$ is denoted by $B_a(x,r)$.
We write $B_a=B_a(0,1)$.
In particular, the unit Euclidean ball will be denoted by $B_2$ or $B_{\reals^n_2}$ if the dimension is relevant.

The symbol $\lVert \cdot \rVert$ is reserved for operator norms and norms on infinite dimensional spaces.
If $n,m\in \nat$, we denote by $\mathcal{L}(\reals^n,\reals^m)=\reals^{n\times m}$ the space of linear operators from $\reals^n$ to $\reals^m$.
If $|\cdot|_a$, $|\cdot|_b$ are norms on $\reals^n$ and $\reals^m$, we denote by $\lVert \cdot \rVert_{a\to b}$ or $\lVert \cdot \rVert_{\reals^n_a\to \reals^m_b}$ the operator norm induced by $|\cdot|_a$ and $|\cdot|_b$, i.e.
\begin{equation*}
    \lVert A \rVert_{a\to b} = \sup_{s\in B_a} |A(x)|_{b} \quad \text{for $A\in \reals^{n\times m}$.}
\end{equation*}
The symbol $B_{a\to b}$ or, if more clarity is needed, the symbol $B_{\reals^n_a \to\reals^m_b}$ will be used to denote the unit ball of $\lVert \cdot \rVert_{a\to b}$ in the space of linear operators $\mathcal{L}(\reals^n_a,\reals^m_b)$.

Recall the following sufficient condition for a function to be Lipschitz in convex subsets of normed spaces.

\begin{lemma}\label{L:prel-Lip}
    Assume $|\cdot|_a$ and $|\cdot|_b$ are norms on $\reals^n$ and $\reals^m$ respectively. If $f\in\tLip(\reals^n,\reals^m)$ and there is some $L\in[0,\infty)$ such that for $\Hn$-a.e.~$x$, the Frech\'et differential $f'(x)\in \mathcal{L}(\reals^n,\reals^m)$ exists and satisfies
    \begin{equation*}
        \lVert f'(x) \rVert_{a\to b}\leq L,
    \end{equation*}
    then $f\in \tLip_L(\reals^n_a,\reals^m_b)$.
\end{lemma}
\begin{proof}
    Let $x,y\in \reals^n$ and
    let $H\subset \reals^n$ be the unique hyperplane for which the Euclidean distance to $x$ equals the Euclidean distance to $y$. Let $\varepsilon>0$ and denote by $\Gamma_\varepsilon$ the set of Lipschitz curves $\gamma\colon [0,1] \to \reals^n$ such that $\Hn_a(\gamma([0,1]))\leq |x-y|_a+ \varepsilon$ and there exists some $z\in H$ such that $\gamma_{|[0,\frac{1}{2}]}$ is the affine curve connecting $x$ to $z$ and $\gamma_{|[\frac{1}{2},1]}$ is the affine curve connecting $z$ to $y$.

    Since $\int_0^1 |\gamma'(t)|_a\diff t = \Hn_a(\gamma([0,1]))$, any curve $\gamma \in \Gamma_\varepsilon$ satisfies
    \begin{equation*}
        \int_0^1 |\gamma'(t)|_a\diff t \leq |x-y|_a + \varepsilon.
    \end{equation*}
    Moreover, by Fubini's theorem, there exists some curve $\gamma \in \Gamma_\varepsilon$ such that $f'(\gamma(t))$ exists for $\H^1$-a.e.~$t\in[0,1]$. Using the fundamental theorem of calculus, Jensen's inequality, the estimate on the operator norm and the above inequality, we may calculate
    \begin{equation*}
        \begin{split}
            |f(x)-f(y)|_b
            &=\left\lvert\int_0^1 f'(\gamma(t))(\gamma'(t)) \diff t \right\rvert_b
            \leq \int_0^1 |f'(\gamma(t))(\gamma'(t))|_b \diff t\\
            &\leq \int_0^1 \lVert f'(\gamma(t))\rVert_{a\to b} |\gamma'(t)|_a\diff t
            \leq L\int_0^1 |\gamma'(t)|_a \diff t
            \leq L(|x-a|_a + \varepsilon).
        \end{split}
    \end{equation*}
    By sending $\varepsilon\to 0$, we obtain $|f(x)-f(y)|_b\leq L |x-y|_a$. Since $x$ and $y$ were arbitrary, $f$ is $L$-Lipschitz.
\end{proof}

If $|\cdot|_a$ is a norm on $\reals^n$ and $W\colon \reals^n \to \reals^n$ is an invertible linear map, we denote by $|\cdot|_{W(a)}$ the norm on $\reals^n$ given by
\begin{equation*}
    |x|_{W(a)}=|W^{-1}x|_a.
\end{equation*}
Observe that with this notation, one has $W(B_a)=B_{W(a)}$.

Finally, for a fixed $n\in\nat$ we define an equivalence relation $\sim$ on the set of all norms on $\reals^n$. This relation is given by $|\cdot|_{a_1}\sim |\cdot|_{a_2}$ if and only if there is an invertible linear map $A\colon \reals^n \to \reals^n$ such that $A(B_{a_1})=B_{a_2}$. Classes of equivalence will be denoted in the standard way by
\begin{equation*}
    [|\cdot|_a]=\{|\cdot|_{a'}: |\cdot|_{a'}\sim |\cdot|_a\}.    
\end{equation*}
Note that $|\cdot|_{a_1}\sim |\cdot|_{a_2}$ if and only if $\reals^n_{a_1}$ is isometrically isomorphic to $\reals^n_{a_2}$.



\begin{definition}
    \label{extremal-point}
Let $X$ be a vector space and $K\subset X$ a convex set. A point $u\in K$ is called an \textit{extremal} point of $K$ if for any $v\in X$
    \begin{equation*}
        u+v\in K \quad \text{and} \quad 
        u-v\in K, \quad \text{implies} \quad v=0.
    \end{equation*}
Suppose $X$ is equipped with a norm $|\cdot|$. Then the unit ball $B=\{x\in X: |x|\leq 1\}$ is a convex set. If $X$ is a finite-dimensional Banach space, then $B$ has an extremal point. If $x$ is an extremal point of $B$, then $x\in\partial B$.
\end{definition}

Suppose now that $X$ is a Banach space and $x\in\partial B$. It follows from the Hahn-Banach theorem, that there exists a \emph{supporting hyperplane} of $B$ containing $x$, i.e.~by definition, there is $x^*\in X^*$ such that $x^*(x)=1$ and $B\subset \{y\in X: x^*(y)\leq 1\}$.
We say that $x$ is \emph{strongly extremal} if there exists $x^*\in X^*$ such that for $y\in B$, $x^*(y)=1$ if and only if $y=x$, and $B\subset \{y\in X: x^*(y)\leq 1\}$.

If $X$ is finite-dimensional, then there always exists a strongly extremal point $x\in\partial B$. Indeed, as $\partial B$ is compact, find $x\in \partial B$ which maximizes the Euclidean distance from $0$. Then consider the tangent (affine) hyperplane $T$ to the Euclidean ball of the corresponding radius at $x$. This is a supporting hyperplane of $B$ and $y\in B\cap T$ if and only if $y=x$. Whence $x$ is strongly extremal.

It is easily verified that a strongly extremal point is also an extremal point.

Suppose $n\in \nat$ and $K\subset \reals^n$ is a convex set.
An affine hyperplane $T\subset \reals^n$ (i.e.~an affine subspace of dimension $\dim T = \dim X - 1$), is called an \emph{affine tangent} to $K$ at $x$ if $T$ is a supporting hyperplane of $K$ and $x\in T$ (this is just a change of name in the finite dimensional case).

Suppose $n,m\in \nat$, $n\leq m$. We define the functional $\vol \colon \reals^{n\times m} \to [0,\infty)$ by
\begin{equation*}
    \vol A = \sqrt{\det A^T A}.
\end{equation*}
Recall that, for $n$-dimensional Hausdorff measure (see Section \ref{s:rectifiability} below) one has 
\begin{equation}\label{E:vol-geom-meaning}
    \Hn(A(E))= \vol(A) \Hn(E) \quad \text{for any $E\subset \reals^n$, $\Hn$-measurable set.}
\end{equation}
We extend this definition to norms. For a norm $|\cdot|_a$ on $\reals^n$, we let
\begin{equation}\label{E:vol-of-norm}
    \vol |\cdot|_a = \frac{2^n}{\Hn(B_a)}.
\end{equation}
Recalling the fact that one always has $\Hn_{a}(B_a)=2^n$ together with Haar's theorem, it follows that 
\begin{equation*}
    \Hn_{a}(E)=\vol(|\cdot|_a)\Hn(E) \quad \text{for any $E\subset \reals^n$, $\Hn$-measurable set.}
\end{equation*}

\subsection{Rectifiable metric spaces}\label{s:rectifiability}
Suppose $X$ is a metric space. For each $s\in(0,\infty)$, $\delta>0$ and $E\subset X$ we define the quantity
\begin{equation*}
    \H^s_\delta(E)=\inf\{\sum_{i\in\nat} (\diam E_i)^s: E\subset \bigcup_{i\in\nat} E_i, \; \diam E_i\leq \delta\}.
\end{equation*}
The $s$-dimensional Hausdorff measure of $E$ is the quantity
\begin{equation*}
    \H^s(E)=\sup_{\delta>0} \H^s_\delta(E).
\end{equation*}
This constitutes an outer measure, which can be restricted to a Borel measure. 
If the underlying metric space needs to be specified, we use notation such as $\H^s_X$ and similar.

Note that if $X$ is complete and $E\subset X$ is $\Hn$-measurable with $\sigma$-finite $\Hn$-measure, then $\Hn_{|E}$ is inner regular by compact sets. Indeed, it is not necessary to assume separability of $X$ as $\overline{E}$ is separable and complete.

If $X=\reals^n$ and $|\cdot|_a$ is a norm on $\reals^n$, then for $k\leq n$, we shall use the conventions
\begin{equation*}
    \H^k=\H^k_{|\cdot|_2} \quad \text{and} \quad \H^k_a= \H^k_{|\cdot|_a}.
\end{equation*}
Note that in the above situation one always has \cite[Lemma 6 (i)]{Kirch}
\begin{equation*}
    \Hn_{a}(B_a)=2^n.
\end{equation*}

\begin{definition}
    Let $X$ be a complete metric space. An $\Hn$-measurable set $E\subset X$ is \emph{$n$-rectifiable} is there exists a countable number of sets $F_i\subset \reals^n$ and Lipschitz maps $f_i\colon F_i \to X$ such that
    \begin{equation*}
        \Hn(E\setminus \bigcup_{i} f_i(F_i))=0.
    \end{equation*}
     If one has $S\subset X$ such that $\Hn(S\cap E)$ for any $n$-rectifiable set $E\subset X$, then $S$ is called \emph{$n$-purely unrectifiable}.
\end{definition}

Given a metric space $X$ and its subset $E$, we say that $x\in X$ is an $\Hn$-density point of $E$ if
\begin{equation*}
    \lim_{r\to 0} \frac{1}{(2r)^n}\Hn(E\cap B(x,r))=1.
\end{equation*}
If $E$ if $n$-rectifiable, than its $\Hn$-a.e.~point is a density point of $E$ \cite[Theorem 9]{Kirch}.

Suppose $F\subset \reals^n$ and $f\colon F \to \reals^m$ for some $m,n\in \nat$. Given $u\in F$ an $\Hn$-density point of $F$, we say that a linear map $f'(u)\colon \reals^n \to \reals^m$ is the \emph{(approximate) Frech\'et differential of $f$ at $u$} if
\begin{equation*}
    \lim_{v\in F, v\to u} \frac{f(u)-f(v)-f'(u)(u-v)}{|u-v|_2} = 0.
\end{equation*}
It is a consequence of the Rademacher's differentiation theorem \cite[3.1.6]{Federer} combined with Kirszbraun's extension theorem \cite[2.10.43]{Federer} that at $\Hn$-a.e.~$u\in F$, the Frech\'et differential of $f$ exists uniquely.

Suppose $|\cdot|_a$ is a norm on $\reals^n$ and $|\cdot|_b$ is a norm on $\reals^m$. It follows from the definition that if $f\colon F_a \to\reals^m_b$ is $L$-Lipschitz for some $L\in [0,\infty)$, then
\begin{equation*}
    \lVert f'(u) \rVert_{a\to b}\leq L
\end{equation*}
for every $u\in F$ for which $f'(u)$ exists.

Now let $X$ be a metric space and $f\colon F \to X$. A semi-norm $s(\cdot)$ on $\reals^n$ is called a \emph{metric differential} of $f$ at $u\in \reals^n$ if one has
\begin{equation*}
    \lim_{v\in F, v\to u}\frac{d(f(v),f(u))-s(v-u)}{|v-u|_2}=0.
\end{equation*}
It is the classical result of Kirchheim \cite[Theorem 2]{Kirch} that if $f$ is Lipschitz, then for $\Hn$-a.e.~$u\in \reals^n$ the metric differential $s$ of $f$ at $u$ exists uniquely. In that case, we shall denote $|f'|(u)=s$.
Note that if $X$ is a Euclidean space of dimension $m$, we denote by $f'(u)$ the classical (approximate) Frech\'et differential of $f$ at $x\in F$.
In this case, according to the conventions above, for any $w \in \reals^n$ one has $|f'(u)(w)|_{\reals^m_2}=|f'|(u)(w)$, provided the left hand side is defined.

\begin{definition}[\cite{Kirch}, Definition 10]
    \label{tangent-norm}
    Let $n\in\nat$, let $E$ be a metric space and let $x\in E$. A norm $|\cdot|_a$ on $\reals^n$ is called an \emph{approximate tangent norm} \emph{to $E$ at $x$}, if there is a set $\widetilde{E}\subset E$ such that $x$ is an $\Hn$-density point of $\widetilde E$ and to each $r>0$, there is a set $F_r\subset \reals^n$ and a map $I_r\colon (F_r, |\cdot|_a) \to \widetilde{E}\cap B(x,r)$, which is a biLipschitz bijection satisfying
    \begin{equation*}
        \lim_{r\to 0} \max\{\tLip(I_r), \tLip(I^{-1}_r)\}=1.
    \end{equation*}  
\end{definition}

From \cite[Theorem 9]{Kirch}, it immediately follows that if $X$ is a complete metric space and $E$ is an $n$-rectifiable subset with $\Hn(E)<\infty$, then $E$ admits an approximate tangent norm at $\Hn$-a.e.~point of $E$.
Moreover, the approximate tangent norm is unique, up to linear isometry, at $\Hn$-a.e.~point of $E$. Finally, it also follows from the proof of \cite[Theorem 9]{Kirch} that if $F\subset \reals^n$, $f\colon F \to E$ and $u\in F$ is such that $|f'|(u)$ is a norm, then $|f'|(u)$ is a tangent norm to $E$ at $f(u)$ and it is unique up to linear isometry. We write
\begin{equation*}
    T(E,x)=[|\cdot|_a],
\end{equation*}
provided $|\cdot|_a$ is an approximate tangent norm to $E$ at $x\in E$, which is unique up to a linear isometry. Note that $T(E,x)$ is defined for $\Hn$-a.e.~$x\in E$.

\begin{remark}
    By \cite[Proposition 5.8]{AKrec}, tangent norms agree with the tangent spaces of Ambrosio and Kirchheim.

    A notion of a tangent metric measure space was recently introduced in \cite{Bate2} that is applicable to our setting.
    One easily verifies that, for $n$-rectifiable $E\subset X$ and $\H^n$-a.e.~$x\in E$, $|\cdot|_a$ is an approximate tangent norm to $E$ at $x$ if and only if $(\reals^n,|\cdot|_a,0)$ is a tangent metric measure space of $(E,x)$ in the sense of \cite{Bate2}.
\end{remark}

We will need to use \cite[Lemma 4]{Kirch}, which we restate here in an equivalent form for the reader's convenience.

\begin{lemma}[Kirchheim]\label{L:Kirch}
    Let $X$ be a metric space and $E\subset X$ an $n$-rectifiable set. Then, for any $\theta>0$, there exists a countable number of compact sets $F_i\subset \reals^n$, $E_i\subset E$, norms $|\cdot|_{a_i}$ on $\reals^n$ and $(1+\theta)$-biLipschitz maps $I_i\colon (F_i, |\cdot |_{a_i})\to E_i$ such that
    \begin{equation*}
        \Hn(E\setminus \bigcup_i E_i)=0.
    \end{equation*}
\end{lemma}

What follows is a refined version of the lemma above.

\begin{lemma}\label{L:MSR-open-decomposition}
    Let $X$ be a metric space, $F\subset \reals^n$ be $\H^n$-measurable with $\Hn(F)<\infty$ and $f\colon F \to X$ Lipschitz.
    For each $\varepsilon>0$ there is a compact set $K \subset f(F)$ with
    \begin{equation*}
        \Hn(f(F)\setminus K)< \varepsilon,
    \end{equation*}
    possessing the following property.
    For each $\theta>0$ there is a finite collection of sets $G_i\subset K$, $i=1,\dots, i_0$ such that the $G_i$ are pairwise disjoint open subsets of $K$ and
    \begin{enumerate}
        \item \label{Enum:open-decom-a} $K = \bigcup_{i=1}^{i_0}G_i$,
        \item \label{Enum:open-decom-b} to each $i$, there is some $x_i \in K$, $F_i\subset \reals^n$ and $|\cdot|\in T(K, x_i)$ such that $G_i$ is $(1+\theta)$-biLipschitz to $(F_i, |\cdot|)$.
    \end{enumerate}
\end{lemma}
\begin{proof}
    First fix $\theta>0$.
    For any $\varepsilon>0$, the existence of a $K\subset f(F)$ satisfying $\Hn(f(F)\setminus K)< \varepsilon$, \ref{Enum:open-decom-a} for compact $G_i$ and \ref{Enum:open-decom-b} for arbitrary norms $|\cdot|$ on $\reals^n$ is precisely given by \cite[Lemma 4]{Kirch} and using the inner regularity of $\H^n$ on $\reals^n$.
    It is evident from the proof of \cite[Lemma 4]{Kirch} that one may in fact take each $|\cdot| \in T(K,x_i)$.

    To obtain relatively open $G_i$, for each $j\in \nat$ apply the established statement for $\varepsilon_j= 2^{-j}\varepsilon$ and $\theta_j=1/j$ to obtain $i_j$ many pairwise disjoint compact sets $G_i^j$.
    Setting
    \[K = \bigcap_{j\in\nat} \bigcup_{i=0}^{i_j} G_i^j\]
    completes the proof, once we show that each $G_i^j\cap K$ is relatively open. To see that, note that for each $j\in \nat$, $\bigcup_{i=1}^{i_j} G^j_i$ is a union of disjoint compact sets $G^j_i$. Therefore each $G^j_i$ is also open in $\bigcup_{i=1}^{i_j} G^j_i$. Since $K\subset \bigcup_{i=1}^{i_j} G^j_i$, we have that $K\cap G_i^j$ is relatively open in $K$, because $G_i^j$ is relatively open in $\bigcup_{i=1}^{i_j} G^j_i$.
\end{proof}

Suppose $X$ is a complete metric space and $E\subset X$ is $n$-rectifiable. 
Let $x\in E$ and suppose there are sets $\widetilde{E}\subset E$, $\widetilde{F}\subset \reals^n$ and a biLipschitz map $I \colon \widetilde F \to \widetilde E$ such that $x$ is a density point of $\widetilde E$. Suppose the metric differential $|I'|(I^{-1}(x))$ exists and is a norm. Suppose $f\colon X \to \reals^m$ is a Lipschitz map such that $(f\circ I)'(I^{-1}(x))$ exists. Then we define the \emph{Jacobian of $f$ at $x$ with respect to $E$} as
\begin{equation}\label{d:jacobian}
    J_E f (x) = \frac{\vol (f\circ I)'(I^{-1}(x)))}{\vol (|I'|(I^{-1}(x)))}.
\end{equation}
This notion is independent of the particular choice of $I$ and $\widetilde E$ and is easily shown to agree with the Jacobian of Ambrosio and Kirchheim, defined via isometric embeddings into separable dual spaces, $\Hn$-a.e.~in $E$ (see \cite[(8.4)]{AKrec}).
In particular we obtain the following metric version of the area formula.
\begin{theorem}[\cite{AKrec}, Theorem 8.2]
    \label{E:Prel-area-formula-metric-eucl}
    For any metric space $X$, $n$-rectifiable $E\subset X$ and Lipschitz $f\colon E \to \reals^m$,
    \begin{equation}\label{E:area-formula}
        \int_E J_E f \diff \Hn=\int_{f(E)}\#f^{-1}(u)\diff \Hn(u).
    \end{equation}
    
\end{theorem}

We remark also that this notion agrees with the classical notion of a Jacobian of a function $\Hn$-a.e.
In particular, if $E\subset \reals^n$, one has $J_E f(x) = \vol f'(x)$ if right hand side is well defined.
Notice that the ``charts'' $I$ above can be obtained from the definition of a tangent norm to $E$ at a given point $x$, provided a tangent exists. If the tangent is also unique up to isomorphism, $J_E f(x)$ depends only on $f$ and $T(E,x)$.

With this definition of the metric Jacobian, we are able to easily obtain the following decomposition lemma.

\begin{lemma}\label{L:prel-decomp-nonzero-Jacobian}
    Suppose $X$ is a complete metric space and $E$ an $n$-rectifiable subset. If $m\geq n$ and $f\colon E \to \reals^m$ is Lipschitz, then there exists a countable number of pairwise disjoint compact sets
    \begin{equation*}
        E_i\subset \{x\in E: J_E(x)>0\},    
    \end{equation*}
    such that
    \begin{equation*}
        \Hn(\{x\in E: J_E(x)>0\}\setminus \bigcup_{i}E_i)=0
    \end{equation*}
    and $f$ is injective on each $E_i$. In particular, the formula
    \begin{equation}\label{E:prel-partial-area-formula}
        \sum_{i} \Hn(f(E_i))= \int_{E} J_E f \diff \Hn
    \end{equation}
    holds.
\end{lemma}
\begin{proof}
    The first part of the assertion follows from the combination of Lemma \ref{L:Kirch}, the definition of $J_E f$ and the Euclidean result \cite[3.2.2. Lemma]{Federer}. The ``in particular'' part then follows immediately from the area formula \eqref{E:area-formula}.
\end{proof}



We shall often work with the introduced notions on subsets of $E$, therefore we require the following statement.

\begin{lemma}\label{L:prel-ideal}
    Let $X$ be a complete metric space and $E$ an $n$-rectifiable subset with $\Hn(E)<\infty$. Suppose $\widetilde{E}\subset E$ is any $\Hn$-measurable set. Then 
    \begin{equation*}
        T(E,x)= T(\widetilde E, x) \quad \text{and}\quad J_E f(x)=J_{\widetilde E} f(x) \quad \text{for $\Hn$-a.e.~$x\in \widetilde{E}$ and every $f\in \tLip(E, \reals^m).$}
    \end{equation*}
\end{lemma}
\begin{proof}
    By \cite[Lemma 2.3]{Bate2}, for $\Hn$-a.e.~$x\in \widetilde E$, we have the density estimate
    \begin{equation*}
        \limsup_{r\to 0_+} \frac{\Hn(B(x,r)\cap (E\setminus \widetilde E))}{(2r)^n}=0.
    \end{equation*}
    From this, the statement about tangents follows easily.
    The statement about the Jacobians then follows from the statement about tangents (together with uniqueness of tangents $\Hn$-a.e.).
\end{proof}

Finally we turn our attention to strongly $n$-rectifiable metric spaces.
\begin{definition}\label{d:strong-rect}
    Let $X$ be a complete metric space. An $\Hn$-measurable set $E\subset X$ is \emph{strongly $n$-rectifiable} if, for any $\varepsilon>0$, there exist a countable number of sets $F_i\subset \reals^n$ and $(1+\varepsilon)$-biLipschitz maps $f_i\colon (F_i,|\cdot|_2)\to E$ such that
    \begin{equation}\label{str-rect}
        \Hn(E\setminus \bigcup_i f_i(F_i))=0.
    \end{equation}

    More generally, given a norm $|\cdot|_a$ on $\reals^n$, an $\Hn$-measurable set $E\subset X$ is \emph{strongly $|\cdot|_a$-rectifiable} if, for any $\varepsilon>0$, there exist a countable number of sets $F_i\subset \reals^n$ and $(1+\varepsilon)$-biLipschitz maps $f_i\colon (F_i,|\cdot|_a)\to X$ such that \eqref{str-rect} holds.
\end{definition}

\begin{lemma}\label{l:strong-rect}
    Let $X$ be a complete metric space, let $E\subset X$ be $n$-rectifiable and $|\cdot|_a$ a norm on $\reals^n$.
    Then $E$ is strongly $|\cdot|_a$-rectifiable if and only if, for $\Hn$-a.e.~$x\in E$, $T(E,x)=[|\cdot|_a]$.
\end{lemma}

\begin{proof}
    First suppose that $E$ is strongly $|\cdot|_a$-rectifiable.
    Fix $\varepsilon>0$ and, for $F\subset \reals^n$, let $f\colon (F,|\cdot|)\to E$ be $(1+\varepsilon)$-biLipschitz.
    Then for $\Hn$-a.e.\ $u\in F$, $T(E,f(u))=[|f'|(u)]$. However, since $f$ is $(1+\varepsilon)$-biLipschitz,
    \[\frac{|v|_a }{1+\varepsilon} \leq |f'|(u)(v) \leq (1+\varepsilon)|v|_a\]
    for all $v\in \reals^n$.
    As for any $\varepsilon>0$, and $\Hn$-a.e.~$x\in E$, we may find $f$, $F$ and $u$ as above with $x=f(u)$, we have $|\cdot|_a\in T(E,x)$ and so, by uniqueness, $T(E,x)=[|\cdot|_a]$ for $\Hn$-a.e.~$x\in E$.


    Conversely, if $E$ is $n$-rectifiable and $T(E,x)=[|\cdot|_a]$ for $\Hn$-a.e $x\in E$, then Lemma \ref{L:MSR-open-decomposition} implies that $E$ is strongly $|\cdot|_a$-rectifiable.
\end{proof}

\begin{remark}\label{strong-rect-rmk}
    A much stronger result is obtained from \cite{Bate2}.
    Indeed, suppose that $E\subset X$ satisfies $\Hn(E)<\infty$ and has positive lower $n$-dimensional Hausdorff density at $\Hn$-a.e.\ point.
    Then $E$ is strongly $|\cdot|_a$-rectifiable whenever, at $\H^n$-a.e.\ $x\in E$, $E$ has a unique ``weak Gromov--Hausdorff tangent'' that equals $(\reals^n,|\cdot|_a)$.
    In fact, it suffices that, for $\Hn$-a.e.~$x\in E$, all such tangents are $K_x$-biLipschitz images of $\reals^n$ and that at least one tangent at $x$ equals $(\reals^n,|\cdot|_a)$.

    Conversely, if $E$ is strongly $|\cdot|_a$-rectifiable, then the weak Gromov--Hausdorff tangents uniquely equal $(\reals^n,|\cdot|_a)$ $\Hn$-a.e., since they agree with $T(E,\cdot)$.
\end{remark}

\section{Lower semi-continuity of some area related functionals}\label{S:openness}

The goal of this section is to study the openness part of the residuality result, i.e.~to study lower semi-continuity of the ``area'' functional given by
\begin{equation*}
    f\mapsto\Hn(f(E))
\end{equation*}
and the ``area formula'' functional given by
\begin{equation*}
    f\mapsto \int_E J_E f\diff\Hn
\end{equation*}
in the relevant settings.
This follows, to some degree, the approach from \cite{B}. Mainly we use a modified version of \cite[Lemma 7.3]{B}.

We structure the section into two subsections. In the first, we study the local behaviour of both of the aforementioned functionals; in the second, we study global behaviour of the area functional and use lower semi-continuity thereof to obtain lower semi-continuity of the area formula functional.

For the entirety of this section, we let $m,n\in\nat$ with $n\leq m$ and denote by $B(x,r)$ the Euclidean ball in $\reals^n$ centred at $x$ of diameter $r$. We equip the spaces $\reals^n$, $\reals^m$ with the Euclidean norms.

\subsection{Local behaviour of area and area formula}\label{SS:local-area}

Firstly, we shall need a result for continuous functions based on the Brouwer's fixed point theorem.
The following lemma is a modified version of \cite[Lemma 7.3]{B} and its proof follows from \cite[Lemma 7.3]{B}.



\begin{lemma}\label{C:LA-Brower}
    Let $\varepsilon>0$ and $F\colon B(0,\varepsilon)\to B(0,\varepsilon)$ be a continuous function. Let $\eta\in (\frac{1}{2^n},1)$ and suppose
\begin{equation*}
    |F(y)-y|<\varepsilon(1-\sqrt[n]{\eta}) \quad \text{for all $y\in \partial B(0,\varepsilon)$.}
\end{equation*}
Then $F(B(0,\varepsilon))\supset B(0,\sqrt[n]{\eta}\varepsilon)$.
\end{lemma}

\begin{theorem}[Local lower semi-continuity of area]\label{T:LA-semicontinuity-area}
    Let $\Omega\subset \reals^n$ be an open set, $f\colon \Omega \to \reals^m$ a continuous function, $x\in\Omega$ and assume $f'(x)$ exists. Let $\eta\in(0,1)$. Then there are $\delta>0$ and $r_0>0$ such that for all $r\leq r_0$ and any continuous function $g\colon B(x,r)\to \reals^m$ with
    \begin{equation*}
        \lVert g - f \rVert_{\infty}\leq\delta r,
    \end{equation*}
     it holds that
    \begin{equation}\label{E:LA-goal1}
        \mathcal{H}^n(g(B(x,r)))\geq \eta \vol f'(x) \mathcal{H}^n(B(x,r)).
    \end{equation}
\end{theorem}
\begin{proof}
    As Hausdorff measures are invariant under translations, we may assume $x=0$ and $f(0)=0$. Let us denote $v=\vol f'(x)$. If $v=0$, the statement is trivial, so we can assume $v>0$ which is equivalent to stating that $A=f'(x)$ is of full rank.
    Thus, the map $A\colon \reals^n \to Y=A(\reals^n)$ is a linear invertible map.
    As $f(0)=0$, we have, by the definition of a Frech\'et derivative,
    \begin{equation}\label{E:LA-limit}
        \lim_{y\to 0} \frac{|f(y)-A(y)|}{|y|}=0.
    \end{equation}
    
    
    Let us denote by $P\colon \reals^m \to Y$ the orthogonal projection onto $Y$. Observe the following properties of $P$:
    \begin{enumerate}[label={(P\arabic*)}]
        \item \label{Enum:P1} $P\circ A = A$,
        \item \label{Enum:P2} $P$ is $1$-Lipschitz,
        \item \label{Enum:P3} for any $u\in \reals^m$ and $w\in Y$, it holds that $|Pu-w|\leq |u-w|$.
    \end{enumerate}
    Let $\lVert A^{-1}\rVert$ be the operator norm of $A^{-1}\colon Y \to \reals^n$. By virtue of \eqref{E:LA-limit}, there is an $r_0$ such that for all $r\leq r_0$ it holds that 
    \begin{equation*}
        |f(y)-A(y)|\leq \frac{1}{\lVert A^{-1}\rVert}\frac{1}{2}(1-\sqrt[n]{\eta})r \quad \text{for all $y\in B(0,r)$.}
    \end{equation*}
    This, by \ref{Enum:P3} and by applying $A^{-1}$ to the left hand side yields
    \begin{equation}\label{E:LA-f-with-P-and-A}
        |A^{-1}Pf(y)-y|\leq \frac{1}{2}(1-\sqrt[n]{\eta})r \quad \text{for all $y\in B(0,r)$.}
    \end{equation}
    Let $\delta=\frac{1}{\lVert A^{-1} \rVert}\frac{1}{2}(1-\sqrt[n]{\eta})$. By the property \ref{Enum:P2} above, if a function $g\colon B(0,r)\to \reals^m$ satisfies
    \begin{equation*}
        \lVert g - f \rVert_{\infty}\leq \delta r,    
    \end{equation*}
    then
    \begin{equation*}
        |Pg(y)-Pf(y)|\leq \delta r \quad \text{for all $y\in B(0,r)$.}
    \end{equation*}
    Therefore, for such $y$, we have
    \begin{equation*}
        |A^{-1} P g(y)-A^{-1} P f(y)|\leq
        \lVert A^{-1} \rVert \delta r=\frac{1}{2}(1-\sqrt[n]{\eta})r,
    \end{equation*}
    which in combination with \eqref{E:LA-f-with-P-and-A} gives
    \begin{equation}\label{E:LA-g-with-P-and-A}
        |A^{-1}Pg(y)-y|\leq (1-\sqrt[n]{\eta})r \quad \text{for all $y\in B(0,r)$.}
    \end{equation}

    To use Lemma \ref{C:LA-Brower} we require $A^{-1}Pg(B(0,r))$ to be a subset of $B(0,r)$.
    To this end, let 
    \begin{equation*}
        \sigma\colon B(0, (1+ \sqrt[n]{\eta})r)\to B(0,r)
    \end{equation*}
    be the radial projection onto $B(0,r)$. More precisely, for an element $y\in B(0, (1+ \sqrt[n]{\eta})r)$, we let $\sigma(y)$ be the unique $u\in B(0,r)$ minimizing the distance $|u-y|$. We observe that  $\sigma$ has properties analogous to those of $P$, namely
    \begin{enumerate}[label={(S\arabic*)}]
        \item \label{Enum:S1} $\sigma$ is an identity on $B(0,r)$,
        \item \label{Enum:S2} $\sigma$ is $1$-Lipschitz,
        \item \label{Enum:S3} for any $y\in B(0,r)$ and $z\in B(0, (1+ \sqrt[n]{\eta})r)$, it holds that $|\sigma (z) - y|\leq |z-y|$.
    \end{enumerate}
    From the estimate \eqref{E:LA-g-with-P-and-A}, we infer that $A^{-1}Pg(B(0,r))\subset B(0, (1+ \sqrt[n]{\eta})r)$ and so we may define
    $G=\sigma\circ A^{-1}\circ P \circ g$. By the property \ref{Enum:S3} and the inequality \eqref{E:LA-g-with-P-and-A} we obtain
    \begin{equation*}
        |G(y)-y|\leq (1-\sqrt[n]{\eta})r \quad \text{for all $y\in B(0,r)$.}
    \end{equation*}
    Finally, if $g$ is continuous then so is $G$ and hence Lemma \ref{C:LA-Brower} gives
    \begin{equation*}
        G(B(0,r))\supset B(0,\sqrt[n]{\eta}r)
    \end{equation*}
    Applying $\mathcal{H}^n$ to both sides, we obtain
    \begin{equation*}
        \mathcal{H}^n(\sigma A^{-1} P g(B(0,r)))\geq \eta \mathcal{H}^n(B(0,r)).
    \end{equation*}
    From \ref{Enum:S2}, the last equation implies
    \begin{equation*}
        \mathcal{H}^n(A^{-1} P g(B(0,r)))\geq \eta \mathcal{H}^n(B(0,r)).
    \end{equation*}
    From \eqref{E:vol-geom-meaning} and from the fact that $\vol A^{-1} = \tfrac{1}{v}$, the last equation implies
    \begin{equation*}
        \mathcal{H}^n(P g(B(0,r)))\geq \eta v \mathcal{H}^n(B(0,r)).
    \end{equation*}
    Finally, by \ref{Enum:P2}, we obtain \eqref{E:LA-goal1}.
\end{proof}

\begin{corollary}\label{C:LA-lsc-area-areaformula}
    Let $\Omega\subset \reals^n$ be an open set and let $f\colon \overline{\Omega}\to \reals^m$ be a Lipschitz map. Assume $x\in\Omega$ is a density point of $\vol f'$ with $\vol f'(x)>0$ and let $\eta\in(0,1)$. Then there is $r_0>0$ and $\delta>0$ such that for every $r\leq r_0$ if $g\in C(B(x,r),\reals^m)$ satisfies
    \begin{equation*}
        \lVert g-f \rVert_{\infty}\leq\delta r,
    \end{equation*}
         then
        \begin{equation*}
            \mathcal{H}^n(g(B(x,r)))\geq \eta \mathcal{H}^n(f(B(x,r))).
        \end{equation*}
\end{corollary}
\begin{proof}
    From Theorem \ref{T:LA-semicontinuity-area} we can find $r_0>0$ and $\delta>0$ such that for $g\in C(B(x,r),\reals^m)$ with $\lVert g-f \rVert_{\infty}\leq\delta r$ it holds that
    \begin{equation}\label{E:LA-lsc-1}
        \mathcal{H}^n(g(B(x,r)))\geq \sqrt{\eta} \vol f'(x)\mathcal{H}^n(B(x,r)).
    \end{equation}
    From the fact that $x$ is a density point of $\vol f'$ it follows that we may possibly decrease $r_0>0$ so that for $r<r_0$ we also have
    \begin{equation}\label{E:LA-lsc-2}
        \vol f'(x)\mathcal{H}^n(B(x,r))\geq \sqrt{\eta} \int_{B(x,r)} \vol f' \diff \mathcal{H}^n.
    \end{equation}
    By the area formula, we obtain
    \begin{equation}\label{E:LA-lsc-3}
        \int_{B(x,r)} \vol f' \diff \mathcal{H}^n\geq \mathcal{H}^n(f(B(x,r))).
    \end{equation}
    Combining \eqref{E:LA-lsc-1}, \eqref{E:LA-lsc-2} and \eqref{E:LA-lsc-3} yields the result.
\end{proof}

\subsection{Global lower semi-continuity of area in rectifiable metric spaces}

We shall continue our previous conventions and assume $n\leq m$ are natural numbers and $B(x,r)$ denotes the Euclidean ball in $\reals^n$. We consider the spaces $\reals^n$ and $\reals^m$ to be equipped with the Euclidean norms, unless stated otherwise.

\begin{lemma}\label{L:lsc-area-general-Euclidean-bilipschitz}
    Let $E\subset \reals^n$ be a compact set and let $f\colon E \to \reals^m$ be a Lipschitz injection. Let $L\in[0,\infty)$. Then, for every $\eta\in(0,1)$ there exists $\delta>0$ such that if $g\in \tLip_L(E,\reals^m)$ satisfies $\lVert g - f \rVert_{\infty}\leq \delta$, then
    \begin{equation*}
        \Hn(g(E))\geq \eta \Hn(f(E)).
    \end{equation*}
\end{lemma}
\begin{proof}
    If $\Hn(f(E))=0$, the statement obviously holds, so we may assume $\Hn(f(E))>0$.
    Firstly, let $L_0$ denote the Lipschitz constant of $f$. Let $C_0=\sqrt{m}(L+2L_0)$. Find $\varepsilon>0$ such that
    \begin{equation}\label{E:lsc-comp-epsilon-choice}
        \sqrt{\eta} \Hn(f(E)) -\sqrt{\eta} \varepsilon L_0 - \varepsilon\geq \eta \Hn(f(E)).
    \end{equation}
    Using the McShane extension theorem, we find an extension of $f$, denoted again by $f$ such that $f$ is $\sqrt{m}L_0$-Lipschitz. 
    
    Let $S\subset E$ be the set of density points of $E$ and $\vol f'$. Then by the Lebesgue differentiation theorem, we have
    \begin{equation}\label{E:lsc-comp-LDT}
        \Hn(E\setminus S)=0.
    \end{equation}
    Let $x\in S$. Then by Corollary \ref{C:LA-lsc-area-areaformula}, there is some $1\geq r_x>0$ and $\delta_x>0$ such that for all $r\leq r_x$ and $h\in C(B(x,r),\reals^m)$ with $\lVert h-f\rVert_{\ell^\infty(B(x,r))}\leq \delta_x r$, it holds that
    \begin{equation}\label{E:lsc-comp-area-estimate-local}
        \Hn(h(B(x,r)))\geq \sqrt{\eta} \Hn(f(B(x,r))).
    \end{equation}
    Observe that since $B(S,1)$ is bounded, there exists some $\Delta>0$ such that for any countable sequence of disjoint balls $B_i$ with radii $r_i$ satisfying $B_i\subset B(S,1)$ for each $i$, we have
    \begin{equation}\label{E:lsc-comp-sum-estimate-delta}
        \sum_i r_i^n \leq \Delta.
    \end{equation}
    As $x$ is a density point of $E$, we have
    \begin{equation*}
        \lim_{r\to 0_+} \frac{\Hn(B(x,r)\setminus E)}{r^n}=0,
    \end{equation*}
    and so we may possibly reduce our $r_x>0$ so that we also get for $r\leq r_x$
    \begin{equation}\label{E:lsc-comp-density-estimate}
        \Hn(h(B(x,r)\setminus E))\leq \frac{\varepsilon}{\Delta}r^n\quad \text{for any $h\in \tLip_{C_0}(\reals^n,\reals^m)$.}
    \end{equation}
    Now the family of balls $\mathfrak{B}=\{B(x,r): x\in S, r\leq r_x\}$ forms a Vitali cover of $S$. Hence, using the Vitali covering theorem, and recalling \eqref{E:lsc-comp-LDT} there is a countable disjoint family of balls $B_i$ in $\mathfrak{B}$ such that
    \begin{equation*}
        \Hn(E\setminus \bigcup_{i} B_i)=0.
    \end{equation*}
    Using continuity of $\Hn$, there is some $i_0\in \nat$ such that even
    \begin{equation}\label{E:lsc-comp-measure-filling}
        \Hn(E\setminus \bigcup_{i=1}^{i_0} B_i)\leq \varepsilon.
    \end{equation}
    On denoting $B_i=B(x_i,r_i)$ and letting $\delta_i=\delta_{x_i}$, $i\in\{1,\dots, i_0\}$ and using \eqref{E:lsc-comp-area-estimate-local} and \eqref{E:lsc-comp-density-estimate}, we obtain for each $i\in\{1,\dots,i_0\}$
    \begin{equation}\label{E:lsc-comp-ball-measure-estimate}
        \Hn(h(B_i)) \geq \sqrt{\eta} \Hn(f(B_i)) \quad \text{for $h\in C(B_i, \reals^m)$ with $\lVert h-f \rVert_{\ell^\infty(B_i)}\leq \delta_i r_i$}
    \end{equation}
    and
    \begin{equation}\label{E:lsc-comp-density-estimate-2}
        \Hn(h(B_i\setminus E)) \leq \frac{\varepsilon}{\Delta} r^n_i \quad \text{for $h\in \tLip_{C_0}(\reals^n,\reals^m)$.}
    \end{equation}
    By \eqref{E:lsc-comp-density-estimate-2}, we now have
    \begin{equation}\label{E:lsc-comp-density-estimate-3}
        \begin{split}
            \Hn(h(B_i)\cap h(E))&=\Hn (h(B_i\setminus(B_i\setminus E)))
            \geq \Hn(h(B_i)\setminus h((B_i\setminus E)))\\
            &=\Hn(h(B_i))-\Hn(h(B_i\setminus E))
            \geq \Hn(h(B_i))- \frac{\varepsilon}{\Delta}r_i^n,
        \end{split}
    \end{equation}
    for $i\in \{1,\dots, i_0\}$ provided $h\in \tLip_{C_0}(\reals^n,\reals^m)$.
    
    Since $f(B_i \cap E)$ are pairwise disjoint compact sets (as $f$ is a continuous injection on the compact set $E$), there is some $\rho>0$ such that
    \begin{equation*}
        \dist(f(B_i\cap E), f(B_j\cap E))\geq \rho\quad \text{for all $i,j\in\{1,\dots i_0\}$ with $i\not=j$.}
    \end{equation*}
    observe that if $h\in C(\reals^n,\reals^m)$ satisfies $\lVert h - f \rVert_{\ell^\infty(\reals^n)}\leq \frac{\rho}{4}$, then
    \begin{equation*}
        \dist(h(B_i\cap E), h(B_j\cap E))\geq \frac{\rho}{2}\quad \text{for all $i,j\in\{1,\dots i_0\}$ with $i\not=j$.}
    \end{equation*}
    In particular, $h(B_i\cap E)$ are pairwise disjoint. Let 
    \begin{equation*}
        \delta=\min\{\tfrac{\rho}{4},\min_{i=1\dots i_0}\delta_i r_i\}.
    \end{equation*}

    Now suppose that $g\in \tLip_{C_0}(\reals^n,\reals^m)$ satisfies $\lVert g- f \rVert_{\ell^\infty(\reals^n)}\leq \delta$, then we have
    \begin{align*}
        \Hn(g(E))&\geq \Hn(g(E\cap \bigcup_{i=1}^{i_0} B_i))
        =\sum_{i=1}^{i_0} \Hn(g(B_i\cap E))
         \\
        \overset{\eqref{E:lsc-comp-density-estimate-3}}&{\geq}\sum_{i=1}^{i_0} \Hn(g(B_i)) - \frac{\varepsilon}{\Delta}\sum_{i=1}^{i_0} r_i^n
        \overset{\eqref{E:lsc-comp-ball-measure-estimate}, \eqref{E:lsc-comp-sum-estimate-delta}}{\geq} \sqrt{\eta}\sum_{i=1}^{i_0} \Hn(f(B_i)) - \varepsilon \\
        &=\sqrt{\eta} \Hn(f(\bigcup_{i=1}^{i_0}B_i))-\varepsilon
        \geq \sqrt{\eta} \Hn(f(E)) - \sqrt{\eta} \Hn(f(E\setminus\bigcup_{i=1}^{i_0}B_i))-\varepsilon\\
        \overset{\eqref{E:lsc-comp-measure-filling}}&{\geq} \sqrt{\eta} \Hn(f(E))-\sqrt{\eta} \varepsilon L_0-\varepsilon\overset{\eqref{E:lsc-comp-epsilon-choice}}{\geq} \eta \Hn(f(E)),
    \end{align*}
    where in the unlabelled equalities, we use disjoint additivity of measure and in the last unlabelled inequality, we use the inclusion
    \begin{equation*}
       f(E)\subset f(\bigcup_{i=1}^{i_0}B_i) \cup f(E\setminus \bigcup_{i=1}^{i_0}B_i).
    \end{equation*}
    
    Now let $g\in \tLip_L(E,\reals^m)$ satisfy
    \begin{equation*}
        \lVert g-f \rVert_{\ell^\infty(E)}\leq \frac{1}{\sqrt{m}} \delta.
    \end{equation*}
    Take $d=g-f_{|E}$ and, using McShane's extension theorem, find an extension thereof onto the entire $\reals^n$ such that 
    $d\in \tLip_{\sqrt{m}(L+L_0)}(\reals^n, \reals^m)$ and 
    $\lVert d \rVert_{\ell^\infty(\reals^n,\reals^m)}\leq \delta$. Then $\widetilde g =d+f$ is an extension of $g$ such that $\widetilde g \in \tLip_{C_0}(\reals^n,\reals^m)$ and 
    \begin{equation*}
        \lVert \widetilde g- f \rVert_{\ell^\infty(\reals^n)}\leq \delta.
    \end{equation*}
    Whence $\Hn(g(E))= \Hn(\widetilde g (E))\geq \eta \Hn(f(E))$ by the above calculation.
\end{proof}

Note that we only require $g\in \tLip_L(\reals^n,\reals^m)$ instead of simply $g\in\tLip(\reals^n,\reals^m)$, or even $g\in C(\reals^n,\reals^m)$, to obtain the estimate
\eqref{E:lsc-comp-density-estimate}. Therefore, in some cases, the assumption is superfluous. For example, if $E=\reals^n$ or, more generally, if $\Omega\subset \reals^n$ is open and $E=\overline{\Omega}$ or $E=\Omega$.

\begin{remark}
    Let $\Omega\subset \reals^n$ be open and bounded and let $f\colon \overline{\Omega}\to \reals^m$ be a Lipschitz injection. Then, for every $\eta\in(0,1)$ there exists $\delta>0$ such that if $g\in C(\Omega,\reals^m)$ satisfies $\lVert g - f \rVert_{\ell^\infty(\overline{\Omega})}\leq \delta$, then
    \begin{equation*}
        \Hn(g(\Omega))\geq \eta \Hn(f(\overline{\Omega})).
    \end{equation*}
\end{remark}

We follow up with a version of Lemma \ref{L:lsc-area-general-Euclidean-bilipschitz} for metric spaces which are biLipschitz images of Euclidean sets.

\begin{lemma}\label{L:lsc-area-metric-space-special}
    Let $K$ be a compact metric space for which there is a set $F\subset \reals^n$ and a biLipschitz bijection $I\colon F \to K$. Assume $f\colon K \to \reals^m$ is biLipschitz and let $L\in[0,\infty)$. Then for every $\eta\in(0,1)$, there exists $\delta>0$ such that if $g\in\tLip_L(K,\reals^m)$ satisfies $\lVert g-f \rVert_{\infty}\leq \delta$, then
    \begin{equation*}
        \Hn(g(K))\geq \eta \Hn(f(K)).
    \end{equation*}
\end{lemma}
\begin{proof}
    Firstly, we observe, that for any $g\in \tLip_L(K,\reals^m)$, we have $g\circ I \in \tLip_{CL}(F,\reals^m)$, where $C$ is the Lipschitz constant of $I$. Let $\varphi=f\circ I\colon F \to \reals^m$. As $f$ and $I$ are biLipschitz, so is $\varphi$, hence Lemma \ref{L:lsc-area-general-Euclidean-bilipschitz} gives a $\tilde{\delta}>0$ such that if $\psi\in\tLip_{CL}(F,\reals^m)$ satisfies $\lVert \varphi - \psi\rVert_{\ell^\infty(F)}\leq \tilde{\delta}$, then
    \begin{equation}\label{E:lsc-area-metric-psi-to-phi}
        \Hn(\psi(F))\geq \eta \Hn(\varphi(F)).
    \end{equation}
    Let $\delta= \frac{\tilde{\delta}}{C}$. Assume $g\in\tLip_{L}(K,\reals^m)$ satisfies $\lVert f - g\rVert_{\ell^\infty(K)}\leq \delta$. Denote $\psi=g\circ I$. Then $\psi\in\tLip_{CL}(F,\reals^m)$ and $\lVert \varphi - \psi\rVert_{\ell^\infty(F)}\leq \tilde{\delta}$, so \eqref{E:lsc-area-metric-psi-to-phi} holds. From this, we have
    \begin{equation*}
        \Hn(g(K))=\Hn(\psi \circ I^{-1}(K))
        =\Hn(\psi(F))\geq \eta \Hn(\varphi(F))
        =\eta \Hn(f\circ I^{-1}(F))=\eta \Hn(f(K)).
    \end{equation*}
\end{proof}

We shall fix a complete metric space $X$ and an $n$-rectifiable subset $E$. Define a functional 
\begin{equation}\label{E:area-def}
    \mathcal{A}_E(g)=\Hn(g(E)),
\end{equation}
for any $g\colon X \to \reals^m$.
For $L\in[0,\infty)$, denote $\Lambda_L=(\tLip_L(X,\reals^m),\lVert \cdot \rVert_{\infty})$. 

\begin{theorem}\label{T:lsc-area-general}
    For any $L\in[0,\infty)$, the functional $\mathcal{A}_E$ is lower semi-continuous on $\Lambda_L$.
\end{theorem}

\begin{proof}
    Firstly, we shall assume $\Hn(E)<\infty$.
    We show that for $C>0$, the set
    \begin{equation}\label{E:lsc-area-set}
        \{f\in \Lambda_L: \Hn(f(E))>C\}
    \end{equation}
    is open. To that end, let $f$ be from the set.
    
    We split the proof into two parts. Firstly, we carefully decompose the space $f(E)$.
    Using Lemma \ref{L:Kirch}, we find a countable number of Borel sets $F_i\subset \reals^n$ and $E_i\subset E$ such that $\Hn(E\setminus \bigcup_{i}E_i)=0$ and each $F_i$ is biLipschitz to $E_i$.
    Let $M_1=f(E_1)$ and for $i>1$, let
    \begin{equation*}
        M_i=f(E_i)\setminus (f(E_1)\cup\dots\cup f(E_{i-1})).
    \end{equation*}
    Then we have
    \begin{equation*}
        \Hn (f(E)\setminus \bigcup_i M_i)=\Hn( E\setminus \bigcup_{i}E_i)=0.
    \end{equation*}
    Let now $\varepsilon>0$ and $\eta\in (0,1)$ be such that
    \begin{equation}\label{E:lsc-area-choice}
        \eta \mathcal{H}^n(f(E))-\eta\varepsilon>C.
    \end{equation}
    By continuity of measure, we find $i_0\in\nat$ such that
    \begin{equation}\label{E:lsc-area-estimate-in-image}
        \Hn(f(E)\setminus \bigcup_{i=1}^{i_0} M_i)< \varepsilon.
    \end{equation}
    
    
    Let $\widetilde{K}_i= f^{-1}(M_i)$. As the $M_i$ are disjoint, so are the $\widetilde{K}_i$. Since $f$ is Lipschitz and $\Hn$ is inner regular on each $\widetilde{K}_i$, the estimate \eqref{E:lsc-area-estimate-in-image} allows us to find, for each $i\in\{1,\dots, i_0\}$ a compact subset $K_i\subset \widetilde{K}_i$, such that
    \begin{equation}\label{E:lsc-area-meas-epsilon-estimate2}
        \Hn(\bigcup_{i=1}^{i_0} f(K_i))> \Hn(f(E)) - \varepsilon.
    \end{equation}
    Note that we do not require the $K_i$ to cover much of $E$; Indeed, if $f$ is highly non-injective then the $K_i$ necessarily cover very little of $E$ as the $f(K_i)$ are disjoint.
    
    In the second part, we use the decomposition above and solve the problem on each $K_i$ separately via Lemma \ref{L:lsc-area-metric-space-special}.
    To this end, note that since the $f(K_i)$'s are disjoint compact sets, there is a $\rho>0$ such that for every $i,j\in\{1,\dots, i_0\}$, with $i\not=j$ we have
    \begin{equation*}
        \dist(f(K_i),f(K_j))>\rho.
    \end{equation*}
    Observe that if $g\in\Lambda_L$ with $\lVert g-f \rVert_{\infty}\leq \frac{\rho}{4}$ then
    \begin{equation*}
        \dist(g(K_i),g(K_j))>\frac{\rho}{2},
    \end{equation*}
    in particular the sets $g(K_i)$ are disjoint. By Lemma \ref{L:lsc-area-metric-space-special}, for each $i\in\{1,\dots, i_0\}$, there is a $\delta_i>0$ such that if $g\in \Lambda_L$ satisfies $\lVert g-f \rVert_{\ell^\infty}\leq \delta_i$, then
    \begin{equation}\label{E:lsc-area-eta-estimate}
        \Hn(g(K_i))\geq \eta \Hn(f(K_i)).
    \end{equation}
    Let now $\delta=\min\{\frac{\rho}{4}, \min_{i=1\dots i_0}\{\delta_i\}\}$ and suppose $g\in \Lambda_L$ satisfies $\lVert g-f \rVert_{\ell^\infty}\leq \delta$. Then by disjointness of $g(K_i)$'s, we get
    \begin{equation*}
        \begin{split}
            \Hn(g(E))&\geq \Hn(g(\bigcup_{i=1}^{i_0} K_i))=\sum_{i=1}^{i_0} \Hn(g(K_i))
            \overset{\eqref{E:lsc-area-eta-estimate}}{\geq} \eta \sum_{i=1}^{i_0} \Hn(f(K_i))\\
            &\geq \eta \Hn(\bigcup_{i=1}^{i_0} f( K_i))
            \overset{\eqref{E:lsc-area-meas-epsilon-estimate2}}{>} \eta \Hn(f(E))- \eta \varepsilon
            \overset{\eqref{E:lsc-area-choice}}{>} C.
        \end{split}
    \end{equation*}
    
    We have shown that to each $f$ in the set in \eqref{E:lsc-area-set} there is a $\delta$ such that the $\delta$-ball around $f$ in $\Lambda_L$ is in the set. Hence the set is open and we are done.
    
    The general case, i.e.~$\Hn(E)=\infty$, can be reduced to the finite case in the following way. Let $f\in \Lambda_L(E)$ be such that
    \begin{equation*}
        \Hn(f(E))>C.
    \end{equation*}
    Then, as $E$ is $n$-rectifiable, its $\Hn$-measure is $\sigma$-finite, whence there is some $\Hn$-measurable set $K\subset E$ with $\Hn(K)<\infty$ and such that
    \begin{equation*}
        \Hn(f(K))>C.
    \end{equation*}
    Now we simply use lower semi-continuity of $\mathcal{A}_K$ on $(\tLip_{L}(K,\reals^m), \lVert \cdot \rVert_{\infty})$.
\end{proof}

\begin{theorem}\label{T:lsc-area-formula-general}
    The functional
    \begin{equation*}
        f\mapsto \int_E J_E f \diff\Hn
    \end{equation*}
    is lower semi-continuous on $\Lambda_L$ for every $L\in[0,\infty)$.
\end{theorem}
\begin{proof}
    Suppose $C\in [0,\infty)$ and $f\in \Lambda_L$ is such that
    \begin{equation*}
        \int_E J_E f \diff \Hn >C.
    \end{equation*}
    Find the sets $E_i$ from Lemma \ref{L:prel-decomp-nonzero-Jacobian} so that by \eqref{E:prel-partial-area-formula} we have
    \begin{equation*}
        \sum_{i} \Hn(f(E_i))= \int_E J_E f \diff \Hn >C.
    \end{equation*}
    Now there is some $i_0\in\nat$ such that
    \begin{equation*}
        \sum_{i=1}^{i_0} \Hn(f(E_i))> C.
    \end{equation*}
    By Theorem \ref{T:lsc-area-general}, to each $i=1,\dots, i_0$, there is some $\delta_i>0$ so that if $g\in \Lambda_L$ satisfies
    \begin{equation*}
        \lVert f-g \rVert_{\ell^{\infty}(E_i)}\leq \delta_i \quad \text{for all $i\in\{1,\dots, i_0\}$},
    \end{equation*}
    then
    \begin{equation*}
        \sum_{i=1}^{i_0}\Hn(g(E_i))>C.
    \end{equation*}
    Let $\delta= \min_{i} \delta_i$. Then if $g\in \Lambda_L$ is such that $\lVert f-g \rVert_{\infty} \leq \delta$, we have by Lemma \ref{L:prel-ideal} together with the fact that $E_i$ are disjoint and the area formula \eqref{E:area-formula}
    \begin{equation*}
        \int_E J_E g \diff \Hn \geq \sum_{i=1}^{i_0} \int_{E_i} J_{E_i} g \diff \Hn
        \geq \sum_{i=1}^{i_0} \Hn(g(E_i)) >C.
    \end{equation*}
\end{proof}


\section{General density statements}\label{S:density-general}

The purpose of this section is introduce some of the density results, in the spaces $\tLip_L^{\textnormal{str}}(X,\reals^m)$ and $\tLip_L(X,\reals^m)$ for a general complete metric space $X$.

The adopted approach is the natural one arising from use of Lemma \ref{L:MSR-open-decomposition}. We construct a Lipschitz function on some pieces of the $n$-rectifiable space $E$ and then extend it onto $X$. However, this is highly non-trivial, as usually it is necessary to do this with no increase (or, in some sense, arbitrarily small increase) in the Lipschitz constant.

This is the purpose of the following lemma. The proof follows, almost to the word, the proof of \cite[Lemma 4.6]{B}. 

\begin{lemma}[Lipschitz extension lemma]\label{L:Lipschitz-extension-lemma}
    Let $X$ be a metric space and $(Y,\lVert \cdot \rVert)$ a normed linear space. Further, let $L\in [0,\infty)$, $N\in \nat$ and let $S_i\subset X$, $i=1,\dots,N$. Let $\delta>0$ and assume $\rho_i\in (0,1]$, $i=1,\dots, N$ are such that the sets $B(S_i,\rho_i)$ are disjoint. Assume that $f\colon X \to Y$ is an $L$-Lipschitz function and let $g_i\colon B(S_i,\rho_i)\to Y$ be $L$-Lipschitz functions such that
    \begin{equation*}
        \lVert g_i-f \rVert_{\ell^\infty(B(S_i,\rho_i))}\leq \delta \rho_i
    \end{equation*}
    for each $i=1,\dots, N$. Then, there exists $g\colon X \to Y$ an $(L+4\delta)$-Lipschitz map such that $g=g_i$ on each $S_i$,
    \begin{equation*}
        \lVert g - f \rVert_{\ell^\infty(X)}<\delta
    \end{equation*}
    and $g=f$ outside $\bigcup_i B(S_i,\rho_i)$.
\end{lemma}
\begin{proof}
    On each $B(S_i,\rho_i)$ we may write
    \begin{equation*}
        g_i=f+E_i,
    \end{equation*}
    where $\lVert E_i \rVert_{\ell^\infty}<\delta \rho_i$. Define $\chi_i\colon X \to \reals$ by
    \begin{equation*}
        \chi_i(x)=\frac{\max\{\frac{1}{2}\rho_i-\dist(x,S_i), 0\}}{\frac{1}{2}\rho_i}.
    \end{equation*}
    Then $\chi_i$'s have disjoint supports contained in $B(S_i,\frac{1}{2}\rho_i)$. Hence, it is valid to define $g\colon X \to Y$ by
    \begin{equation*}
        g=f+\sum_{i=1}^N \chi_i E_i.
    \end{equation*}
    From this definition and the fact that each $\rho_i\leq 1$, all of the stated properties of $g$, except for the Lipschitz constant, immediately follow. 
    
    It remains to show that for any $x,y\in X$, we have
    \begin{equation}\label{E:L-ext-lemma-lipschitz}
        \lVert g(x)-g(y)\rVert \leq L d(x,y)+2\delta d(x,y).
    \end{equation}
    If there is an $i=1,\dots, N$ such that $\chi_i(x),\chi_i(y)>0$, then one can show \eqref{E:L-ext-lemma-lipschitz} mutatis mutandis as in \cite[Lemma 4.6]{B}. If there are $i\not= j$ such that $\chi_i(x)>0$ and $\chi_j(y)>0$, then one can show
    \begin{equation}\label{E:L-ext-lemma-estimate}
        \tfrac{1}{2}\rho_i - \dist(x,S_i)\leq d(x,y)\quad \text{and} \quad 
        \tfrac{1}{2}\rho_j - \dist(y,S_j)\leq d(x,y).
    \end{equation}
    From these inequalities, we obtain
    \begin{equation*}
        \begin{split}
            \lVert g(y)-g(x)\rVert&
            =\lVert f(y)-f(x) + \chi_i(x)E_i(x)-\chi_j(y)E_j(y)\rVert\\
            &\leq L d(x,y)+ |\chi_i (x)| \lVert E_i(x)\rVert + |\chi_j(y)| \lVert E_j(y) \rVert\\
            &\leq L d(x,y) + \delta \rho_i \frac{\frac{1}{2}\rho_i - \dist(x,S_i)}{\frac{1}{2}\rho_i}+\delta \rho_j \frac{\frac{1}{2}\rho_j - \dist(x,S_j)}{\frac{1}{2}\rho_j}\\
            &\leq L d(x,y)+ 4\delta d(x,y).
        \end{split}
    \end{equation*}
    What remains to show is the case when $\chi_i(x)>0$ holds but $\chi_j(y)=0$ for all $j=1,\dots, N$ (which implies the first half of \eqref{E:L-ext-lemma-estimate}). In this case we similarly have
    \begin{equation*}
        \begin{split}
            \lVert g(y)-g(x)\rVert&
            =\lVert f(y)-f(x) + \chi_i(x)E_i(x)\rVert\\
            &\leq L d(x,y)+ |\chi_i (x)| \lVert E_i(x)\rVert \\
            &\leq L d(x,y) + \delta \rho_i \frac{\frac{1}{2}\rho_i - \dist(x,S_i))}{\frac{1}{2}\rho_i}\\
            &\leq L d(x,y)+ 2\delta d(x,y).
        \end{split}
    \end{equation*}
\end{proof}

The following simple observation is the main idea behind how to avoid losing measure (of $f(E)$) by overlapping.

\begin{lemma}\label{L:intersection-lemma-general}
    Let $m,n\in \nat$ and $S\subset \reals^m$ be an $\Hn$-measurable set with $\Hn(S)<\infty$. Let $M$ be any set of $n$-dimensional affine subspaces of $\reals^m$. Then the set
    \begin{equation*}
        \{T\in M: \Hn(T\cap S)>0\}
    \end{equation*}
    is countable.
\end{lemma}
\begin{proof}
     Suppose the statement fails. Then there exists some $\varepsilon>0$ such that the set
     \begin{equation*}
         \{T\in M: \Hn(T\cap S)>\varepsilon\}
     \end{equation*}
     is infinite (even uncountable). Therefore, we may find a countable family $T_i$, $i\in \nat$ of distinct elements of the aforementioned infinite set. 
     As $\Hn(T_i \cap T_j)=0$ whenever $j\not=i$ and
     as $S$ is $\Hn$-measurable and so is every $T_i$, we have
     \begin{equation*}
         \begin{split}
             \Hn(S)\geq \Hn(S\cap \bigcup_{i\in\nat} T_i)= \sum_{i\in \nat} \Hn(S\cap T_i)\geq \sum_{i\in\nat} \varepsilon = \infty,
         \end{split}
     \end{equation*}
     a contradiction.
\end{proof}

The idea of the following sequence of statements is to show that a Lipschitz function may be approximated in the strong distance $\lVert \cdot \rVert_{\infty}+\tLip(\cdot)$ by Lipschitz functions $g$ which simultaneously have very little overlap and ``almost'' satisfy $g_\#\Hn_{|E}\ll \Hn$. This is firstly done assuming $E$ is a normed set.

Within the proof of these results we will repeatedly make use of the following fact.
If $f\colon X \to Y$ is a $L$-Lipschitz function on a \emph{compact} metric space then the functions $\lambda f$, for $0<\lambda <1$, approximate $f$ in the strong distance and have Lipschitz constant strictly less than $L$.
This gives us a small amount of space in which to construct a modification of $\lambda f$ without leaving the space $\tLip_L^{\textnormal{str}}(X,Y)$ and whilst also approximating $f$.

\begin{lemma}\label{L:density-construction-vol-and-overlap-normed}
    Let $E\subset \reals^n$ be a compact set and let $S\subset \reals^m$ be $\Hn$-measurable with $\Hn(S)<\infty$.
    Let $|\cdot|_a$ and $|\cdot|_b$ be norms on $\reals^n$ and $\reals^m$ respectively.
    Let $C>0$, $L\in[0,\infty)$ and let $f\colon E_a\to \reals_b^m$ be an $L$-Lipschitz function. Then for every $\varepsilon>0$, there is a function $g\colon E\to \reals^m$ and $\Hn$-measurable sets $F\subset E$, $N\subset \reals^m$ such that
    \begin{enumerate}
        \item\label{Enum:NL1} $g$ is Lipschitz as a map $g\colon E_a \to \reals^m_b$, with Lipschitz constant $\tLip (g)< L$,
        \item\label{Enum:NL2} $\lVert g-f \rVert_{\infty}< \varepsilon$,
        \item\label{Enum:NL3} $\tLip_{E_a \to \reals^m_b}(g-f)<\varepsilon$,
        \item\label{Enum:NL4} $\Hn_a(E\setminus F)<\frac{1}{L}C$,
        \item\label{Enum:NL5} $\Hn(N)=0$ and $\{u\in\reals^m: \#g^{-1}(u)>1\}\subset g(E\setminus F) \cup N$,
        \item\label{Enum:NL6} the set $F$ admits a decomposition $F=\bigcup_{i=1}^{i_0}F_i $, for some $i_0\in\nat$, such that for each $i=1,\dots,i_0$, $F_i$ is $\Hn$-measurable and $g_{|F_i}$ is a restriction of an affine map with positive volume,
        \item\label{Enum:NL7} $\Hn(g(F)\cap S)=0$.
    \end{enumerate}
\end{lemma}
\begin{proof}
 First, as $B(E,1)$ is bounded, we can find $\Delta<\infty$ so that for any countable disjoint system of balls $B_i=B(x_i,r_i)\subset B(E,1)$, we have
    \begin{equation*}
        \sum_i r_i^n \leq \Delta.
    \end{equation*}
    Let $\sigma\in(0,1)$ be such that 
    \begin{equation}\label{E:DC-sigma-choice}
        \Delta(1-\sigma^n)< \frac{1}{2}\frac{1}{L}C.
    \end{equation}
    We may assume that $L_0=\tLip (f) < L$.
    Let $\delta\in (0,\tfrac{1}{\sqrt{m}}\varepsilon)$ be such that $L_0+3\delta<L$
    
    We use Lusin's theorem to find an $\Hn$-measurable set $\widetilde{E}\subset E$ with
    \begin{equation}\label{E:additional-meas-est}
        \Hn(E\setminus \widetilde E) < \frac{1}{2}\frac{1}{L}C,
    \end{equation}
    such that $x\mapsto f'(x)$ is uniformly continuous on $\widetilde E$. This implies that $x\mapsto \vol f'(x)$ is also uniformly continuous and it also allows us to obtain uniform approximations with derivatives in the following way.
    To each $\delta>0$, there is $r_0>0$ such that for every $r\leq r_0$ and every $x\in \widetilde E$ a density point of $\widetilde E$, we have
    \begin{equation*}
        |f(y)-f(z)-f'(x)(y-z)|_b<\frac{1}{2}\delta(1-\sigma)|y-z|_a\quad\text{for all $y,z\in B(x,r)\cap \widetilde E$.}
    \end{equation*}
    Moreover, the operator norm $\lVert f'(x)\rVert_{a\to b}$ is no larger than $L_0$.
    
    Let now $x$ and $r$ be as above and consider the function $d\colon B(x,r)\cap \widetilde E \to \reals^m$ given by
    \begin{equation*}
        d(y)=f(y)-f(x)-f'(x)(y-x) \quad \text{for $B(x,r)\cap \widetilde E$.}
    \end{equation*}
    Then for $y, z\in B(x,r)\cap \widetilde E \to \reals^m$ we may estimate
    \begin{equation*}
    \begin{split}
        |d(y)-d(z)|_b
        &\leq|f(y)-f(x)+f'(x)(y-x)-f(z)+f(x)-f'(x)(z-x)|_b\\
        &\leq |f(y)-f(z)-f'(x)(y-z)|< \frac{1}{2}\delta(1-\sigma)|y-z|_a,
    \end{split}
    \end{equation*}
    which means that $d$ is $(\frac{1}{2}\delta(1-\sigma))$-Lipschitz on the relevant domain.
    Hence, there is an open neighbourhood $\widetilde{U}_x$ of $f'(x)$ in $\mathcal{L}(\reals_a^n,\reals_b^m)$, the space of linear operators between $\reals_a^n$ and $\reals_b^m$, such that for each $A\in \widetilde{U}_x$ it holds that
    \begin{equation*}
        \lVert A \rVert_{a\to b} < L_0+\delta,
    \end{equation*}
    \begin{equation*}
        |f(y)-f(x)-A(y-x)|_b< \delta(1-\sigma) r \quad \text{for all $y\in B(x,r)\cap E$}
    \end{equation*}
    and
    \begin{equation*}
        \tLip_{B(x,r)\cap \widetilde E}(y\mapsto f(y)-f(x)+A(y-x))< \delta (1-\sigma).
    \end{equation*}
    Observe also that $U_x=\{A\in \widetilde{U}_x: \vol A>0\}$ is open and not empty.
    Using Vitali's covering theorem, we find a countable number of disjoint balls $B_i=B(x_i,r_i)$ and open, non-empty sets $U_i\in \mathcal{L}(\reals_a^n,\reals_b^m)$ such that
    \begin{equation}\label{E:DC-f-A-estimate}
        |f(y)-f(x_i)-A(y-x_i)|_b\leq \delta(1-\sigma) r_i\quad\text{for all $y\in B_i\cap E$ and $A\in U_i$,}
    \end{equation}
    \begin{equation}\label{E:DC-f-A-estimate-lip}
        \tLip_{B_i\cap \widetilde E}(y\mapsto f(y)-f(x_i)+A(y-x_i))< \delta (1-\sigma)\quad\text{for any $A\in U_i$,}
    \end{equation}
    \begin{equation*}
        \Hn(\widetilde E\setminus \bigcup_i B_i)=0,
    \end{equation*}
    and $\vol A>0$ for any $A\in U_i$.
    Observe that now
    \begin{equation*}
        \Hn(\widetilde E\setminus \bigcup \sigma B_i)
        \leq \Delta(1-\sigma^n).
    \end{equation*}
    Whence, by \eqref{E:DC-sigma-choice} and $\eqref{E:additional-meas-est}$, there is an index $i_0\in \nat$ such that
    \begin{equation}\label{E:DC-measure-estimate}
        \Hn(E\setminus \bigcup_{i=1}^{i_0} \sigma B_i)< \frac{1}{L}C.
    \end{equation}
    
    Suppose now that $d \colon \bigcup_{i=1}^{i_0} \sigma B_i \cap \widetilde E \to \reals^m$ satisfies for each $i=1,\dots, i_0$ 
    \begin{equation}\label{E:DC-d-sat}
        d(y)=f(y)-f(x_i)+A_i(y-x_i)\quad \text{for some $A_i\in U_i$ and all $y\in \sigma B_i\cap \widetilde E_i$.}
    \end{equation}
    Then by \eqref{E:DC-f-A-estimate-lip}, $d$ is $(\delta (1-\sigma))$-Lipschitz on each $\sigma B_i\cap \widetilde E$. If $i\not=j$ and $y\in \sigma B_i\cap \widetilde E$, $z\in \sigma B_j\cap \widetilde E$, then
    \begin{equation*}
    \begin{split}
        |d(y)-d(z)|_b
        &\leq |d(y)-d(x_i)|_b + |d(x_i)-d(x_j)|_b + |d(x_j)-d(z)|_b\\
        &\overset{\eqref{E:DC-f-A-estimate}}{\leq} \delta(1-\sigma) r_i + \delta(1-\sigma) r_j.
    \end{split}
    \end{equation*}
    On the other hand
    \begin{equation*}
        |y-z|_a \geq (1-\sigma)r_i+(1-\sigma)r_j
    \end{equation*}
    as $B_i \cap B_j = \emptyset$.
    This means that $d$ is $\delta$-Lipschitz on the set $\bigcup_{i=1}^{i_0} \sigma B_i \cap \widetilde E$
    
    We construct $g_i$ inductively on the sets $B_i$ for $i=1,\dots, i_0$. Let $S_0=S$. By Lemma \ref{L:intersection-lemma-general}, there is some $A_1\in U_1$ such that $\Hn([f(x_1)+A_1(\reals^n)]\cap S_0)=0$. Define
    \begin{equation*}
        g_1(y)=f(x_1)+A_1(y-x_1) \quad \text{for $y\in B_1$.}
    \end{equation*}
    Assume we have defined $g_i$ for $i=1,\dots, k-1$, $k\leq i_0$. Let $S_k=\bigcup_{i=1}^{k-1} g(B_i) \cup S_0$. As $\Hn(S_k)<\infty$, using Lemma \ref{L:intersection-lemma-general}, there is an $A_k\in U_k$ such that $\Hn([f(x_k)+A_k(\reals^n)]\cap S_k)=0$. Let
    \begin{equation*}
        g_k(y)=f(x_k)+A_k(y-x_k)\quad \text{for $y\in B_k$.}
    \end{equation*}
    Thus, we have constructed $(L_0+\delta)$-Lipschitz functions $g_i\colon B_i \to \reals^m$ satisfying
    \begin{equation*}
        \lVert g_i-f \rVert_{\ell^\infty(B_i\cap E)}\leq \delta(1-\sigma) r_i.
    \end{equation*}
    Now we let
    \begin{equation*}
        d=f-g_i \quad\text{on $B_i\cap \widetilde E$, $i=1,\dots,i_0$.}
    \end{equation*}
    By construction, $d$ satisfies \eqref{E:DC-d-sat}. Whence, by our choice of $\delta$, $d$ admits an extension (denoted again by $d$) onto entire $\reals^n$ such that $\tLip(d) < \varepsilon$ and $\lVert d \rVert_{\infty}< \varepsilon$.
    Indeed, here we have just used McShane extension together with the estimate $\delta < \frac{1}{\sqrt{m}} \varepsilon$.
    We let $g=f-d$ on $E$. This implies immediately that \ref{Enum:NL1}, \ref{Enum:NL2}, and \ref{Enum:NL3} are satisfied.
    
    We let $F_i=\sigma B_i\cap E$ and
    \begin{equation*}
        N=\bigcup_{i=1}^{i_0} S_i \cap g(F_i).
    \end{equation*}
    Then by construction $\Hn(N)=0$. Let $F=\bigcup_{i=1}^{i_0} F_i$. From the construction \ref{Enum:NL6} and \ref{Enum:NL7} follow immediately and \ref{Enum:NL4} holds due to \eqref{E:DC-measure-estimate}.
    
    Finally, to show \ref{Enum:NL5}, suppose $u\in \reals^m$ has two preimages under $g$ neither of which lies in $E\setminus F$, i.e.~there are some $x,y\in F$ such that $g(x)=g(y)=u$.
    Then, there are some $i,j \in \{1,\dots, i_0\}$ such that $x\in F_i$, $y\in F_j$ and, without loss of generality, $j\leq i$.
    As $g$ is by construction injective on each $F_i$, we have $j<i$.
    Therefore by definition of $S_i$, necessarily $g(y)\in S_i$. On the other hand, $x \in F_i$ implies $g(x)\in g(F_i)$.
    Altogether $u\in S_i\cap g(F_i)$ and so $u\in N$.
\end{proof}

Now we may use Lemma \ref{L:Kirch} to push the results from normed sets into general metric spaces.

\begin{theorem}\label{T:density-construction-vol-and-overlap-metric}
 Let $X$ be a complete metric space and let $E\subset X$ be an $n$-rectifiable subset with $\Hn(E)<\infty$, let $L\in[0,\infty)$ and let $|\cdot|_b$ be a norm on $\reals^m$. Suppose $f \colon X\to \reals^m_b$ is an $L$-Lipschitz function. Then to each $\varepsilon>0$ and $C>0$, there is a function $g\colon X \to \reals^m$ and $\Hn$-measurable sets $F\subset E$, $N\subset \reals^m$ such that
 \begin{enumerate}
     \item\label{Enum:MT1} $g\colon X \to \reals^m_b$ is Lipschitz with constant $\tLip(g)<L$,
     \item\label{Enum:MT2} $\lVert g-f \rVert_{\infty}<\varepsilon$,
     \item\label{Enum:MT3} $\tLip(g-f)< \varepsilon$
     \item\label{Enum:MT4} $\Hn(E \setminus F)< \frac{1}{L} C$,
     \item\label{Enum:MT5} $\Hn(N)=0$ and $\{u\in g(E): \#g^{-1}(u)>1\}\subset g(E\setminus F)\cup N$,
     \item\label{Enum:MT6} $J_E g>0$ $\Hn$-a.e.~on $F$.
 \end{enumerate}
\end{theorem}
\begin{proof}
    We may assume that $\tLip(f)=L_0<L$. Let $0< \theta < \infty$. Let $C_i>0$, $i\in\nat$ be a sequence for which 
    \begin{equation}\label{E:DC-M-Ci}
        \frac{1}{2L}C+ (1+\theta)L \sum_i C_i< \frac{1}{L}C.
    \end{equation}
    By Lemma \ref{L:Kirch}, we find compact sets $E_i\subset E$, $K_i\subset \reals^n$, norms $|\cdot|_{a_i}$ on $\reals^n$ and $(1+\theta)$-biLipschitz bijections $I_i\colon (K_i,|\cdot|_{a_i})\to E_i$ such that
    \begin{equation*}
        \Hn(E\setminus \bigcup_i E_i)=0.
    \end{equation*}
    Find $i_0\in\nat$ such that 
     \begin{equation}\label{E:DC-M-measure-est}
         \Hn(E\setminus \bigcup_{i=1}^{i_0} E_i)< \frac{1}{2L}C.
     \end{equation}
     As $E_i$'s are compact, there is some $\rho>0$ such that $B(E_i, \rho)\cap B(E_j,\rho)=\emptyset$ provided $i,j\in\{1,\dots, i_0\}$, $i\not= j$.
     Let now $0\leq C^0_i\leq 1$, $i=1,\dots, i_0$ be such that
     \begin{equation}\label{E:DC-M-Cnaught-con}
         \sum_{i=1}^{i_0} \sqrt{m}\varepsilon C^0_i \max\{(1+\theta), \tfrac{1}{\rho}\}\leq \delta
     \end{equation}
     For each $i=1,\dots, i_0$ we let $\widetilde{f}_i=f\circ I_i\colon (K_i, |\cdot|_{a_i}) \to \reals^m_b$, a $((1+\theta)L_0)$-Lipschitz function. Denoting $S_1=\emptyset$,
     we use Lemma \ref{L:density-construction-vol-and-overlap-normed} to find $\widetilde{g_1}\colon (K_1, |\cdot|_{a_1})\to \reals^m_b$, a $(1+\theta)L_0$-Lipschitz function and sets $\widetilde{F}_1\subset K_1$, $N_1 \subset \reals^m$ such that
     \begin{enumerate}[label={(\roman*)$^{1}$}]
         \item $\lVert \widetilde{g}_1-\widetilde{f}_1 \rVert_{\ell^\infty(K_1)} < \varepsilon C^0_1$,
         \item $\tLip_{(K_1, |\cdot|_{a_1})\to \reals^m_b}(\widetilde{g}_1-\widetilde{f}_1)  < \varepsilon C^0_1$
         \item $\H^n_{a_1}(K_1\setminus \widetilde{F}_1)< \frac{1}{L} C_1$,
         \item $\Hn(N_1)=0$,
         \item $\{u\in \widetilde{g}_1(K_1): \#\widetilde{g}_1^{-1}(u)>1\}\subset \widetilde{g}_1(K_1\setminus \widetilde{F}_1)\cup N_1$
         \item $\vol \widetilde{g}_1'>0\;$ $\Hn$-a.e.~on $\widetilde{F}_1$.
         \item $\Hn(S_1 \cap \widetilde{g}_1(\widetilde{F}_1))=0$
     \end{enumerate}
     Suppose now that $\widetilde{g}_i$ have been constructed for $i=1,\dots, k-1$ where $k\in\{1, \dots, i_0\}$. We set
     \begin{equation*}
         S_k=\bigcup_{i=1}^{k-1} \widetilde{g}_i(K_i).
     \end{equation*}
     Using Lemma \ref{L:density-construction-vol-and-overlap-normed} once again, we find $\widetilde{g_k}\colon (K_k, |\cdot|_{a_k})\to \reals^m_b$, a $(1+\theta)L_0$-Lipschitz function and sets $\widetilde{F}_k\subset K_k$, $N_k \subset \reals^m$ such that
     \begin{enumerate}[label={(\roman*)$^{k}$}]
         \item \label{Enum:bddness} $\lVert \widetilde{g}_k-\widetilde{f}_k \rVert_{\ell^\infty(K_k)} < \varepsilon C^0_k$,
         \item \label{Enum:lipsch}$\tLip_{(K_k, |\cdot|_{a_k})\to \reals^m_b}(\widetilde{g}_k-\widetilde{f}_k)  < \varepsilon C^0_k$
         \item $\H^n_{a_k}(K_k\setminus \widetilde{F}_k)< \frac{1}{L} C_k$,
         \item $\Hn(N_k)=0$,
         \item $\{u\in \widetilde{g}_k(K_k): \#\widetilde{g}_k^{-1}(u)>1\}\subset \widetilde{g}_k(K_k\setminus \widetilde{F}_k)\cup N_k$
         \item $\vol \widetilde{g}_k'>0\;$ $\Hn$-a.e.~on $\widetilde{F}_k$.
         \item \label{Enum:overlap} $\Hn(S_k \cap \widetilde{g}_k(\widetilde{F}_k))=0$.
     \end{enumerate}
     
     
     For each $i=1, \dots, i_0$, we let $F_i=I_i^{-1}(\widetilde{{F_i}})$ and $g_i=\widetilde{g}_i\circ I^{-1}_i$.
     Moreover, we define
     \begin{equation*}
         d_i(x)=\begin{cases}
				g_i(x)-f(x)
					& \text{if $x\in E_i$},
					\\
				0
					& \text{if $x\in X\setminus B(E_i,\rho)$}.
			\end{cases}
     \end{equation*}
     Using the fact that $B(E_i,\rho)$'s are pairwise disjoint together with the properties \ref{Enum:bddness} and \ref{Enum:lipsch}, we observe that $\lVert d_i \rVert_{\infty}\leq \varepsilon C^0_i$ and $\tLip(d_i)\leq \max\{(1+\theta) \varepsilon C^0_i, \frac{1}{\rho}\varepsilon C^0_i\}$ on the relevant domains. 
     Using McShane extension, we extend each $d_i$ onto the entire $X$, denoting the extensions again by $d_i$, thereby obtaining functions satisfying
     \begin{enumerate}[label={(\alph*)$^{i}$}]
         \item $\lVert d_i \rVert_{\ell^\infty(X)}< \varepsilon C^0_i$,
         \item \label{Enum:d-lip} $\tLip_{X \to \reals^m_b} d_i \leq \sqrt{m} \varepsilon C^0_1 \max\{(1+\theta, \frac{1}{\rho})\}$,
         \item \label{Enum:d-null} $d_i=0$ on $E_j$ for each $j\not= i$.
     \end{enumerate}
     We let
     \begin{equation}
         d=\sum_{i=1}^{i_0} d_i \quad \text{and} \quad g=f-d.
     \end{equation}
     As the function $\tLip(\cdot)$ is sub-additive on $\tLip(X,\reals^m_b)$, using \ref{Enum:d-lip} and \eqref{E:DC-M-Cnaught-con} we have
     \begin{equation*}
         \tLip(d) \leq \sum_{i=1}^{i_0} \sqrt{m} \varepsilon C_i^0 \max\{(1+\theta), \tfrac{1}{\rho}\}\leq \delta. 
     \end{equation*}
     Which, as $\delta< \varepsilon$, implies \ref{Enum:MT3} and, as $L_0+\delta<L$, implies \ref{Enum:MT1}.
     Similarly, since $C^0_i\leq 1$, we also obtain \ref{Enum:MT2}.
     
     By \ref{Enum:d-null}, $g=g_i$ on each $E_i$.
     Therefore, using the fact that $I_i's$ are $(1+\theta)$-biLipschitz together with the properties \ref{Enum:bddness} - \ref{Enum:overlap} we observe the properties
     \begin{enumerate}[label={(\arabic*)}]
         \item \label{Enum:area-Fi} $\Hn(E_i\setminus F_i)\leq (1+\theta) \Hn(K_i\setminus \widetilde{F}_i)\leq (1+\theta)\frac{1}{L}C_i$,
         \item $N_i\subset F_i$ with $\Hn(N_i)=0$ and it holds that $\{u\in g_i(E_i): \#g^{-1}(u)>1\}\subset g_i((E_i\setminus F_i)\cup N_i)$,
         \item \label{Enum:vol-nonzero} $J_E g>0$ $\Hn$-a.e.~on each $F_i$.
     \end{enumerate}
     Let 
     \begin{equation*}
         F=\bigcup_{i=1}^{i_0} F_i.
     \end{equation*}
     Now \ref{Enum:MT6} holds by \ref{Enum:vol-nonzero}. Moreover, we may estimate
     \begin{equation*}
         \begin{split}
             \Hn (E\setminus F)
             = \Hn(\bigcup_{i=1}^{i_0} E_i \setminus F_i)+ \Hn(E\setminus \bigcup_{i=1}^{i_0} E_i),
         \end{split}
     \end{equation*}
     where, by the property \ref{Enum:area-Fi} and the estimate \eqref{E:DC-M-measure-est}, the last expression is estimated from above by 
     \begin{equation*}
         (1+\theta)\frac{1}{L}\sum_{i=1}^{i_0} C_i + \frac{1}{2L}C.
     \end{equation*}
     Whence \ref{Enum:MT4} follows from the choice of $C_i$'s \eqref{E:DC-M-Ci}.
     
     It remains to find $N$ and show \ref{Enum:MT5}. To that end, we simply let
     \begin{equation*}
         N=\left(\bigcup_{i=1}^{i_0}N_i\right)\cup \bigcup_{i=1}^{i_0} S_i \cap g(F_i). 
     \end{equation*}
     As $g(F_i)=\widetilde{g}_i(\widetilde{F}_i)$, using \ref{Enum:overlap}, we see that $\Hn(N)=0$.
     To show \ref{Enum:MT5}, suppose $u\in g(E)$ and $g^{-1}(u)\cap (E \setminus F)=\emptyset$ and assume $u$ has two distinct preimages under $g$.
     It suffices to show that $u\in N$.
     There are $x,y\in F$ such that $g(x)=g(y)=u$. 
     By definition of $F$, we may assume that there are $i,j\in \{1,\dots, i_0\}$, $j\leq i$ such that $x\in F_i$, $y\in F_j$.
     Firstly, assume that $j=i$. Then $\#\widetilde{g}_i^{-1}(u)>1$ and so $u\in \widetilde{g}_i(K_i \setminus \widetilde{F}_i)\cap N_i$.
     However, $u\not\in \widetilde{g}_i(K_i \setminus \widetilde{F}_i)= g_i(E_i\setminus F_i)$ as $g^{-1}(u)\cap (E \setminus F)=\emptyset$.
     Therefore it is necessary that $u \in N_i \subset N$.
     Secondly, assume that $j<i$.
     Then $g(y)\in S_i$ by the definition of $S_i$ and $g(x)\in g(F_i)$ simply because $x\in F_i$.
     Therefore $u\in S_i \cap g(F_i) \subset N$.
     Either way $u\in N$ and we are done.
\end{proof}
    
\begin{corollary}\label{C:no-overlaps}
    Let $n<m$, let $X$ be a complete metric space and let $E\subset X$ be an $n$-rectifiable subset.
    Let $L\in [0,\infty)$. Then
    the set 
    \begin{equation*}
        \{f\in \tLip_L^{\textnormal{str}}(X,\reals^m): \int_E J_E f \diff \Hn = \Hn(f(E))\},
    \end{equation*}
    is residual in $\tLip_L^{\textnormal{str}}(X,\reals^m)$.
\end{corollary}
\begin{proof}
    Firstly, we may write $E=\bigcup_{i=1}^{\infty} E_i$, where $E_i$ form an increasing sequence of $\Hn$-measurable sets with $\Hn(E_i)<\infty$.
    Now if a function $f\in \tLip^{\textnormal{str}}_{L}(X,\reals^m)$ satisfies
    \begin{equation*}
        \int_{E_i} J_{E_i} f \diff \Hn=\Hn(f(E_i))
    \end{equation*}
    on each $E_i$, then it obviously also satisfies the assertion.
    As a countable intersection of residual sets is residual,
    we may assume, without loss of generality, that $\Hn(E)<\infty$.
     By the area formula \eqref{E:area-formula}, it suffices to show that the sets
     \begin{equation*}
         \{f\in \tLip_L^{\textnormal{str}}(X,\reals^m): \int_E J_E f \diff \Hn - \Hn(f(E))< \frac{1}{i}\}
     \end{equation*}
     are open and dense in $\tLip_L^{\textnormal{str}}(X,\reals^m)$ for every $i\in \nat$.
     Density follows immediately from Theorem \ref{T:density-construction-vol-and-overlap-metric} by taking $C$ sufficiently small (depending on $i$).
     For openness, it is sufficient to observe that the functional
     \begin{equation}\label{E:usc-functional}
         f\mapsto \int_E J_E f \diff \Hn - \Hn(f(E))
     \end{equation}
     is upper semi-continuous on $\tLip_L^{\textnormal{str}}(X,\reals^m)$.
     The functional 
     \begin{equation*}
          f\mapsto \int_E J_E f \diff \Hn
     \end{equation*}
     is easily seen to be continuous on $\tLip_L^{\textnormal{str}}(X,\reals^m)$.
     Moreover, the functional $f\mapsto \Hn(f(E))$ is lower semi-continuous on $\tLip_L(X,\reals^m)$, therefore it is also lower semi-continuous on $\tLip_L^{\textnormal{str}}(X,\reals^m)$.
     Therefore the functional $f\mapsto -\Hn(f(E))$ is upper semi-continuous on $\tLip_L^{\textnormal{str}}(X,\reals^m)$.
     Altogether, we have shown that the functional in \eqref{E:usc-functional} is a sum of upper semi-continuous functionals and as such it is upper semi-continuous.
\end{proof}

The proof of the following corollary is now trivial upon recalling Baire's theorem and the preceding corollary, together with the fact that sets which are dense in $\tLip_L^{\textnormal{str}}(X,\reals^m)$ are also dense in $\tLip_L(X,\reals^m)$.

\begin{corollary}\label{C:no-overlap-density}
     Let $n<m$, let $X$ be a complete metric space and let $E\subset X$ be an $n$-rectifiable subset.
    Let $L\in [0,\infty)$. Then
    the set
    \begin{equation*}
        \{f\in \tLip_L(X,\reals^m): \int_E J_E f \diff \Hn = \Hn(f(E))\},
    \end{equation*}
    is dense in $\tLip_L(X,\reals^m)$.
\end{corollary}

\begin{corollary}\label{C:pushforward}
    Let $n\leq m$, let $X$ be a complete metric space and let $E\subset X$ be an $n$-rectifiable subset.
    Let $L\in [0,\infty)$. Then the set of the $L$-Lipschitz functions $f\colon X \to E$ such that
    \begin{equation}\label{E:pushforward}
        f_\# \Hn_{|E} \ll \Hn_{\reals^m}
    \end{equation}
    is residual in $\tLip_L^{\textnormal{str}}(X,\reals^m)$.
\end{corollary}
\begin{proof}
    As in the proof of Corollary \ref{C:no-overlaps}, we may assume $\Hn(E)<\infty$.
     Recalling the area formula \eqref{E:area-formula} it is easily seen that for a Lipschitz function $f\colon X \to \reals^m$ one has \eqref{E:pushforward} if and only if $J_E f > 0$ $\Hn$-a.e. in $E$. We can write
     \begin{equation}\label{E:pushforward-intersection}
     \begin{split}
         &\{f\in \tLip_L^{\textnormal{str}}(X,\reals^m): J_E f > 0 \; \text{$\Hn$-a.e.~on $E$}\}\\
         =&\bigcap_{i\in\nat}\{f\in \tLip_L^{\textnormal{str}}(X,\reals^m): \Hn(\{x\in E: J_E f(x)=0\})< \tfrac{1}{i}\}.
     \end{split}
     \end{equation}
     It is therefore enough to show that the sets on the right hand side of \eqref{E:pushforward-intersection} are open and dense.
     Density follows immediately from Theorem \ref{T:density-construction-vol-and-overlap-metric} by choosing $C>0$ sufficiently small (we use particularly the properties \ref{Enum:MT4} and \ref{Enum:MT6} therein). 
     
     Openness follows by a rearrangement argument. Assume $\varphi_k\in L^\infty(E)= L^{\infty}(E, \Hn_{|E})$ is a sequence which converges in $L^\infty(E)$ to some function $\varphi \in L^\infty(E)$. Then their non-increasing rearrangements satisfy $\varphi_k^* \to \varphi^*$ in $L^\infty([0,\Hn(E)], \mathcal{L}^1)$. Here the non-increasing rearrangement is defined as
     \begin{equation*}
         \varphi^*(t)= \inf\{\lambda\ge 0: \Hn(\{x\in E: {\varphi(x)>\lambda}\}) \le t\}
		\quad\text{for $t\in[0,\Hn(E)]$.}
     \end{equation*}
     Suppose further that $\Hn(\{x\in E:\varphi_k(x)=0\})\geq \tfrac{1}{i}$. Then $\varphi_k^*(t)=0$ for all $t\in (\Hn(E)-\frac{1}{i}, \Hn(E)]$.
     Then we also have $\varphi^*(t)=0$ for all $t\in (\Hn(E)-\frac{1}{i}, \Hn(E)]$.
     We apply this to Jacobians of the relevant functions to show that the complements of the sets on the right hand side of \eqref{E:pushforward-intersection} are closed. 
     To that end, let $i\in \nat$ be fixed and let $f_k, f\in \tLip_L^{\textnormal{str}}(X,\reals^m)$ be such that $f_k \to f$ in $\tLip_L^{\textnormal{str}}(X,\reals^m)$ and
     \begin{equation*}
         \Hn(\{x\in E: J_E f_k(x)=0\})\geq \frac{1}{i} \quad \text{for all $k\in\nat$.}
     \end{equation*}
     By our assumption on convergence of $f_k$'s, we have $J_E f_k \to J_E f$ in $L^{\infty}(E)$.
     Therefore, by the argument above $(J_E f)^*(t)=0$ for all $t\in (\Hn(E)-\frac{1}{i}, \Hn(E)]$.
     By the definition of the non-increasing rearrangement, this implies $\Hn(\{x\in E: J_E f(x)=0\})\geq \frac{1}{i}$.
     Therefore the set
     \begin{equation*}
         \{f\in \tLip_L^{\textnormal{str}}(X,\reals^m): \Hn(\{x\in E: J_E f(x)=0\})\geq \tfrac{1}{i}\}
     \end{equation*}
     is closed and we are done.
\end{proof}

\section{Residuality of functions with images of large measure}\label{S:residuality-positive}

The purpose of this section is to provide residuality results in the positive direction. Firstly, using the tools developed in previous sections \ref{S:openness} and \ref{S:density-general} we provide a useful characterisation of the residuality of the sets
\begin{equation*}
    \{f\in\tLip_1(X,\reals^m_b): \Hn(f(E))\geq \numberdelta\}=\bigcap_{i\in\nat}  \{f\in\tLip_1(X,\reals^m_b): \Hn(f(E))> \numberdelta-\tfrac{1}{i}\}.
\end{equation*}
This is the subject of the first subsection.

In the second subsection, we use this characterisation to prove the relevant residuality results assuming $X$ is a normed set and $E$ its subset. This only works under particular assumptions on the norm (and also the norm $|\cdot|_b$ in the target space). We do not discuss sharpness of these conditions.

Finally, in the third subsection, we concentrate on the particular instance of Euclidean norms. In this case we are able not only to get the best possible $\numberdelta=\Hn(E)$ but also to push these results into the setting of $n$-rectifiable subsets of $\reals^k$ (as opposed to mere subsets of $\reals^n$). It is there that we prove our first main result Theorem \ref{T:main-Euclidean}.

\subsection{Characterisations of residuality}

\begin{lemma}\label{L:lsc-with-target-norm}
    Let $n\leq m$, let $X$ be a complete metric space and $E\subset X$ an $n$-rectifiable set.
    Let $|\cdot|_b$ be a norm on $\reals^m$.
    The functionals
    \begin{equation*}
        f\mapsto \Hn_{|\cdot|_2}(f(E))
    \end{equation*}
    and
    \begin{equation*}
        f\mapsto \int_{E} J_E f \diff \Hn
    \end{equation*}
    are lower semi-continuous on $\tLip_L(X,\reals^m_b)$ for any $L\in[0,\infty)$.
\end{lemma}
\begin{proof}
    This follows from theorems \ref{T:lsc-area-general} and \ref{T:lsc-area-formula-general} respectively.
\end{proof}

\begin{theorem}\label{T:char-area-area-formula}
    Let $n\leq m$, let $X$ be a complete metric space and let $E\subset X$ be an $n$-rectifiable subset.
    Let $|\cdot|_b$ be a norm on $\reals^m$ and suppose that $L\in[0,\infty)$ and $C>0$. Consider the following statements.
    \begin{enumerate}
        \item The set $\mathcal{A}_{\geq C}=\{f\in \tLip_L(X,\reals^m_b): \Hn(f(E))\geq C\}$ is residual in $\tLip_L(X,\reals^m_b)$.
        \item The sets $\mathcal{A}_{> \widetilde{C}}=\{f\in \tLip_L(X,\reals^m_b): \Hn(f(E))> \widetilde{C}\}$ are dense in $\tLip_L(X,\reals^m_b)$ for all $\widetilde{C}<C$.
        \item The set $\mathcal{V}_{\geq C}=\{f\in \tLip_L(X,\reals^m_b): \int_E J_E f \diff \Hn\geq C\}$ is residual in $\tLip_L(X,\reals^m_b)$.
        \item The sets $\mathcal{V}_{> \widetilde{C}}=\{f\in \tLip_L(X,\reals^m_b): \int_E J_E f\diff \Hn> \widetilde{C}\}$ are dense in $\tLip_L(X,\reals^m_b)$ for all $\widetilde{C}<C$.
    \end{enumerate}
    If $n< m$ then all of the statements are mutually equivalent. If $n=m$ then (i) and (ii) are false and (iii) and (iv) are equivalent.
\end{theorem}
\begin{proof}
    If $n=m$, (i) and (ii) are false as, in fact, any sequence $f_k \colon X \to \reals^m_b$ with $f_k \to 0$  uniformly satisfies $\Hn(f_k(E))\to 0$. 
    If $n<m$ then equivalence of (i) and (ii) follows from Lemma \ref{L:lsc-with-target-norm} as we can write
    \begin{equation*}
        \mathcal{A}_{\geq C}= \bigcap_{i\in\nat} \mathcal{A}_{>(C-\frac{1}{i})}.
    \end{equation*}
    Similarly, equivalence of (iii) and (iv) (for any $n\leq m$) is also obtained from Lemma \ref{L:lsc-with-target-norm}.
    
    Now let $n<m$.
    The fact that (ii) implies (iv) is obviously true due to the area formula \eqref{E:area-formula}.
    It remains to show that (iv) implies (ii). To that end, fix $\widetilde{C}<C$ and $f\in \tLip_L(X,\reals^m_b)$. From (iv) we can find, for each $\varepsilon>0$, some $g\in \tLip_L(X,\reals^m_b)$ with $\lVert g-f \rVert_{\infty}<\varepsilon$ satisfying
    \begin{equation*}
        \int_E J_E g \diff \Hn> \widetilde{C}.
    \end{equation*}
    As the functional $g\mapsto \int_E \vol J_E g \diff \Hn$ is lower semi-continuous, there is some $\delta>0$ such that for any $h\in\tLip_L(X,\reals^m_b)$ satisfying $\lVert g-h \rVert_{\infty}< \delta$ we still have
    \begin{equation}\label{E:h-estimate}
        \int_E J_E h \diff \Hn>\widetilde{C}.
    \end{equation}
    From Corollary \ref{C:no-overlap-density}, we see that the set of functions $h$ satisfying the improved area formula
    \begin{equation}\label{E:improved-area-formula}
        \int_{E} J_E h \diff \Hn = \Hn(h(E))
    \end{equation}
    is dense in $\tLip_L(X,\reals^m)$.
    Therefore, there exists some $h\in \tLip_L(X,\reals^m_b)$ satisfying $\lVert g-h \rVert_{\infty}< \delta$ and \eqref{E:improved-area-formula}.
    By \eqref{E:h-estimate}, such $h$ satisfies
    \begin{equation*}
        \Hn(h(E)) =\int_{E} J_E h \diff \Hn > \widetilde{C}.
    \end{equation*}
    As $\delta>0$ may be reduced to an arbitrarily small number, we have shown that $\mathcal{A}_{\widetilde C}$ is dense, which is (ii).
\end{proof}

\subsection{Positive results in normed sets}

Recalling Lemma \ref{L:MSR-open-decomposition} it would seem that a good starting point to tackling $n$-rectifiable metric spaces, is the study of metric spaces which are merely subsets of $\reals^n$ equipped with a distance induced by a particular norm $|\cdot|_a$.
The object of this section is to provide a sufficient condition (on $|\cdot|_a$) so that there exists some $\lambda>0$ so that the set
\begin{equation*}
    \{f\in \tLip(E_a, \reals^m): \Hn(f(E))> \lambda \Hn(E)\}
\end{equation*}
is residual in $\tLip(E_a, \reals^m)$ for all $m>n$ and all $E\subset \reals^n$ bounded $\Hn$-measurable sets.

We shall work in a slightly more general setting and allow a general norm (which we denote exclusively by $|\cdot|_b$) on the target space $\reals^m$ as well.

It should be noted, that right now, both on the domain and on the target side we are working either with the whole Euclidean space or a subset thereof. We do equip it with a different metric (norm), but the metric is always equivalent to the Euclidean one. This implies that the induced Hausdorff measures (of any dimension) are always equivalent. However, available area formulas are far more conveniently used if the Hausdorff measures considered on either side are induced by the Euclidean distance. This is, up to a constant, without loss of generality. In other words, up to a constant, whenever we write $\Hn$ one may replace it with $\Hn_a$ (on the domain) or $\Hn_b$ (on the target).

There is, however,  a small caveat to what is said above. If we want to prove an estimate holding for an entire \emph{family} of norms then constants matter. This is just a small technical detail, however, to avoid confusion, we will state some of our results in a ``duplicate'' form. One dealing with Euclidean Hausdorff measure and one dealing with the Hausdorff measure induced by the particular norm.

Let us fix some notation for the entirety of this section. We shall assume that $n,m\in\nat$ satisfy $n\leq m$, and that $|\cdot|_a$ and $|\cdot|_b$ are norms on $\reals^n$ and $\reals^m$ respectively.

\begin{observation}\label{Obs:good-curve}
    Let $A\colon \reals\to \reals^m$ be a linear map. Let $u\in \reals^m$ and if $A\not=0$, assume also that $u= \kappa A(1)$ for some $\kappa\in\reals$ with $|\kappa|\geq 1$. Then for every $\varepsilon>0$, there exists a Lipschitz curve $\gamma\colon \reals\to A(\reals)$ such that
    \begin{enumerate}
        \item $\lVert \gamma - A \rVert_{\infty}<\varepsilon$,
        \item $\gamma'$ exists everywhere in $\reals$ up to a discrete set of points,
        \item if $x\in \reals$ is such that $\gamma'(x)$ exists, then $\gamma'(x)=\pm u$.
    \end{enumerate}
\end{observation}
\begin{proof}
    If $A=0$ or $\kappa=1$ the proof is obvious. Assume $\kappa>1$ and $A\not=0$.
    There is a partition of $\reals$ into intervals $[a_i,b_i]$, $i\in\mathbb{Z}$ with $a_i=b_{i-1}$ such that on each $[a_i,b_i]$ we can define $\gamma$ to be an affine curve with $\gamma'=\pm u$ (sign depends on parity of $i$), $\gamma(b_i)=\gamma(a_{i+1})$ for each $i$ and such that $\lVert \gamma - A\rVert_{\ell^\infty([a_i,b_i])}$ is comparable to $|b_i-a_i|$.
    Here we needed to use the assumption that $\kappa> 1$ as otherwise $A$ would ``run away'' from $\gamma$.
    
    For any $\eta>0$, the partition can be made such that $|a_i-b_i|\leq \eta$, for all $i\in \mathbb{Z}$, while still having some $\delta>0$ such that $|b_i-a_i|\geq \delta$ for all $i\in \mathbb{Z}$. The derivative $\gamma'$ then exists everywhere except the endpoints of the intervals, which form a discrete set of points. By making the partition fine enough, that is, taking $\eta>0$ small enough and recalling that $|b_i-a_i|$ is comparable to $\lVert \gamma - A\rVert_{\ell^\infty([a_i,b_i])}$, we can make it so that (i) holds.
\end{proof}

Given a linear space $Y$ of dimension $n$ and a map $I\colon Y \to Y$ which is diagonalizable, i.e.~there are eigenvectors $(u_1,\dots,u_n)$ of $I$, which form a basis of $Y$, we say that $\widetilde{I}$ is a \emph{sign permutation of $I$}, if there are $j(i)\in \{0,1\}$, $i=1,\dots,n$ such that
\begin{equation*}
    \widetilde{I}(u_i)=(-1)^{j(i)}I(u_i).
\end{equation*}
We denote by $\sp I$ the set of all sign permutations of $I$. Observe that the set $\sp I$ is independent of the particular choice of eigenvector basis $(u_1,\dots u_n)$ and indeed $\sp I$ \emph{only} depends on $I$ (and $Y$).
We shall further adopt the name \emph{non-shrinking} to refer to diagonalizable linear maps, whose eigenvalues have absolute values greater than or equal to $1$.

The following definition is a useful way of describing the possibility of approximating a linear map with a piecewise affine map, which has large volume almost everywhere.

\begin{definition}\label{D:inflation}
    Let $A\in B_{a\to b}$ be of full rank and let $\lambda>0$. We say that $A$ \emph{admits a $\lambda$-inflation} if there exists a diagonalizable map $I\colon A(\reals^n) \to A(\reals^n)$ satisfying the following properties. The absolute value of every eigenvalue of $A$ is no smaller than $1$ (i.e.~$I$ is non-shrinking) and for every $\widetilde{I}\in\sp I$ we have
    \begin{enumerate}
        \item\label{Enum:infl-1} $\lVert \widetilde{I}\circ A\rVert_{a\to b}\leq 1$.
        \item\label{Enum:infl-2} $\vol (\widetilde{I}\circ A)\geq \lambda$.
    \end{enumerate}
\end{definition}

\begin{remark}
    In the case that $|\cdot|_b$ is Euclidean (and hence invariant under linear reflections), it is enough to verify $\lVert I \circ A \rVert_{a\to b}\leq 1$ and one does not need to consider sign permutations of $I$.
\end{remark}

\begin{definition}
    The pair of norms $|\cdot|_a$ on $\reals^n$ and $|\cdot|_{b}$ on $\reals^m$ is said to form a \emph{$\lambda$-inflating pair} for a $\lambda>0$ if, for every linear map $A\in B_{a\to b}$ of full rank, $A$ admits a $\lambda$-inflation. We also write that $\infl{a}{b}$ forms a $\lambda$-inflating pair.
\end{definition}

Some geometric intuition for the preceding definitions is in order. If $A$ admits a $\lambda$-inflation, this means that the convex set $A(B_a)\subset A(\reals^n)$ (having a non-empty interior in $A(\reals^n)$) can be inflated in such a way that $A(B_a)\subset B_b\cap A(\reals^n)$ and that the $n$-dimensional (Euclidean) Hausdorff measure of $A(B_a)$ is sufficiently large (this depends on the norms and $\lambda$). Moreover, this inflation must be achieved using a map, which admits a diagonal form with respect to some basis of $A(\reals^n)$ and this map may not shrink in any direction. It is useful to note that this basis corresponds via $A^{-1}\colon A(\reals^n) \to \reals^n$ with a basis in $\reals^n$.

The requirement that the inflation be diagonal will become clear once we prove the principal result. The reason we require that $I$ does not shrink in any direction (a condition on its eigenvalues) is so that we can use Observation \ref{Obs:good-curve} (this corresponds to the condition on $\kappa$ therein).

\begin{example}\label{Ex:Euclidean-inflation}
    Let $|\cdot |_a=|\cdot|_2$ and $|\cdot |_b=|\cdot|_2$. Then every linear map $A\in B_{2\to 2}$ of full rank admits a $1$-inflation. Indeed, any such map $A$ is of the form $A=S\circ D \circ R$, where $R\in \mathrm{O}(n)$, $S\in \mathrm{O}(m)$ and $D$ is diagonal with diagonal values less than or equal to one in magnitude. Moreover, as $A$ is of full rank, the diagonal values of $D$ are non zero and so $D$ admits the inverse $D^{-1}\colon D(\reals^n) \to \reals^n$. Let $E\colon \reals^n\to \reals^m$ be the diagonal matrix with only $1$'s on the diagonal and let $\overline{A}=S\circ E \circ R$. It is an easy exercise to now show that $\overline{A}=IA$ for
    \begin{equation}\label{E:I-formula}
        I=S\circ E \circ D^{-1} \circ S^{-1}\colon A(\reals^n) \to A(\reals^n).    
    \end{equation}
    From \eqref{E:I-formula}, as $E\circ D^{-1}$ is diagonal, $I$ is diagonalizable, and both properties \ref{Enum:infl-1} and \ref{Enum:infl-2} from the definition of $1$-inflation are satisfied.
\end{example}

\begin{proposition}\label{P:PR-inflation-linear}
    Let $A\colon \reals^n \to \reals^m$ be an affine map of the form $A=A_{\textnormal{lin}}+u$, where $A_{\textnormal{lin}}$ is a linear map and $u\in \reals^m$. Assume that $A_{\textnormal{lin}}$ is of full rank, $L_0=\lVert A_{\textnormal{lin}}\rVert_{a\to b}\leq 1$ and $A_{\textnormal{lin}}$ admits a $\lambda$-inflation for some $\lambda>0$.
    Let $E\subset\reals^n$ be a bounded $\Hn$-measurable set. Then, for every $\varepsilon>0$, there is $g\in \tLip_{L_0}(E_a,\reals^m_b)$ such that
    \begin{enumerate}
        \item\label{Enum:PR1-1} $\lVert g-A \rVert_{\infty}<\varepsilon$ and
        \item\label{Enum:PR1-2} $\vol g' \geq L_0\lambda$ $\Hn$-a.e.~in $E$.
    \end{enumerate}
\end{proposition}
\begin{proof}
By a standard scaling argument, we may assume $L_0=1$.
    We may assume, without loss of generality, that $A$ is linear. We first construct $g\colon \reals^n \to \reals^m$ and then restrict to $E$. By our assumption on $A$, there is a diagonalizable map 
    $I\colon A(\reals^n)\to A(\reals^n)$ whose eigenvectors $(u_1,\dots,u_n)=(A(x_1),\dots,A(x_n))$ form a basis of $A(\reals^n)$ and which satisfies the properties from Definition \ref{D:inflation}. 
    Note that $(x_1,\dots,x_n)$ form a basis of $\reals^n$. 
    
    Fix $i\in\{1,\dots, n\}$. Recalling Observation \ref{Obs:good-curve}, there is a Lipschitz curve $\gamma_i\colon \reals \to \tspan\{u_i\}$ such that
    \begin{equation*}
        |\gamma_i(t)-A(tx_i)|<\varepsilon\quad \text{for all $t\in \reals$,}
    \end{equation*}
    and
    \begin{equation}\label{E:gamma-diff}
        \gamma_i'(t)=\pm I(A(x_i))\quad \text{for all $t\in\reals$ up to a discrete set.}
    \end{equation}
    Denote by $t\colon \reals^n \to \reals^n$ the coordinate function with respect to the basis $(x_1,\dots, x_n)$, i.e.~we have $x=\sum_{i=1}^n t_i(x)x_i$ for each $x\in \reals^n$.
    We now simply let
    \begin{equation*}
        g(x) = \sum_{i=1}^n \gamma_i(t_i(x)),\quad \text{for $x\in \reals^n$.}
    \end{equation*}
    As each $\gamma_i$ is Lipschitz, $g$ is also Lipschitz.
    Fix an $i\in\{1, \dots, n\}$.
    Given an arbitrary $x\in\reals^n$ such that $\gamma_i'(t_i(x))$ exists, since $t$ is a linear map, we have $t'(x)=t$. Therefore, using \eqref{E:gamma-diff}, we have for any $\alpha_j \in \reals$, $j=1,\dots, n$
    \begin{equation*}
    \begin{split}
        (\gamma_i \circ t_i)'(x)(\alpha_1 x_1 + \dots + \alpha_n x_n)&=\gamma_i'(t_i(x))(t_i'(x))(\alpha_1 x_1 + \dots + \alpha_n x_n)\\
        &=\gamma_i'(t_i(x))(\alpha_i)= \pm \alpha_i I(A(x_i)).
    \end{split}
    \end{equation*}
    Therefore, by definition of $g$, we have
    \begin{equation*}
        g'(x)(\alpha_1 x_1+ \dots + \alpha_n x_n)= \sum_{j=1}^{n} \pm \alpha_j I(A(x_j)).
    \end{equation*}
    This means that for every $x\in \reals^n$, up to a set of $\Hn$-measure $0$, $g'(x)=\widetilde I A$ for some $\widetilde I \in \sp I$.
    Therefore, by \ref{Enum:infl-1} in Definition \ref{D:inflation}, $g$ is a Lipschitz function on the entire $\reals^n$ satisfying $\lVert g'(x) \rVert_{a\to b}\leq 1$ for every such $x\in \reals^n$.
    This by Lemma \ref{L:prel-Lip} implies that $g$ is $1$-Lipschitz between $\reals^n_a$ and $\reals^m_b$.
    The restriction $g_{|E}$ therefore satisfies $g_{|E}\in \tLip_1(E_a, \reals^m_b)$ and \ref{Enum:PR1-1}.
    By the property \ref{Enum:infl-2} in Definition \ref{D:inflation}, using once again the fact that $g'(x) = \widetilde I A$ for $\Hn$-a.e.~$x\in E$, \ref{Enum:PR1-2} is satisfied as well.
\end{proof}

\begin{proposition}\label{P:PR-inflation-cube}
    Let $E\subset \reals^n$ be a bounded $\Hn$-measurable set and suppose $E\subset D \subset \reals^k$. Let $f\colon D \to \reals^m$ and suppose $f\in \tLip_1(D_a,\reals^m_b)$. Assume that $\infl{a}{b}$ forms a $\lambda$-inflating pair for some $\lambda>0$. Then to each $\varepsilon>0$ and $\eta\in(0,1)$ there exists $g\colon D \to \reals^m$ such that
    \begin{enumerate}
        \item\label{Enum:PR2-1} $g\in \tLip_1(D_a,\reals^m_b)$,
        \item\label{Enum:PR2-2} $\lVert g-f \rVert_{\infty}<\varepsilon$,
        \item\label{Enum:PR2-3} $\int_E \vol g'\diff \Hn \geq \eta\lambda \Hn(E)$.
    \end{enumerate}
\end{proposition}
\begin{proof}
    If $\Hn(E)=0$, the statement obviously holds, so we may assume $\Hn(E)>0$.
    We may assume, without loss of generality, that the Lipschitz constant of $f$ is strictly less than $1$. We therefore have $L_0=\max\{\tLip_{a\to b}(f),\tfrac{1+\eta}{2}\}<1$.
    Find $C>0$ such that
    \begin{equation}\label{E:PR-IC-C-choice}
        \lambda L_0 \Hn(E)-\lambda L_0 2C \geq \eta\lambda \Hn(E).
    \end{equation}
    Find $\Delta>0$ such that for any sequence $B_i=B_a(x_i,r_i)\subset B_a(E,1)$, $x_i\in E$ of disjoint balls we have
    \begin{equation*}
        \sum_{i} r_i^n \leq \Delta.
    \end{equation*}
    Let $\sigma\in(0,1)$ be such that
    \begin{equation}\label{E:PR-IC-sigma-choice}
        \Delta(1-\sigma^n)<C.
    \end{equation}
    
     Find $\delta\in (0,\frac{1}{2}\varepsilon)$ such that $L_0+4\delta<1$. Then, for any $x\in E$ a density point of $E$ such that $f'(x)$ exists, we find $r_x\in(0,1]$ such that
    \begin{equation*}
        |f(y)-f(x)-f'(x)(y-x)|_b\leq \delta (1-\sigma)|y-x|_a\quad \text{for all $y\in B_a(x,r_x)\cap D$.}
    \end{equation*}
    In a standard way, using Vitali covering and continuity of measure, we thereby obtain a finite sequence of disjoint balls $B_i=B_a(x_i,r_i)\cap D$, $x_i\in E$, $i=1,\dots,i_0$ in $D_a$ such that
    \begin{equation*}
        \Hn(E\setminus\bigcup_{i=1}^{i_0}B_i)<C
    \end{equation*}
    and for each $i\in \{1,\dots, i_0\}$ we have an affine map $A_i\colon \reals^n \to \reals^m$ with
    \begin{equation*}
        |f(y)-A_i(y)|_b< \delta(1-\sigma) r_i, \quad\text{for all $y\in B_i$.}
    \end{equation*}
    As $f$ is $L_0$-Lipschitz, we may also assume that the linear part of each $A_i$ lies in $L_0 B_{a\to b}$.
    From the density of maps of full rank in $B_{a\to b}$, we may assume that each $A_i$ is of full rank. Recalling Proposition \ref{P:PR-inflation-linear}, we find $g_i\colon B_i \to \reals^n$ such that $g_i\in \tLip_{L_0}((B_i)_a,\reals^m_b)$ and
    \begin{equation*}
        \vol g' \geq L_0\lambda \quad \text{$\Hn$-a.e.~in $B_i$,}
    \end{equation*}
    and 
    \begin{equation*}
        \lVert g_i - A_i \rVert_{\ell^\infty(B_i)}\leq \delta(1-\sigma) r_i.
    \end{equation*}
    Using the extension Lemma \ref{L:Lipschitz-extension-lemma} (for $X=D$), there is a function $g\colon D \to \reals^m$ such that
    \begin{enumerate}[label={(\alph*)}]
        \item\label{Enum:PR2-a} $\tLip_{a\to b} (g)\leq L_0+4\delta$,
        \item\label{Enum:PR2-b} $g=g_i$ on $\sigma B_i$ for each $i\in\{1,\dots, n\}$,
        \item\label{Enum:PR2-c} $\lVert g-f \rVert_{l^{\infty}(D)}\leq 2\delta(1-\sigma)\leq \varepsilon$.
    \end{enumerate}
    By our choice of $\delta$, we have $g\in\tLip_1(D_a,\reals^m_b)$ which is \ref{Enum:PR2-1} and from \ref{Enum:PR2-c} we we get \ref{Enum:PR2-2}. It remains to show \ref{Enum:PR2-3}.
    
    Firstly, by \eqref{E:PR-IC-sigma-choice}, we have
    \begin{equation*}
        \Hn(E\setminus \bigcup_{i=1}^{i_0}\sigma B_i)\leq \Hn(E\setminus \bigcup_{i=1}^{i_0} B_i)+\Hn(\bigcup_{i=1}^{i_0} B_i\setminus \sigma B_i)
        \leq C+ \Delta(1-\sigma^n)<2C.
    \end{equation*}
    Moreover, by \ref{Enum:PR2-b}, $\vol g'\geq L_0\lambda$ in each $B_i$ and using also \eqref{E:PR-IC-C-choice} it follows that
    \begin{equation*}
        \begin{split}
            \int_E \vol g' \diff \Hn &\geq \sum_{i=1}^{i_0} \int_{\sigma B_i\cap E} \vol g' \diff \Hn
            \geq \lambda L_0 \sum_{i=1}^{i_0} \Hn(\sigma B_i\cap E)\\
            &\geq \lambda L_0 \left(\Hn(E)-\Hn(E\setminus \bigcup_{i=1}^{i_0} \sigma B_i)\right)
            \geq \lambda L_0 \Hn(E)-\lambda L_0 2C\\
            &\geq \eta \lambda \Hn(E).
        \end{split}
    \end{equation*}
\end{proof}

We are now ready to prove the main result of this section.

\begin{theorem}\label{T:positive-result-cubes}
Suppose $n\leq m$ and let $|\cdot |_a$ be a norm on $\reals^n$ and $|\cdot |_b$ a norm on $\reals^m$. Let $E\subset \reals^n$ be bounded and $\Hn$-measurable and let $E\subset D \subset \reals^n$. Let $\lambda>0$ and assume $\infl{a}{b}$ forms a $\lambda$-inflating pair. The set
\begin{equation*}
    \{f\in \tLip(D_a,\reals^m_b): \int_E \vol f' \diff \Hn \geq \lambda \Hn(E)\}
\end{equation*}
is residual in $\tLip(D_a,\reals^m_b)$.
Suppose $n<m$. Then the set
\begin{equation*}
    \{f\in \tLip(D_a,\reals^m_b): \Hn(f(E)) \geq \lambda \Hn(E)\}
\end{equation*}
is residual in $\tLip(D_a,\reals^m_b)$.
\end{theorem}
\begin{proof}
    Firstly, we realize that by Haar's theorem, $\H^n$ is a constant multiple of $\Hn_a$.
    Therefore, we may use the general result of Theorem \ref{T:char-area-area-formula} and it suffices to show that
    \begin{equation*}
        \{f\in \tLip(D_a,\reals^m_b): \int_E \vol f' \diff \Hn > \widetilde{\lambda} \Hn(E)\}
    \end{equation*}
    is dense for every $\widetilde\lambda < \lambda$ provided $n\leq m$. This follows from Proposition \ref{P:PR-inflation-cube}.
    
\end{proof}

By virtue of Haar's theorem, we can obtain the relevant result with ``correct'' Hausdorff measure on the domain side.

\begin{corollary}
    Suppose $n\leq m$, let $|\cdot |_a$ be a norm on $\reals^n$ and let $|\cdot|_b$ be a norm on $\reals^m$. Let $E\subset \reals^n$ be bounded and $\Hn$-measurable and let $E\subset D \subset \reals^n$. Let $\lambda>0$ and assume $\infl{a}{b}$ forms a $(\vol(|\cdot|_a)\lambda)$-inflating pair.
    The set
\begin{equation*}
    \{f\in \tLip_1 (D_a,\reals_b^m): \int_E \vol f' \diff \Hn \geq \lambda \Hn_a(E)\}
\end{equation*}
is residual in $\tLip_1 (D_a,\reals_b^m)$.
Suppose $n<m$. Then the set
\begin{equation*}
    \{f\in \tLip_1 (D_a,\reals^m_b): \Hn(f(E)) \geq\lambda \Hn_a(E)\}
\end{equation*}
is residual in $\tLip_1 (D_a,\reals^m_b)$.
\end{corollary}
\begin{proof}
    As $\Hn_a$ is a Haar measure on $\reals^n$ and $\Hn_a(B_a)=2^n$ we have $\Hn_a = \vol |\cdot|_a \Hn$ by Haar's theorem. The rest follows from Theorem \ref{T:positive-result-cubes}.
\end{proof}

Recalling Example \ref{Ex:Euclidean-inflation}, we immediately obtain the following Euclidean result.

\begin{corollary}\label{C:Euclidean-strong}
    Suppose $m\leq n$. Let $E\subset \reals^n$ be bounded and $\Hn$-measurable and let $E\subset D \subset \reals^n$. The set
    \begin{equation*}
        \{f\in \tLip_1 (D,\reals^m): \int_E \vol f' \diff \Hn = \Hn(E)\}
    \end{equation*}
    is residual in $\tLip_1 (D,\reals^m)$.
    Suppose $n<m$. Then the set
    \begin{equation*}
        \{f\in \tLip_1 (D,\reals^m): \Hn(f(E)) = \Hn(E)\}
    \end{equation*}
    is residual in $\tLip_1 (D,\reals^m)$.
\end{corollary}

\subsection{Strongest possible results in Euclidean spaces.}

While the results of the last subsection were ``local'' in the sense that they required the $n$-rectifiable set to be in fact a normed piece of $\reals^n$, in the following, we push some of the results to general \emph{Euclidean} $n$-rectifiable sets, obtaining a proof of Theorem \ref{T:main-Euclidean}.

\begin{theorem}\label{T:Eucl-strongest-possible}
    Let $n\leq k$, $n\leq m$ and suppose $E\subset \reals^k$ is $n$-rectifiable and satisfying $\Hn(E)<\infty$. Let $E\subset D \subset\reals^k$. Let $f\in \tLip_1(D,\reals^m)$. Then for every $\varepsilon>0$ and $\eta\in(0,1)$, there is a $g\in \tLip_1(D,\reals^m)$ such that
    \begin{enumerate}
        \item\label{Enum:SPR-1} $\lVert g-f \rVert_{\infty}\leq \varepsilon$,
        \item\label{Enum:SPR-2} $\int_E J_E g\diff \Hn>\eta \Hn(E)$.
    \end{enumerate}
\end{theorem}
\begin{proof}
    If $\Hn(E)=0$, the statement obviously holds, so we may assume $\Hn(E)>0$.
    Without loss of generality, $L_0=\tLip(f)<1$. Find $\theta>0$ such that $(1+\theta)^2 L_0<1$. Find $\eta_0\in (0,1)$ and $C>0$ such that
    \begin{equation}\label{E:SPRE-estimate1}
        \frac{1}{(1+\theta)^2}\eta_0(\Hn(E)-C)> \eta \Hn(E).
    \end{equation}
    Using \cite[Lemma 3.2.2]{Federer} (which is the Euclidean version of Lemma \ref{L:Kirch}), we find countably many $E_i\subset E$ Borel and disjoint, $F_i\subset \reals^n$ and $(1+\theta)$-biLipschitz maps $I_i\colon F_i \to E_i$ such that
    \begin{equation*}
        \Hn(E\setminus \bigcup_i E_i)=0.
    \end{equation*}
    Using continuity and inner regularity of $\Hn$, there is some $i_0\in\nat$ such that
    \begin{equation}\label{E:SPRE-measure-estimate}
        \Hn(E\setminus \bigcup_{i=1}^{i_0}E_i)<C
    \end{equation}
    and we may assume that $E_i$'s are compact. There is some $r\in (0, \frac{1}{3}\varepsilon)$ such that for every $i,j\in\{1,\dots, i_0\}$ with $i\not=j$ we have
    \begin{equation}\label{E:SPRE-dist-estimate}
        \dist(E_i,E_j)\geq r.
    \end{equation}
    Let $\varepsilon_0\in (0,\frac{1}{3}\varepsilon)$ be such that
    \begin{equation}\label{E:SPRE-eps0-choice}
        \frac{\varepsilon_0}{r}+L_0<1.
    \end{equation}
    
    Fix $i=1,\dots, i_0$ and let $\varphi_i=f\circ I_i\colon F_i \to \reals^m$. Then $\varphi_i$ is $L_0(1+\theta)$-Lipschitz. Whence, by Corollary \ref{C:Euclidean-strong} (we are using only density), there is $\psi_i\colon F_i \to \reals^m$ such that
    \begin{enumerate}[label={(\alph*)}]
        \item\label{Enum:SPR-a} $\psi_i\in \tLip_{L_0(1+\theta)}(F_i,\reals^m)$,
        \item\label{Enum:SPR-b} $\lVert \psi_i - \varphi_i \rVert_{\ell^{\infty}(F_i)}\leq \varepsilon_0$,
        \item\label{Enum:SPR-c} $\int_{F_i} \vol \psi_i' \diff \Hn \geq \eta_0\Hn(F_i)$.
    \end{enumerate}
    
    We let $g_i\colon E_i \to \reals^m$ be given by $g_i=\psi_i\circ I_i^{-1}$. We define $S= \bigcup_{i=1}^{i_0} E_i$ and let $g\colon S \to \reals^m$ be given by
    \begin{equation*}
        g(x)=g_i(x)\quad \text{for $x\in E_i$, $i=1,\dots, i_0$}.    
    \end{equation*}
    By \ref{Enum:SPR-a}, \ref{Enum:SPR-b}, \eqref{E:SPRE-dist-estimate} and \eqref{E:SPRE-eps0-choice}, $g$ is $1$-Lipschitz. By \ref{Enum:SPR-b}, we have
    \begin{equation}\label{E:SPRE-estimate-on-S}
        \lVert g-f \rVert_{\ell^\infty(S)}\leq \varepsilon_0
    \end{equation}
     and by \ref{Enum:SPR-c}, \eqref{E:SPRE-measure-estimate} and \eqref{E:SPRE-estimate1}, we have
     \begin{equation}\label{E:SPRE-area-estimate}
         \int_S J_E g \diff \Hn
         \geq \frac{1}{(1+\theta)^2}\eta_0 \Hn(S)
         \geq \frac{1}{(1+\theta)^2}\eta_0 (\Hn(E)-C)>\eta \Hn(E).
     \end{equation}
     It remains to extend $g$. To that end, let
     \begin{equation*}
         c(x)=\begin{cases}\displaystyle
				g(x)
					& \text{if $x\in S$},
					\\
				f(x)
					& \text{if $x\in D\setminus B(S,r)$}.
			\end{cases}
     \end{equation*}
     From \eqref{E:SPRE-eps0-choice} and \eqref{E:SPRE-estimate-on-S} it follows that $c$ is $1$-Lipschitz. Using Kirszbraun's extension theorem, we find a $1$-Lipschitz extension of $c$ onto $D$. This extension also extends $g$ on $S$ and we will denote it by $g$. By \eqref{E:SPRE-area-estimate}, we have \ref{Enum:SPR-1} and so it remains to show \ref{Enum:SPR-2}.
     
     Let $x\in B(S,r)\setminus S$ and find $y\in S$ with $|x-y|=r$. Then
     \begin{equation*}
     \begin{split}
         |g(x)-f(x)|&\leq |g(x)-g(y)|+|g(y)-f(y)|+|f(y)-f(x)|\\ 
         &\leq |x-y|+\varepsilon_0+|x-y|\leq r+\varepsilon_0+r<\varepsilon.
     \end{split}
     \end{equation*}
     This shows that \ref{Enum:SPR-2} holds and we are done.
\end{proof}

\begin{theorem}[Restatement of Theorem \ref{T:main-Euclidean}]\label{T:main-Euclidean-rest}
    Let $n\leq k$, $n\leq m$ and suppose $E\subset \reals^k$ is $n$-rectifiable. Let $E\subset D \subset\reals^k$.
    Then the set
    \begin{equation*}
        \{f\in\tLip_1(D, \reals^m): \int_E J_E f \diff \Hn = \Hn(E)\}
    \end{equation*}
    is residual in $\tLip_1(D, \reals^m)$.
    Moreover, if $n<m$, then the set
    \begin{equation*}
        \{f\in\tLip_1(D, \reals^m): \Hn(f(E)) = \Hn(E)\}
    \end{equation*}
    is residual in $\tLip_1(D, \reals^m)$.
\end{theorem}
\begin{proof}
    We may reduce to the case $\Hn(E)<\infty$ in the standard way as $\Hn_{|E}$ is $\sigma$-finite.
     Recalling Theorem \ref{T:char-area-area-formula},
     it suffices to show that the set
     \begin{equation*}
     \begin{split}
         \{f\in\tLip_1(D_a, \reals^m): \int_E J_E f \diff \Hn > \eta \Hn(E)\}.
     \end{split}
     \end{equation*}
     is dense for every $\eta\in (0,1)$.
     This follows from Theorem  \ref{T:Eucl-strongest-possible}.
\end{proof}

The result of Theorem \ref{T:main-Euclidean-rest} is, as many other residuality results, strange in the sense that constructing any specific examples of Lipschitz maps that satisfy the required properties is highly non-trivial. Even considering $E=\mathbb{S}^1$, the unit circle embedded in $\reals^2$, it is not clear at all how to construct a $1$-Lipschitz map into $\reals$ having the tangential Jacobian equal to $\pm 1$ $\H^1$-a.e.
If we were allowed to have $\tLip(f)=1+\varepsilon$, this would be easy (one may, for example, consider a parametrization of $E$ of speed $1$ and locally invert it on small intervals), but there is no natural way of sending $\varepsilon \to 0$.
If, in this case, we consider maps into $\reals^2$ it is once again difficult to construct any $f\colon \reals^2 \to \reals^2$ which is $1$-Lipschitz and satisfies $\H^1(f(E))=\H^1(E)$ which is \emph{not} a linear isometry. 

\section{Negative results in normed sets}\label{S:residuality-negative}

While in the Euclidean space very strong results hold, this fails in more general spaces.
In fact it is enough to consider different (finite dimensional) normed spaces for some of the results to fail completely.
The purpose of this section is to provide conditions on norms $|\cdot|_a$ and $|\cdot|_b$ so that sets of the form
\begin{equation*}
    \{f\in \tLip(\Omega_a, \reals^m_b): \Hn(f(\Omega))\geq \numberdelta\}
\end{equation*}
are \emph{not} dense in $\tLip(\Omega_a, \reals^m_b)$. Here the most general instance of $\Omega$ is a bounded open set - further generality is possible but not of much interest to us. The sets being open makes several technical steps significantly easier.

Recall the definition of a strongly extremal point from Definition \ref{extremal-point}.
We begin with a simple observation about strongly extremal points. Recall that a linear map $P\colon \reals^n \to \reals^n$ is called a \emph{linear projection} if $P\circ P = P$.

\begin{proposition}\label{P:str-ext-char}
    Suppose $|\cdot|_a$ is a norm on $\reals^n$. If $u\in \partial B_a$ is strongly extremal, then it is extremal. Moreover, for $u\in \partial B_a$, the following statements are equivalent:
    \begin{enumerate}
        \item\label{Enum:SE-1} $u$ is a strongly extremal point of $B_a$,
        \item\label{Enum:SE-2} there exists a linear projection $P\colon \reals^n \to \operatorname{span}\{u\}$ such that $P^{-1}(u)\cap B_a = \{u\}$,
        \item\label{Enum:SE-3} there exists a linear projection $P\colon \reals^n \to \operatorname{span}\{u\}$ such that if $u_n\in B_a$ is a sequence with $P(u_n)\to u$, then $u_n \to u$.
    \end{enumerate}
\end{proposition}
\begin{proof}
    If \ref{Enum:SE-1} holds, then there is an affine tangent $T$ to $B_a$ at $u$ with $T\cap B_a=\{u\}$. Taking $P\colon \reals^n \to \tspan{u}$ to be the linear map satisfying $P^{-1}(u)= T$, we see that both \ref{Enum:SE-2} and \ref{Enum:SE-3} hold. On the other hand, if \ref{Enum:SE-2} or \ref{Enum:SE-3} hold, we can take $T=P^{-1}(u)$ and in either case, we get $T\cap B_a =\{u\}$, which means that \ref{Enum:SE-1} holds.
\end{proof}

Given $u\in \reals^m$ and $V=(v_2,\dots, v_n)\in (\reals^m)^{n-1}$, we use the notation $(u|V)$ for the linear map from $\reals^n \to \reals^m$ satisfying $(u|V)(e_1)=u$ and $(u|V)(e_j)=v_j$ for $j\in\{2,\dots, n\}$.

\begin{definition}
    Suppose $n,m\in\nat$, $2\leq n\leq m$.
    Let $|\cdot |_a$ be a norm on $\reals^n$ and $|\cdot|_b$ a norm on $\reals^m$.
    For any $u\in \reals^m$ satisfying $\lVert (u|0) \rVert_{a\to b}\leq 1$, we define the \emph{maximal volume} of $u$ as the quantity $\mv u= \mv_{a\to b} u\in [0,\infty)$ given by
    \begin{equation*}
        \mv u = \sup \{\vol (u| V): V\in (\reals^m)^{n-1},\, \lVert (u| V) \rVert_{a\to b}\leq 1\}.
    \end{equation*}
\end{definition}

The following observation, stating essentially that $\mv$ is upper semi-continuous, is an immediate consequence the fact that $\vol$ is $\frac{1}{2}$-H\"older (it is even Lipschitz, but that does require a proof) and that $\lVert \cdot \rVert_{a\to b}$ is Lipschitz.

\begin{observation}\label{O:mv-semicont}
    Suppose $n,m\in\nat$, $2\leq n\leq m$.
    Let $|\cdot |_a$ be a norm on $\reals^n$ and $|\cdot|_b$ a norm on $\reals^m$.
    Let $u \in \reals^m$ be such that $\lVert (u|0) \rVert_{a\to b}\leq 1$.
    Then, to each $\delta>0$, there is $\varepsilon>0$ such that for every $\widetilde u \in B_a(u, \varepsilon)$ the following holds.
    If $V\in (\reals^m)^{n-1}$ is such that $\lVert(\widetilde u| V)\rVert_{a\to b}\leq 1$, then
    \begin{equation*}
        \vol (\widetilde{u}|V)\leq \mv u + \delta.
    \end{equation*}
\end{observation}

We shall also require a particular property of integral averages, which is the subject of the following lemma. As this holds in an arbitrary finite measure space, we state it in full generality.

\begin{lemma}\label{L:balls}
Let $(R,\mu)$ be a measure space with $0<\mu(R)<\infty$ and let $\psi\colon R \to \reals$ be $\Hn$-measurable. Let $K>0$, $\delta>0$ and $N\in\nat$ be given. Then there is $\varepsilon>0$ such that if
\begin{equation}\label{E:con-L1-to-L1infty}
    \begin{split}
        &\psi\leq K \quad \text{a.e. in $R$}\quad \text{and}\\
        &K(1-\varepsilon)\leq \frac{1}{\mu(R)}\int_{R}\psi\diff\mu,
    \end{split}
\end{equation}
then $\mu(\{\psi\geq K - \delta\})\geq \mu(R)(1-\tfrac{1}{N})$.
\end{lemma}
\begin{proof}
    Denote $\lambda=\tfrac{1}{\mu(R)}\mu(\{\psi\geq K - \delta\})$. Then for every $\varepsilon>0$, assuming \eqref{E:con-L1-to-L1infty} holds, one has
    \begin{equation*}
        \begin{split}
            K(1-\varepsilon)&\leq \frac{1}{\mu(R)}\int_{R}\psi\diff\mu
            =\frac{1}{\mu(R)}\left(\int_{\{\psi\geq K-\delta\}}\psi\diff\mu + \int_{\{\psi< K-\delta\}}\psi\diff\mu\right)\\
            &\leq \frac{1}{\mu(R)}(\mu(R)\lambda K + (1-\lambda)\mu(R)(K-\delta))
            =\lambda K +(1-\lambda)(K-\delta).
        \end{split}
    \end{equation*}
    The above inequality is equivalent to
    \begin{equation}\label{E:1}
        \frac{-K\varepsilon+\delta}{\delta}\leq \lambda.
    \end{equation}
    We may find an $\varepsilon>0$ so that the left hand side of \eqref{E:1} is greater than or equal to $1-\tfrac{1}{N}$. This, however, implies that $\lambda\geq 1-\tfrac{1}{N}$ and the statement follows.
\end{proof}

\begin{theorem}\label{T:mv-negative}
    Suppose $n,m\in\nat$, $2\leq n\leq m$ and denote $Q=[-1,1]^n$.
    Let $|\cdot |_a$ be a norm on $\reals^n$ and $|\cdot|_b$ a norm on $\reals^m$. Let $u\in \reals^m$ be a strongly extremal point of $B_b$ such that $\lVert (u|0)\rVert_{a\to b}=1$.
    Then for any $r>0$ and any sequence $g_i\in \tLip_1(Q_a,\reals^m_b)$ such that $g_i \to (u|0)$ it holds that
    \begin{equation*}
        \lim_{i\to\infty} \Hn(\{x\in Q: \vol g_i' (x)\geq\mv u +r\})=0.
    \end{equation*}
\end{theorem}
\begin{proof}
    We claim that for every $N\in\nat$ and $\sigma>0$
    there exists $\varepsilon>0$ such that if $g\in\tLip_1(Q_a,\reals^m_b)$ satisfies $\lVert g-(u|0) \rVert_{\ell^\infty(Q_a,\reals^m_b)}\leq \varepsilon$, then
\begin{equation}\label{E:non-zero-entries-estimate}
        \lVert \tfrac{\partial g}{\partial e_1}-u \rVert_b\leq \sigma\quad\text{on a Borel set $M$ of $\mathcal{H}^n$-measure at least $2^n-\frac{2^n}{N}$.}
    \end{equation}

    Suppose for a moment that the claim holds true. Let $N\in\nat$, by taking $\sigma>0$ small enough and using Observation \ref{O:mv-semicont} together with the fact that $\lVert g_i \rVert_{a\to b}\leq 1$ $\Hn$-a.e., there exists some $i_N\in\nat$ (depending only on $N$) such that
    \begin{equation*}
        \vol g_i' \leq \mv u +\tfrac{r}{2} \quad \text{$\Hn$-a.e. in $M$, for $i\geq i_N$.}
    \end{equation*}
    It follows that
    \begin{equation*}
        \Hn(\{x\in Q: \vol g_i'(x)\geq \mv u + r\})\leq \Hn(Q\setminus M)\leq \frac{2^n}{N}\quad\text{for $i\geq i_N$}
    \end{equation*}
    and so the statement of the theorem follows by sending $N\to \infty$.

It remains to show the claim.
As $u$ is strongly extremal, by Proposition \ref{P:str-ext-char}, we find a linear projection $P\colon \reals^m \to \tspan\{u\}$ with the following property. Whenever $w^\alpha\in B_b$ satisfy $P(w^\alpha)\to u$ as $\alpha \to \infty$, we have $w^\alpha \to u$. Hence, we may find $\delta>0$ such that for every $w\in B_b$,
    \begin{equation}\label{E:nr-delta}
        |P(w)-u|_b\leq \delta \quad\text{implies}\quad |w-u|\leq\sigma.
    \end{equation}
    Finally, fixing $K=1$, $R=[-1,1]$ and $\mu=\mathcal{L}^1$, find $\varepsilon>0$ from Lemma \ref{L:balls}. 
    
Fix now $t_2,\dots, t_n\in [-1,1]$ and consider the Lipschitz curve
\begin{equation*}
    \varphi(t)=g(t,t_2,\dots,t_n).
\end{equation*}
Since $\lVert(u|0)\rVert_{a\to b}\leq 1$ and $u\in \partial B_b$, we infer that $(1,0,\dots,0)^T\in \partial B_a$, which means that the restriction of $|\cdot|_a$ to $\tspan\{(1,0,\dots,0)^T\}$ is the Euclidean distance.
Whence, as $\varphi$ is $1$-Lipschitz, we have
\begin{equation*}
    |\varphi'(t)|_b\leq 1 \quad\text{for a.e.~$t\in[-1,1]$.}
\end{equation*}

Applying the fundamental theorem of calculus, we obtain
    \begin{equation*}
        \int_{-1}^1\varphi'(s)\diff s=\varphi(1)-\varphi(-1)\in B_b(2u,2\varepsilon),
    \end{equation*}
    i.e.
    \begin{equation*}
        |\int_{-1}^1\varphi'(s)\diff s - 2u|_b\leq 2\varepsilon,
    \end{equation*}
    hence, recalling that $P$ has operator norm $1$,
    \begin{equation*}
        |\int_{-1}^1 P(\varphi'(s))\diff s - u|_b\leq 2\varepsilon.
    \end{equation*}
    If we now identify $\tspan\{u^i\}$ with $\reals$ by assigning $\lambda\in\reals$ to $\lambda u_i$, we may use the reverse triangle inequality and obtain
    \begin{equation*}
        \int_{-1}^1 P(\varphi'(s))\diff s \geq 2-2\varepsilon,
    \end{equation*}
    in the sense of the described identification.
    By the choice of $\varepsilon$, we find a Borel set $M(t_2,\dots,t_n)\subset [-1,1]$ with
    \begin{equation*}
        \H^1 (M(t_2,\dots,t_n)) \geq 2-\tfrac{2}{N}
    \end{equation*}
    such that
    \begin{equation*}
        |P(\varphi')-u|_b\leq\delta \quad \text{on $M(t_2,\dots, t_n)$}.
    \end{equation*}
    By definition, whenever $\varphi'$ exists, one has $\varphi'=\tfrac{\partial g}{\partial e_1}$. Whence, the \emph{Borel} set
    \begin{equation*}
        M=\{x\in Q: |P(\tfrac{\partial g}{\partial e_1}(x))-u|_b\leq \delta\}
    \end{equation*}
    has the following property. For any choice of $t_2,\dots t_n\in [-1,1]$ the one-dimensional projection of $M$ satisfies
    \begin{equation*}
        \{x\in M: x_2=t_2,\dots,x_n=t_n\}\supset M(t_2,\dots, t_n).
    \end{equation*}
    Whence Fubini's theorem gives
    \begin{equation*}
        \Hn(M)\geq 2^n-\tfrac{2^n}{N}.
    \end{equation*}
    By the definition of $M$ and the choice of $\delta$ \eqref{E:nr-delta}, we have
    \begin{equation*}
        |\tfrac{\partial g}{\partial e_1}-u|_b\leq \sigma\quad \text{on $M$},
    \end{equation*}
    which is \eqref{E:non-zero-entries-estimate} as we wanted.
\end{proof}

\begin{corollary}\label{C:mv-negative}
    Suppose $n,m\in\nat$, $2\leq n\leq m$. Let $\Omega\subset \reals^n$ be bounded and open.
    Let $|\cdot |_a$ be a norm on $\reals^n$ and $|\cdot|_b$ a norm on $\reals^m$. Let $u\in \reals^m$ be a strongly extremal point of $B_b$ such that $\lVert (u|0)\rVert_{a\to b}=1$.
    Then for any $r>0$ and any sequence $g_i\in \tLip_1(\Omega_a,\reals^m_b)$ such that $g_i \to (u|0)$ it holds that
    \begin{equation*}
        \lim_{i\to\infty} \Hn(\{x\in \Omega: \vol g_i' (x)\geq \mv u+r\})=0.
    \end{equation*}
    In particular, if $\mv u = 0$, then the sets
    \begin{equation*}
        \{f\in \tLip_1(\Omega_a,\reals^m_b): \int_\Omega \vol f' \diff\Hn \geq \numberdelta\}
    \end{equation*}
    and
    \begin{equation*}
        \{f\in \tLip_1(\Omega_a,\reals^m_b): \Hn(f(\Omega))\geq \numberdelta\}
    \end{equation*}
    are dense in $\tLip_1(\Omega_a,\reals^m_b)$ if and only if $\numberdelta=0$.
\end{corollary}
\begin{proof}
Any open bounded set $\Omega \subset \reals^n$ may be arbitrarily well (with respect to measure) filled with a finite set of non-overlapping squares, thus the first statement easily follows from Theorem \ref{T:mv-negative}.
    The non-density of the first set then follows immediately from the first statement. The statement about the non-density of the second set follows from the area formula and the non-density of the first set.
\end{proof}

\begin{example}
    Let $n,m\in\nat$, $m\geq n$ and denote $u=(1,0,\dots,0)^T\in\reals^m$.
    It is an easy observation that $\mv_{\infty \to 2} u = 0$.
    Indeed, one may even show that for $V\in (\reals^m)^{n-1}$ one has $\lVert(u|V) \rVert_{\infty \to 2}\leq 1$ if and only if $V=0$.
    This in particular shows that for any open bounded set $\Omega\subset \reals^m$, the set
    \begin{equation*}
        \{f\in \tLip_1(\Omega_\infty,\reals^m_2): \Hn(f(\Omega))\geq \numberdelta\}
    \end{equation*}
    is dense in $\tLip_1(\Omega_\infty,\reals^m_2)$ if and only if $\numberdelta=0$.
\end{example}

In fact, the idea of the example above can be easily used to show a far stronger statement.

\begin{theorem}[Restatement of Theorem \ref{T:main-general-failure}]
Let $n,m\in\nat$, $m\geq n$. Suppose $|\cdot|_a$ is a norm on $\reals^n$ such that $\partial B_a$ contains a non-extremal point of $B_a$. Suppose further that $|\cdot|_b$ is an arbitrary norm on $\reals^m$. Then, for any open bounded set $\Omega\subset \reals^n$, the sets
\begin{equation*}
    \{f\in \tLip_1(\Omega_a, \reals^m_b): \int_\Omega \vol f' \diff \Hn \geq \numberdelta\}
\end{equation*}
and
\begin{equation*}
    \{f\in \tLip_1(\Omega_a, \reals^m_b): \Hn(f(\Omega)) \geq \numberdelta\}
\end{equation*}
are dense in $\tLip_1(\Omega_a, \reals^m_b)$ if and only if $\numberdelta=0$.
\end{theorem}
\begin{proof}
    Let us denote by $e_1, \dots, e_n$ the canonical vectors in $\reals^n$.
    By our assumptions, there is a point $x\in \partial B_a$, which is non-extremal in $B_a$. There is a linear invertible map $A\colon \reals^n \to \reals^n$ such that $A(x)=e_1$ and $e_1+\tspan\{e_2,\dots, e_n\}$ is an affine tangent to $A(B_a)$. As $A\colon (\reals^n, |\cdot|_a) \to (\reals^n, |\cdot|_{A(a)})$ is an isometry, and the statement we are proving is invariant under isometries, we may assume that $x=e_1$ and $e_1+\tspan\{e_2,\dots, e_n\}$ is an affine tangent to $B_a$ at $x=e_1$. 
    
     As $\partial B_b$ is compact, there exists $u\in\partial B_b$ maximizing the quantity $|u|_2$.
     By taking the unique supporting hyperplane to the Euclidean ball of radius $|u|_2$ at $u$, one can easily observe that $u$ is a strongly extremal point of $B_b$. Therefore, in particular, it is also extremal.
     
     Clearly $\lVert (u|0) \rVert_{a\to b}=1$ as $(u|0)(B_a)=\{tu: t\in[-1,1]\}$ and so $\mv u$ is well defined.
     It is enough to show that $\mv u = 0$ and recall Corollary \ref{C:mv-negative}. 
     
     Suppose that $V\in (\reals^m)^{n-1}$ is such that $\vol (u|V)>0$. It suffices to prove that $\lVert (u|V) \rVert_{a\to b}>1$.
     To that end, let $l\subset B_a$ be any non-degenerate line segment having $e_1$ as its midpoint. As $(u|V)$ is of full rank, $(u|V) l$ is a non-degenerate line having $u$ as its midpoint.
     As $u$ is extremal in $B_b$, this implies that $(u|V)l\not\subset B_b$. Which, as $l\subset B_a$, implies 
     $\lVert (u|V) \rVert_{a\to b}>1$ and we are done.
\end{proof}

\section{Results in metric spaces}\label{S:results-metric}

The goal of this final section is to present results in general metric spaces. Of course, the ``generality'' here is fairly limited by the fact that even in case of normed spaces, the relevant results simply need not be true. Therefore, our positive results concentrate on $n$-rectifiable metric spaces whose $\Hn$-a.a.~approximate tangents are $\lambda$-inflating.

\begin{definition}\label{D:inflating-space}
    Given $n\in\nat$, we denote by $\mathcal{N}(n)$ the set of all norms on $\reals^n$ and by $\sim$ we understand the equivalence relation on $\mathcal{N}(n)$ given by 
        $|\cdot|_{a_1}\sim |\cdot|_{a_2}$ if and only if there is an invertible linear map $A\colon \reals^n \to \reals^n$ such that $A(B_{a_1})=B_{a_2}$.
    The space $\mathcal{N}(n)/_\sim$ is called the \emph{Banach-Mazur compactum}.
    Given $\lambda>0$, $n\leq m\in\nat$ and a norm $|\cdot|_b$ on $\reals^m$, we let
    \begin{equation*}
        \mathcal{N}^{b}_{\textnormal{infl}(\lambda)}(n)=\{[|\cdot|_{a}]\in\mathcal{N}(n)/_\sim: \text{$\infl{a}{b}$ forms a $(\vol(|\cdot|_a)\lambda)$-inflating pair}\}.
    \end{equation*}
\end{definition}

\begin{remark}
    The notion of forming $(\vol(|\cdot|_a)\lambda)$-inflating pair descends to quotient, i.e.~one has that $\infl{a}{b}$ forms a  $(\vol(|\cdot|_a)\lambda)$-inflating pair if and only if for every $|\cdot|_{a'}\in [|\cdot|_a]$, $(|\cdot|_{a'},|\cdot|_b)$ forms a $(\vol(|\cdot|_{a'})\lambda)$-inflating pair. This means that the particular representative chosen when dealing with the family $ \mathcal{N}^{b}_{\textnormal{infl}(\lambda)}(n)$ is irrelevant.
\end{remark}

\begin{theorem}\label{T:MSR-density}
    Suppose that $n,m\in \nat$, $n\leq m$, $X$ is a complete metric space and $E\subset X$ an $n$-rectifiable subset with $\Hn(E)<\infty$. Suppose $|\cdot|_b$ is a norm on $\reals^m$.
    Let $\lambda>0$ and assume that for $\Hn$-a.e.~$\xi\in E$, one has
    \begin{equation*}
        T(E,\xi)\in \mathcal{N}^b_{\textnormal{infl}(\lambda)}(n).
    \end{equation*}
    Then for each $\varepsilon>0$, there is a set $\widetilde{E}\subset E$ with $\Hn(E\setminus \widetilde{E})<\varepsilon$ which, moreover, satisfies the following.
    For every $\widetilde{E}\subset X' \subset X$, $f\in \tLip_1(X',\reals^m_b)$, $\delta>0$ and $\eta \in (0,1)$, there exists $g\in \tLip_1(\widetilde{E}, \reals^m_b)$ with $\lVert g-f \rVert_{\ell^{\infty}(\widetilde{E})}\leq \delta$ and
    \begin{equation}\label{E:MSR-measure-est-main}
        \int_{\widetilde E} J_{\widetilde E} g \diff \Hn\geq \eta\lambda\Hn_X(\widetilde{E}).
    \end{equation}
\end{theorem}

\begin{proof}
    If $\Hn(E)=0$, the statement obviously holds, so we may assume $\Hn(E)>0$.
    There is some $K_b\in (0,\infty)$ such that for every $1$-Lipschitz function $g$ defined on any $\Hn$-measurable subset of $E$, one has $J_E g \leq K_b$ on the set where the left hand side is well defined.
    After discarding a set of measure zero, we may assume that $T(E,\xi)\in \mathcal{N}^b_{\textnormal{infl}(\lambda)}(n)$ for every $\xi\in E$.
    
    Let $\varepsilon>0$.
    Firstly, by Lemma \ref{L:Kirch}, there exists a finite collection of disjoint compact sets $E_i\subset E$ each of which is biLipschitz to some subset of $\reals^n$ and such that $\Hn(E\setminus \bigcup_i E_i)<\varepsilon$. This allows us to use Lemma \ref{L:MSR-open-decomposition} on each $E_i$ separately, thus obtaining for each $i$ a compact set $\widetilde E_i \subset E$ such that still $\Hn(E\setminus \bigcup_{i} \widetilde E_i)<\varepsilon$ and, moreover, the property described in Lemma \ref{L:MSR-open-decomposition} holds on each $\widetilde E_i$. That is, for each $i$ and any $\theta>0$, there exists a finite number of pairwise disjoint open sets $G_i^j\subset \widetilde E_i$ such that
    \begin{enumerate}
        \item\label{Enum:MSR-1} $\widetilde E_i = \bigcup_j G_i^j$,
        \item\label{Enum:MSR-2} for each $j$, there is some $x_i^j\in \widetilde E_i$, $F^j_i\subset\reals^n$, $|\cdot|\in T(\widetilde E_i, x_i^j)$ such that $G_i^j$ is $(1+\theta)$-biLipschitz to $(F_i^j, |\cdot|)$.
    \end{enumerate}
    We let
    \begin{equation*}
        \widetilde E = \bigcup_i \widetilde E_i.
    \end{equation*}
    
    Now let $f\in \tLip_1(X,\reals^m_b)$, $\delta>0$ and $\eta \in (0,1)$ be given. We may assume, without loss of generality, that $L_0 = \tLip(f)<1$. Let $\theta>0$ and $\varepsilon_0>0$ be such that $(1+\theta)^2 L_0 <1$ and
    \begin{equation}\label{E:MSR-choice}
        \frac{\lambda}{1+\theta} \Hn_X(\widetilde{E})-K_b\varepsilon_0\geq \lambda \eta \Hn_X(\widetilde E).
    \end{equation}
    
    Recall that the sets $\widetilde E_i$ are pairwise disjoint and compact.
    Using \ref{Enum:MSR-1}, \ref{Enum:MSR-2} and Lemma \ref{L:prel-ideal}, we may suitably re-index so as to obtain a finite family of non-empty sets $G_j \subset \widetilde E$, $j\in\{1,\dots, j_0\}$ such that
    \begin{enumerate}[label={\rm(\alph*)}]
        \item \label{Enum:open-decom-a2} each $G_j$ is open in $\widetilde E$,
        \item \label{Enum:open-decom-b2} $\widetilde E= \bigcup_{j=1}^{j_0}G_j$,
        \item \label{Enum:open-decom-c2} $\overline{G}_j^{\widetilde E}$, the closure in $\widetilde E$ of the $G_j$, are pairwise disjoint,
        \item \label{Enum:open-decom-d2} for each $j\in\{1,\dots, j_0\}$, there is a norm $|\cdot|_{a_j}$ such that $[|\cdot|_{a_j}]\in \mathcal{N}^b_{\textnormal{infl}(\lambda)}(n)$, a set $F_j\subset F$ and a map $I_j\colon (F_j, |\cdot|_{a_j})\to G_j$ which is a $(1+\theta)$-biLipschitz bijection.
    \end{enumerate}
    
    For an open non-empty set $G\subset \widetilde E$ and $\sigma>0$, we let
    \begin{equation*}
        G^\sigma =\{\xi \in G: d_X(\xi, \widetilde E \setminus G)\geq \sigma\}.
    \end{equation*}
    We shall need two properties of this construction.
    Firstly, by continuity of measure, it holds that
    \begin{equation}\label{E:MSR-sigma-limit}
        \lim_{\sigma \to 0} \Hn(G\setminus G^{\sigma})=0.
    \end{equation}
    Secondly, one has $B_{\widetilde E}(G^\sigma, \sigma)\subset \overline{G}^{\widetilde E}$.
    Here, for convenience, we define $B_{\widetilde E}(\emptyset, \sigma)=\emptyset$.
    Therefore, by \ref{Enum:open-decom-c2}, the sets $B_{\widetilde E}(G_j^\sigma, \sigma)$ are disjoint.
    As there is a finite number of sets $G_j$, using \eqref{E:MSR-sigma-limit} and \ref{Enum:open-decom-b2}, there is some $\sigma>0$ such that $B_{\widetilde E}(G_j^\sigma, \sigma)$ are disjoint and
    \begin{equation}\label{E:MSR-sigma-measure-est}
        \Hn(\widetilde E \setminus \bigcup_{j=1}^{j_0} G_j^\sigma)<\varepsilon_0.
    \end{equation}
     Let $0<\delta_0\leq \delta$ be such that $(1+\theta^2)L_0+4\frac{\delta_0}{\sigma}\leq 1$.
    
    Fix $j\in\{1,\dots, j_0\}$ and let $\widetilde{f}_j \colon F_j \to \reals^m_b$ be given by $\widetilde{f}_j= f \circ I_j$.
    Now
    \begin{equation*}
        \widetilde{f}_j\in \tLip_{(1+\theta)L_0}((F_j, |\cdot|_{a_j}), \reals^m_b),
    \end{equation*}
    whence we may use Theorem \ref{T:positive-result-cubes} 
    (we require only density) to find  $\widetilde{g}_j\in \tLip_{(1+\theta)L_0}((F_j, |\cdot|_{a_j}), \reals^m_b)$ such that $\lVert \widetilde{f}_j - \widetilde{g}_j\rVert_{\ell^{\infty}(F_i)}\leq \delta_0$ and
    \begin{equation*}
        \int_{F_j} J_{F_j} g_j \diff\Hn
        \geq \lambda \vol(|\cdot|_{a_j}) \Hn_{|\cdot|_2}(F_j)= \lambda\Hn_{a_j}(F_j).
    \end{equation*}
    
    Let $g_j\colon G_j \to \reals^m_b$ be given by $g_j=\widetilde{g}_j \circ I_j^{-1}$.
    Then $g_j \in \tLip_{(1+\theta)^2L_0}(G_j, \reals^m_b)$ and, by the area formula, it satisfies
    \begin{equation}\label{E:MSR-G_j-measure-est}
    \begin{split}
        \int_{G_j} J_{G_j} g_j \diff\Hn
        &=\int_{g_j(G_j)} \#g_j^{-1}(u)  \diff\Hn(u)
        =\int_{\widetilde{g}_j(F_j)} \#\widetilde{g}_j^{-1}(u)  \diff\Hn(u)\\
        &=\int_{F_j} J_{F_j} g_j \diff\Hn
        \geq \lambda \Hn_{a_j}(F_j) \geq \frac{\lambda}{1+\theta}\Hn_X(G_j).
    \end{split}
    \end{equation}
    
    Now we may use Lemma \ref{L:Lipschitz-extension-lemma} to find a function $g\colon \widetilde{E}\to \reals^m_b$ such that
    $\lVert g - f \rVert_{\ell^{\infty}(\widetilde{E})}\leq \delta$ and $g=g_j$ on each $G_j^\sigma$. Moreover, we may require $\tLip(g)\leq (1+\theta)^2L_0 + 4\frac{\delta_0}{\sigma}\leq 1$. It remains to show that \eqref{E:MSR-measure-est-main} holds.
    Using disjointness of $G_j$'s we may estimate
    \begin{equation*}
    \begin{split}
        \int_{\widetilde E} J_{\widetilde E} g \diff \Hn
        &\geq \sum_{j=1}^{j_0} \int_{G_j} J_{G_j} g \diff \Hn
        \geq \sum_{j=1}^{j_0} \int_{G^\sigma_j} J_{G_j} g_j \diff \Hn\\
        &=\sum_{j=1}^{j_0} \int_{G_j} J_{G_j} g_j \diff \Hn
        -\sum_{j=1}^{j_0} \int_{G_j\setminus G_j^\sigma} J_{G_j} g_j \diff \Hn\\
        \overset{\eqref{E:MSR-G_j-measure-est}}&{\geq} \sum_{j=1}^{j_0}\frac{\lambda}{1+\theta}\Hn_X(G_j)
        -(\esssup_{E} J_E g) \Hn(\widetilde{E}\setminus \bigcup_{j=1}^{j_0} G_j^{\sigma})\\
        \overset{\eqref{E:MSR-sigma-measure-est}}&{>} \sum_{j=1}^{j_0} \frac{\lambda}{1+\theta} \Hn_X(G_j)
        -K_b \varepsilon_0
        \overset{}{\geq} \frac{\lambda}{1+\theta} \Hn_X(\widetilde{E})-K_b\varepsilon_0\\
        \overset{\eqref{E:MSR-choice}}&{\geq} \eta\lambda\Hn_X(\widetilde{E}).
    \end{split}
    \end{equation*}
\end{proof}


\begin{theorem}\label{T:MSR-residuality}
    Suppose that $n,m\in \nat$, $n\leq m$, $X$ is a complete metric space and $E\subset X$ is an $n$-rectifiable subset. Suppose $|\cdot|_b$ is a norm on $\reals^m$.
    Let $\lambda>0$ and assume that for $\Hn$-a.e.~$\xi\in E$, one has
    \begin{equation*}
        T(E,\xi)\in \mathcal{N}^b_{\textnormal{infl}(\lambda)}(n).
    \end{equation*}
    Then for each $\varepsilon>0$, there is a set $\widetilde{E}\subset E$ with $\Hn(E\setminus \widetilde{E})<\varepsilon$ and such that the set
    \begin{equation*}
        \{f\in\tLip_1(\widetilde E, \reals^m_b): \int_{\widetilde E} J_{\widetilde E} f \diff \Hn \geq \lambda \Hn(\widetilde E)\}
    \end{equation*}
    is residual in $\tLip_1(\widetilde E, \reals^m_b)$.
    Moreover, if $m>n$, then the set
    \begin{equation*}
        \{f\in\tLip_1(\widetilde E, \reals^m_b): \Hn(f(\widetilde E)) \geq \lambda \Hn(\widetilde E)\}
    \end{equation*}
    is residual in $\tLip_1(\widetilde E, \reals^m_b)$.
\end{theorem}
\begin{proof}
Once again, we may reduce to the case $\Hn(E)<\infty$ as $\Hn_{|E}$ is $\sigma$-finite.
    Due to Theorem \ref{T:char-area-area-formula}, it is sufficient to show density of
    \begin{equation*}
    \begin{split}
         \{f\in\tLip_1(\widetilde E, \reals^m_b): \int_{\widetilde E} J_{\widetilde E} f \diff \Hn > \lambda(1-\frac{1}{i}) \Hn(\widetilde E)\},
    \end{split}
    \end{equation*}
    for each $i\in \nat$. However, this is Theorem \ref{T:MSR-density}.
\end{proof}

Recall the definition of strongly $n$-rectifiable sets, Definition \ref{d:strong-rect} and the subsequent characterisation, Lemma \ref{l:strong-rect}.
For these spaces we have the following result in the spirit of Theorem \ref{T:Eucl-strongest-possible}. In relation to this, note in particular that any $1$-rectifiable metric subset of a complete metric space in also strongly $1$-rectifiable since $\mathcal{N}(1)=[|\cdot|_2]$ where $|\cdot|_2$ is the Euclidean norm (absolute value) on $\reals$.

\begin{corollary}\label{C:strongly-Euclidean}
    Suppose $n\in \nat$ and assume that $E$ is strongly $n$-rectifiable metric space. Then, to each $\varepsilon>0$, there is a set $\widetilde E \subset E$ satisfying $\Hn(E\setminus \widetilde{E})<\varepsilon$ such that for every $m\geq n$, the set
    \begin{equation*}
        \{f\in \tLip(\widetilde{E}, \reals^m): \int_{\widetilde{E}} J_{\widetilde E} f \diff \Hn = \Hn(\widetilde E)\}
    \end{equation*}
    is residual in $\tLip(\widetilde{E}, \reals^m)$.
    Moreover, for any $m>n$, the set
    \begin{equation*}
        \{f\in \tLip(\widetilde{E}, \reals^m): \Hn(f(\widetilde E)) = \Hn(\widetilde E)\}
    \end{equation*}
    is residual in $\tLip(\widetilde{E}, \reals^m)$. In particular for any $m\geq n$ a typical $f\in \tLip(\widetilde{E}, \reals^m)$ satisfies $J_{\widetilde E} f=1$ $\Hn$-a.e.~in $\widetilde E$.
\end{corollary}
\begin{proof}
    Recalling Example \ref{Ex:Euclidean-inflation}, we see that $(\reals^n_2, \reals^m_2)$ form a $1$-inflating pair for every $m\geq n$.
    Therefore the statement follows from Theorem \ref{T:MSR-residuality}.
\end{proof}

In the one-dimensional case, if one assumes also the target dimension $m$ to be equal to $1$, it is possible to use McShane's extension to obtain the following.

\begin{theorem}\label{1-rect}
     Suppose $X$ is a complete metric space and $E$ a $1$-rectifiable subset with $\H^1(E)<\infty$. Then
     \begin{equation*}
         \{f\in \tLip_1(X,\reals): \int_E J_E f \diff \H^1 = \H^1(E)\}
     \end{equation*}
     is residual in $\tLip_1(X,\reals)$. In particular, a typical $f\in \tLip_1(X,\reals)$ satisfies $J_E f=1$ $\H^1$-a.e.~in $E$.
\end{theorem}
\begin{proof}
    Every $1$-rectifiable metric subspace of a complete metric space is strongly $1$-rectifiable. Therefore Corollary \ref{C:strongly-Euclidean} together with McShane's extension implies density of the set
    \begin{equation*}
        \{f\in \tLip_1(X,\reals): \int_E J_E f \diff \H^1 = \H^1(E)\}.    
    \end{equation*}
    Whence Theorem \ref{T:char-area-area-formula} yields the result.
\end{proof}

\begin{theorem}
    Suppose $X$ is a complete metric space and $E$ is its strongly $n$-rectifiable subspace. Denote by $E^*$ the set of points of $E$, where the approximate tangent to $E$ exists and is Euclidean.
    Let $k\in\nat$, $k\leq n$ and suppose $K\subset E^*$ is $k$-rectifiable in $X$ (or equivalently in $E$ or $E^*$). Then $K$ is strongly $k$-Euclidean.
\end{theorem}
\begin{proof}
    Suppose $x\in K$ is such that $T(K,x)$ exists. We show that $T(K,x)=[|\cdot|_{\reals^k_2}]$.
    Let $|\cdot|_a \in T(K,x)$.
    To each $\theta>0$, we find $r>0$, Borel sets $\widetilde K \subset K$, $\widetilde E \subset E^*$, $H_r\subset \reals^k$, $F_r \subset \reals^n$ and maps $I_r\colon (H_r, |\cdot|_a) \to \widetilde{K}\cap B(x,r)$, $J_r\colon (F_r, |\cdot|_2) \to \widetilde{E}\cap B(x,r)$ such that
    \begin{enumerate}
        \item $x$ is an $\H^k$-density point of $\widetilde K$,
        \item $x$ is an $\Hn$-density point of $E^*$,
        \item both $I_r$ and $J_r$ are $(1+\theta)$-biLipschitz.
    \end{enumerate}
    Moreover, this may be done in such a way that $\widetilde K\subset \widetilde E$.
    Let now $\theta>0$ be fixed and find the $r>0$ from above. Let $\iota= J_r \circ I^{-1}_r\colon H_r \to F_r$ and observe that $\iota$ is a well defined $(1+\theta)^2$-biLipschitz map. There exists a density point $y$ of $H_r$ such that $\iota(y)$ is a density point of $F_r$ and both $\iota'(y)$ and $(\iota^{-1})'(\iota(y))$ exist. In that case, it is necessary that $\iota'(y)\colon \reals^{k} \to \reals^n$ is a linear map and $\lVert \iota'(y) \rVert_{a\to 2}\leq (1+\theta)^2$. Moreover, $(\iota^{-1})'(\iota(y)) = \frac{1}{\iota'(y)}$ and so
    \begin{equation*}
        \lVert (\iota'(y))^{-1}\rVert_{(\iota(\reals^k)) \to \reals^k_a}\leq (1+\theta)^2.
    \end{equation*}
    All in all, $\iota'(y)$ is a $(1+\theta)^2$-biLipschitz linear map form $\reals^k_a$ onto a $k$-dimensional linear subspace of $\reals^n$. By sending $\theta\to 0$, and observing that all $k$-dimensional subspaces of $\reals^n$ are linearly isometric to $(\reals^k,|\cdot|_2)$, we obtain that $(\reals^k, |\cdot|_a)$ is linearly isometric to $(\reals^k, |\cdot|_2)$ as we wanted.
\end{proof}

The preceding theorem asserts that if our ambient metric space is strongly $n$-rectifiable, then all of its $k$-rectifiable subsets are strongly $k
$-rectifiable, which, in combination with Corollary \ref{C:strongly-Euclidean} yields the following corollary.

\begin{corollary}\label{C:k-Euclidean-subspace}
    Suppose $n\in \nat$ and let $E$ be a strongly $n$-rectifiable subspace of a complete metric space $X$ with. Suppose $k\in \nat$ , $k\leq n$. Then, for any $k$-rectifiable subset $K$ of
    \begin{equation*}
        E^*=\{x\in E: T(E,x)\;\text{exists and is Euclidean}\}
    \end{equation*}
    with $\H^k(K)<\infty$,
     we have the following. To each $\varepsilon>0$, there is a set $\widetilde K \subset K$ satisfying $\Hn(K\setminus \widetilde{K})<\varepsilon$ such that for every $m\geq k$, the set
    \begin{equation*}
        \{f\in \tLip(\widetilde{K}, \reals^m): \int_{\widetilde{K}} J_{\widetilde K} f \diff \H^k = \H^k(\widetilde K)\}
    \end{equation*}
    is residual in $\tLip(\widetilde{K}, \reals^m)$.
    Moreover, for any $m>n$, the set
    \begin{equation*}
        \{f\in \tLip(\widetilde{K}, \reals^m): \H^k(f(\widetilde K)) = \H^k(\widetilde K)\}
    \end{equation*}
    is residual in $\tLip(\widetilde{K}, \reals^m)$. In particular for any $m\geq k$ a typical $f\in \tLip(\widetilde{K}, \reals^m)$ satisfies $J_{\widetilde K} f=1$ $\H^k$-a.e.~in $\widetilde K$.
\end{corollary}

In case $k=n$, it suffices to assume $K\subset E$ instead of $K\subset E^*$ as the exceptional set is $\H^k$-null.

\begin{remark}\label{RCD}
We bring to attention a particular important example of a strongly $n$-rectifiable metric space.
The so-called RCD spaces (see \cite{Amb} for relevant definitions) are metric measure spaces $(X,\mu)$ such that $X$ is strongly $n$-rectifiable for some $n\in\nat$ (this follows by the combination of \cite[Theorem 1.3]{MonNa} and \cite[Theorem 0.1]{BruEli}) and $\mu\ll \Hn$ \cite[Theorem 8.1]{Amb} (see also \cite{KellMon, GiPas}). The fact that $\mu\ll \Hn$ is particularly useful in connection with the moreover part of Corollary \ref{C:k-Euclidean-subspace} as one then obtains
\begin{equation*}
    \int_{\widetilde K} J_{\widetilde K} f(x) F(x) \diff \Hn(x) = \mu(\widetilde K)
\end{equation*}
for a typical $1$-Lipschitz $f$. Here $F$ is the Radon-Nikod\'ym derivative of $\mu$ with respect to $\Hn$.
\end{remark}

\bibliographystyle{abbrv} 

\bibliography{bibliography}
\end{document}